\documentclass[notitlepage,11pt,reqno]{amsart}
\usepackage{amssymb,amsmath,amscd,xspace,tocvsec2,color,mathrsfs,enumerate}
\usepackage[breaklinks=true]{hyperref}%To get clickable links to papers
\usepackage[letterpaper]{geometry}
\theoremstyle{plain}
\newtheorem{theorem}[equation]{Theorem}
\newtheorem{lemma}[equation]{Lemma}
\newtheorem{prop}[equation]{Proposition}
\newtheorem{cor}[equation]{Corollary}

\newtheorem{utheorem}{\textrm{\textbf{Theorem}}}

\theoremstyle{definition}
\newtheorem{definition}[equation]{Definition}
\newtheorem{remark}[equation]{Remark}

\newtheorem{exam}[equation]{Example}
\newtheorem{assum}[equation]{Assumption}

\numberwithin{equation}{section}

\DeclareMathOperator{\ad}{\ensuremath{ad}\xspace}
\DeclareMathOperator{\ch}{\ensuremath{char}\xspace}
\DeclareMathOperator{\tr}{\ensuremath{tr}\xspace}
\DeclareMathOperator{\id}{\ensuremath{id}\xspace}

\DeclareMathOperator{\im}{\ensuremath{im}\xspace}
\DeclareMathOperator{\rt}{\ensuremath{root}\xspace}
\DeclareMathOperator{\wt}{\ensuremath{wt}\xspace}
\DeclareMathOperator{\hgt}{\ensuremath{ht}\xspace}
\DeclareMathOperator{\mis}{\ensuremath{mis}\xspace}
\DeclareMathOperator{\sym}{\ensuremath{Sym}\xspace}
\DeclareMathOperator{\symm}{\ensuremath{sym}\xspace}
\DeclareMathOperator{\hhh}{\ensuremath{Hom}\xspace}

\DeclareMathOperator{\Mod}{\ensuremath{Mod}\xspace}
\DeclareMathOperator{\Ext}{\ensuremath{Ext}\xspace}
\DeclareMathOperator{\End}{\ensuremath{End}\xspace}
\DeclareMathOperator{\Aut}{\ensuremath{Aut}\xspace}
\DeclareMathOperator{\Rad}{\ensuremath{Rad}\xspace}
\DeclareMathOperator{\Spec}{\ensuremath{{Spec}\xspace}}
\newcommand{\lie}[1]{\ensuremath{\mathfrak{#1}}\xspace}
\newcommand{\ghat}{\ensuremath{\widehat{\mathfrak{g}} }\xspace}
\newcommand{\qhat}{\ensuremath{Q^\vee}\xspace}
\newcommand{\phat}{\ensuremath{P^\vee}\xspace}
\newcommand{\U}[1]{\ensuremath{\mathfrak{U}_{q,C}(#1)}\xspace}
\newcommand{\m}{\ensuremath{\mathfrak{m}}\xspace}
\newcommand{\hatt}[1]{\ensuremath{{#1}\ \widehat{}}\xspace}

\newcommand{\vi}{\ensuremath{\varepsilon}\xspace}
\newcommand{\la}{\ensuremath{\lambda}\xspace}

\newcommand{\G}{\ensuremath{\Gamma}\xspace}
\newcommand{\z}{\ensuremath{\widetilde{z}} }

\newcommand{\N}{\ensuremath{\mathbb N}\xspace}
\newcommand{\nn}{\ensuremath{\mathbb{Z}^+}\xspace} %{\geqslant 0}}\xspace}
\newcommand{\Z}{\ensuremath{\mathbb Z}\xspace}
\newcommand{\Q}{\ensuremath{\mathbb Q}\xspace}
\newcommand{\R}{\ensuremath{\mathbb R}\xspace}
\newcommand{\C}{\ensuremath{\mathbb C}\xspace}
\newcommand{\F}{\ensuremath{\mathbb F}\xspace}

\newcommand{\E}{\ensuremath{\mathbb E}\xspace}
\newcommand{\vla}{\ensuremath{\mathbb V}^\lambda\xspace}
\newcommand{\zze}{\ensuremath{\mathbb Z}^k \ltimes_{\boldsymbol \zeta}
\mathbb{E}\xspace}

\newcommand{\calo}{\ensuremath{\mathcal{O}}\xspace}
\newcommand{\cala}{\ensuremath{\mathcal{A}}\xspace}
\newcommand{\alga}{\ensuremath{\mathcal{A}_{\boldsymbol{\zeta}}\xspace}}
\newcommand{\calp}{\ensuremath{\mathcal{P}}\xspace}
\newcommand{\cals}{\ensuremath{\mathcal{W}}\xspace}

\newcommand{\scrc}{\mathscr{C}}

\newcommand{\qbar}{\ensuremath{\overline{Q}}\xspace}
\newcommand{\hfree}{\ensuremath{\widehat{H}^{free}}\xspace}
\newcommand{\hofree}{\ensuremath{\widehat{H_1}^{free}}\xspace}
\newcommand{\hzfree}{\ensuremath{\widehat{H_0}^{free}}\xspace}
\newcommand{\tilf}{\widetilde{f}}
\newcommand{\calq}{\ensuremath{\mathcal{Q}}\xspace}
\newcommand{\hk}{\ensuremath{{\sf H}_\kappa}\xspace}
\newcommand{\infgl}{\ensuremath{\mathcal{H}_\beta(\mathfrak{gl}_n,
\mathbb{F}^n \oplus (\mathbb{F}^n)^*)}\xspace}
\newcommand{\infsp}{\ensuremath{\mathcal{H}_\beta(\mathfrak{sp}_{2n},
\mathbb{F}^{2n})}\xspace}

\newcommand{\mapdef}[1]{\ensuremath{\overset{#1}{\longrightarrow}}\xspace}
\newcommand{\one}[1]{\ensuremath{{#1}_{(1)}}\xspace}
\newcommand{\two}[1]{\ensuremath{{#1}_{(2)}}\xspace}
\newcommand{\three}[1]{\ensuremath{{#1}_{(3)}}\xspace}
\newcommand{\tangle}[1]{\ensuremath{\langle #1 \rangle}\xspace}
\newcommand{\comment}[1]{}
\begin{document}

%{{{1 Begin the paper - maketitle
\title{Axiomatic framework for the BGG Category $\calo$}

\author[Apoorva Khare]{Apoorva Khare\\Stanford University}
\email[A.~Khare]{\tt khare@stanford.edu}
\address{Department of Mathematics, Stanford University, Stanford, CA -
94305}
\date{\today}


\begin{abstract}
In this paper we introduce a general axiomatic framework for algebras
with triangular decomposition, which allows for a systematic study of the
Bernstein-Gelfand-Gelfand Category $\mathcal{O}$. Our axiomatic framework
can be stated via three relatively simple axioms, and it
encompasses a very large class of algebras studied in the literature.
We term the algebras satisfying our axioms as \textit{regular triangular
algebras (RTAs)}; these include
(a)~generalized Weyl algebras,
(b)~symmetrizable Kac-Moody Lie algebras $\lie{g}$,
(c)~quantum groups $U_q(\lie{g})$ over ``lattices with possible torsion",
(d)~infinitesimal Hecke algebras,
(e)~higher rank Virasoro algebras,
and others.

In order to incorporate these special cases under a common setting, our
theory distinguishes between roots and weights, and does not require the
Cartan subalgebra to be a Hopf algebra. We also allow RTAs to have roots
in arbitrary monoids rather than root lattices, and the roots of the
Borel subalgebras to lie in cones with respect to a strict subalgebra of
the Cartan subalgebra. These relaxations of the triangular structure have
not been explored in the literature. 

We then define and study the BGG Category $\mathcal{O}$ over an arbitrary
RTA. In order to work with general RTAs -- and also bypass the use of
central characters -- we introduce certain conditions (termed the
\textit{Conditions (S)}), under which distinguished subcategories of
Category $\mathcal{O}$, termed ``blocks", possess desirable homological
properties including:
(a) being a finite length, abelian, self-dual category;
(b) having enough projectives and injectives; or
(c) being a highest weight category satisfying BGG Reciprocity.
We discuss the above examples and whether they satisfy the various
Conditions (S). We also discuss two new examples of RTAs that cannot be
studied by using previous theories of Category $\mathcal{O}$, but require
the full scope of our framework. These include the first construction of
a family of algebras for which the ``root lattice" is non-abelian.
\end{abstract}
\maketitle
%}}}

%\renewcommand{\thefootnote}{}

\settocdepth{section}
\tableofcontents

%{{{1 Section 1 - Introduction
\section{Introduction}

This paper is motivated by the study of the Bernstein-Gelfand-Gelfand
category $\calo$ \cite{BGG1} associated with a complex semisimple Lie
algebra $\lie{g}$. The definition of $\calo$ depends on the fact that $U
\lie{g}$ has a triangular decomposition.
This category has been studied quite intensively for both classical and
modern reasons, and has connections to geometry, combinatorics, crystals,
categorification, primitive ideals, abelian ideals, Kac-Moody theory,
quantum algebras, and mathematical physics.
To name but a few references, see
\cite{AndStr,H2,Ja2,Jos,Kac2,Maz,MP,Soe} (and the references therein).
One important property of the category $\calo$ is that its blocks are
highest weight categories in the sense of \cite{CPS1}, and hence satisfy
BGG Reciprocity.

Subsequently, Category $\calo$ has been studied over a large number of
algebras with triangular decomposition, and similar results on BGG
Reciprocity and other homological properties of blocks have been shown in
these settings. Thus the main goal of this paper is to simultaneously
generalize both
(a) the structure of the algebra over which to define and study $\calo$,
and
(b) the setup of Category $\calo$ over semisimple $\lie{g}$, in several
different ways. We do so in order to
(a) systematize and unify the treatment of a large number of examples
studied in the literature, and at the same time,
(b) preserve the homological and representation-theoretic properties that
are desirable in the case of semisimple Lie algebras.

Thus, the present paper studies algebras with triangular decomposition $A
\cong B^- \otimes H_1 \otimes B^+$, with the ``middle" subalgebra $H_1$
called the \textit{Cartan subalgebra}. We begin by discussing the ways in
which the structure of the underlying algebras is simultaneously
generalized in the present paper, in order to incorporate a very large
class of examples in the literature:\smallskip

\noindent {\bf 1.}
First, Lie algebras with triangular decomposition as well as their
quantum analogues are combined under a common framework. Recall that
several well-known Lie algebras in representation theory possess a
triangular decomposition similar to $U \lie{g}$ -- for example,
symmetrizable Kac-Moody Lie algebras \cite{Kac2}, contragredient Lie
algebras \cite{KK}, the (centerless) Virasoro algebra \cite{FeFr}, and
extended (centerless) Heisenberg algebras. An analogue of Category
$\calo$ has been explored for such Lie algebras in \cite{MP} (see also
\cite{RCW}).

At the same time, a closely related setting involves quantum analogs of
the aforementioned algebras. These algebras have also been studied in
detail in the literature (see e.g.~\cite{Ja2,Jos}).
Our common framework incorporates both of these settings as special cases
of algebras with triangular decomposition $A \cong B^- \otimes H_1
\otimes B^+$, where the Cartan subalgebra $H_1$ is a commutative,
cocommutative Hopf algebra.

There are similarities between our framework and that of \cite{AS}, in
that Hopf algebras, weight spaces, and quantum groups are involved.
However, our construction is significantly different as well: the
algebras here are neither finite-dimensional, nor do they need to be Hopf
algebras (and \textit{a priori}, we also do not impose restrictions on
the ground field).\smallskip

\noindent {\bf 2.}
While the case of the Cartan subalgebra being a Hopf algebra is
incorporated into our framework, we do not require it to necessarily be
thus. In particular, the framework proposed in this paper also
encompasses algebras arising from topology as well as \textit{low rank
continuous Hecke algebras}, for which the Cartan subalgebras are not Hopf
algebras. See Section \ref{Sgwa2}.\smallskip

\noindent {\bf 3.}
In our framework, there is another strict weakening of the axioms for
$\calo$ used in the literature to date. In all of the examples mentioned
above, if we denote the triangular decomposition as $A = B^- \otimes H_1
\otimes B^+$, then one requires the roots of $B^\pm$ to lie in positive
and negative cones with respect to the entire Cartan subalgebra $H_1$.
However, the present paper only requires this condition to hold with
respect to a (possibly proper) subalgebra $H_0 \subset H_1$. This allows
us to consider certain \textit{higher (Lie) rank infinitesimal Hecke
algebras}, for which Category $\calo$ could not have been studied using
traditional approaches in the literature.\smallskip

\noindent {\bf 4.}
Recall that in the theory of semisimple (or Kac-Moody) Lie algebras, the
root lattice embeds in the weight space. Such a phenomenon also occurs in
their quantum analogues. We provide an explanation by showing that in all
such cases in the literature, there are natural identifications between
the two spaces, which we call the \textit{weight-to-root map} and the
\textit{root-to-weight map}; see Definition \ref{Drtwt} and Proposition
\ref{Prho}.

However, such maps need not exist in general, because the Cartan
subalgebra is not always a Hopf algebra. Thus we will differentiate
between the group generated by the roots, and the space of all weights,
for the Cartan subalgebra $H_1$. This dichotomy between roots and weights
allows us to incorporate generalized Weyl algebras into our
framework.\medskip

\noindent {\bf 5.}
Usually, the group generated by the roots is a ``lattice", generated by a
finite base of simple roots. This is the case for both Lie algebras with
triangular decomposition \cite{MP} as well as quantum groups
$U_q(\lie{g})$ over Kac-Moody Lie algebras $\lie{g}$. Our framework
weakens this restriction by removing the lattice assumption. In fact, we
remove the commutativity assumption altogether, and work with arbitrary
``torsion-free" monoids.
This allows us to incorporate many algebras from mathematical physics
such as \textit{generalized Virasoro and Schr\"odinger-Virasoro algebras}
(as well as more traditional examples such as Witten's family of algebras
and conformal $\lie{sl}_2$-algebras).

Furthermore, in Section \ref{Snonabelian} we prove an Existence Theorem
that allows us to construct algebras with triangular decomposition, in
which the span of positive roots can be any monoid that satisfies certain
``cocycle conditions" \eqref{Ecocycle1},\eqref{Ecocycle2}.
These conditions are novel and incorporate \textit{all} abelian,
torsion-free monoids as well as some non-abelian ones. This enables us to
construct algebras with \textit{non-abelian} ``root lattices"; to our
knowledge, no such algebras have been studied in the literature. Our
results show that the proposed axiomatic framework is at once not
unnecessarily ``too broad", as well as broad enough to incorporate a very
large class of settings in the literature -- traditional as well as
modern, classical as well as quantum.\medskip

\noindent {\bf 6.}
Recall \cite{BGG1,H2} that in studying Category $\calo$ over a semisimple
Lie algebra $\lie{g}$, and its decomposition into finite length blocks
with enough projectives, central characters have played a crucial role,
via a finiteness condition that we call (S4) in this paper (see
Definition \ref{DcondS}). The condition is useful in proving results in
representation theory because the center of $U \lie{g}$ is ``large
enough".

In general, however, this is false: there are algebras with triangular
decomposition, whose center is trivial. In fact one of our motivating
examples was the infinitesimal Hecke algebra of $\lie{sl}_2$ (and $\C^2$)
studied with Tikaradze in \cite{Kh,KT} -- as well as its quantized
analogue, which was studied in joint work \cite{GGK} with Gan. It was
shown that the latter, quantum version has trivial center; yet a theory
of $\calo$ and its block decomposition (with BGG Reciprocity) was
developed in \cite{GGK}.

Thus, we do not use the center in this paper. Instead, we propose a
strictly weaker condition which we call (S3), and which holds for
semisimple $\lie{g}$ because of condition (S4) involving central
characters. We show that Condition (S3) already implies a block
decomposition into highest weight categories. Thus our framework allows
us to incorporate relatively modern constructions such as \textit{rank
one (quantum) infinitesimal Hecke algebras}, even though they may have
trivial center.\smallskip

Additionally, we now describe two ways in which we extend in this paper,
the treatment of Category $\calo$ found in the literature.\medskip

\noindent {\bf 7.}
In studying representations of Lie algebras $\lie{g}$ with triangular
decomposition, one often focuses on representations on which the center
$Z(\lie{g}) \subset \lie{g}$ acts by a fixed linear functional, or
\textit{level}. This is indeed the case for Kac-Moody Lie algebras and
for other algebras such as higher rank Virasoro algebras; see
e.g.~\cite{FeFr,HWZ}. Similarly, in the present paper we define
distinguished subcategories of $\calo$ that satisfy conditions such as
(S3) (discussed above and defined in Section \ref{Sos}). In other words,
we identify ``good parts" of $\calo$ that possess desirable homological
properties.\medskip

\noindent {\bf 8.}
Finally, the framework we propose is ``functorial", in that the structure
of Category $\calo$ (or its ``good parts" as in the previous point) over
a tensor product of commuting factors can be deduced from similar
structural facts for $\calo$ over each individual factor. For example,
the connection between modules over a semisimple Lie algebra and those
over its simple ideals, is a specific manifestation of a broader
phenomenon that holds in the general setting studied in this
paper.\medskip

Given the phenomena discussed above, we develop in this paper a general
framework of a \textit{regular triangular algebra (RTA)} for which the
notion of Category $\calo$ makes sense, and which encompasses all of the
aforementioned examples. We conclude this paper with Example \ref{Efull},
which describes an RTA $A$ for which one has to use the full level of
generality of our framework to study Category $\calo$ (and one can show
$\calo$ has very desirable properties), but the previously developed
treatments of $\calo$ are not adequate to describe its representation
theory. See Theorem \ref{Tfullrta}.

\subsection*{Organization of the paper}

This paper is organized as follows. In Section \ref{Sdef}, we present our
axiomatic framework, which encompasses a wide variety of algebras.
Section \ref{SHopf} discusses the special case when the Cartan subalgebra
$H_1$ is a Hopf algebra, and ends by characterizing such algebras inside
our framework. Next, we introduce Verma modules and other key concepts in
Section \ref{Sverma}. We then state in Section \ref{Sos} -- and show in
Section \ref{Sproofs} -- the main theoretical results about Category
$\calo$, including block decompositions and homological properties.

The second half is devoted to examples. In Section \ref{Snonabelian} we
provide the first example of a regular triangular algebra for which the
analogue of the root lattice is not abelian. We also prove an Existence
Theorem for all (abelian) variants of the root lattice.
Sections \ref{Sexam} and \ref{Suqg} discuss familiar examples, including
Lie algebras with ``regular triangular decomposition", and a family of
``extended quantum groups" for every symmetrizable Kac-Moody Lie algebra.
Section \ref{Smore} discusses further examples of our broad framework,
including one in which the center is trivial and yet $\calo$ has a block
decomposition.
Section \ref{Sgwa} studies generalized Weyl algebras in detail -- as an
additional result, we prove that generalized down-up algebras admit
quantizations, which are themselves deformations of quantum $\lie{sl}_2$.
In Section \ref{Sgwa2} we provide two examples of such algebras where the
Cartan subalgebra is not a Hopf algebra. Finally in Section
\ref{Sinfhecke}, we study infinitesimal Hecke algebras of higher (Lie)
rank, for which the root lattice and weight space are not contained in
the same vector space. We end with Example \ref{Efull}, which uses the
full generality of our framework to study Category $\calo$.
%}}}

\section{The main definition: Regular Triangular Algebras}\label{Sdef}

%{{{1 Section 2.0
We work throughout over a ground field $\F$. Unless otherwise specified,
$\ch \F$ is arbitrary, and all tensor products below are over $\F$.
We will often abuse notation and claim that two modules or functors are
equal, when they are isomorphic (e.g.~double duals).
Now define $\Z^\pm := \pm (\N \cup \{0\})$. Given $S \subset \Z$ and a
subset $\Delta$ of an abelian group $\Theta_0$, define $S \Delta$ to be
the set of all finite $S$-linear combinations $\sum_{\alpha \in \Delta}
n_\alpha \alpha$, where $n_\alpha \in S\ \forall \alpha$.
Finally, given any group $\Theta$ and a subset $\calq^+ \subset \Theta$,
define $-\calq^+ := \{ \theta^{-1} : \theta \in \calq^+ \} \subset
\Theta$, $\tangle{\calq^+}$ to be the subgroup of $\Theta$ generated by
$\calq^+$, and $\F \Theta$ to be the group algebra of $\Theta$.

\begin{definition}\label{Dweight}
Fix a ground field $\F$, and $\F$-algebras $H \subset A$.
\begin{enumerate}
\item Define the spaces of \emph{roots} and \emph{weights} of $H$ to be
$\Aut_{\F-alg}(H)$ and $\widehat{H} := \hhh_{\F-alg}(H,\F)$ respectively.

\item Given a weight $\la \in \widehat{H}$ and an $H$-module $M$, the
{\em $\la$-weight space} of $M$ is $M_\la := \{ m \in M : hm = \la(h)m\
\forall h \in H \}$. The \emph{set of $H$-weights} of $M$ is $\wt_H(M) :=
\{ \lambda \in \widehat{H} : M_\lambda \neq 0 \}$, and $M$ is an
\emph{$H$-weight module} if $M = \bigoplus_{\lambda \in \widehat{H}}
M_\lambda$ is $H$-semisimple.

\item Define the \emph{$\theta$-root space} of $A$ corresponding to a
root $\theta \in \Aut_{\F-alg}(H)$, as well as the \emph{set of
$H$-roots of $A$}, to respectively equal
\begin{equation}
A_\theta := \{ a \in A : a h = \theta(h) a\ \forall h \in H \}, \qquad
\rt_H(A) := \{ \theta \in \Aut_{\F-alg}(H) : A_\theta \neq 0 \}.
\end{equation}

\item If $H_0 \subset H$ is an $\F$-subalgebra, let $\pi'_{H_0} :
\End_{\F-alg}(H) \to \hhh_{\F-alg}(H_0,H)$ denote the restriction map.
Similarly, denote by $\pi_{H_0} : \widehat{H_1} \to \widehat{H_0}$ the
restriction map to $H_0$.
\end{enumerate}
\end{definition}

\noindent The axiomatic framework introduced in this paper will display
the aforementioned dichotomy between \textit{roots}, which pertain to
algebras and belong to $\Aut_{\F-alg}(H_0)$; and \textit{weights}, which
pertain to representations and live in $\widehat{H_1}$.
In this paper, we use $\theta$ to refer to roots.
%We also remark that $\Theta_H(\Delta)$ will not be considered further in
%the paper.

Equipped with the above terminology, it is now possible to propose a
broad framework in which to study the BGG Category $\calo$, and which
incorporates many examples in the literature.

\begin{definition}\label{Drta}
An associative $\F$-algebra $A$, together with data $(B^\pm, H_1, H_0,
\calq^+_0, i)$ satisfying the following conditions, is called a {\em
regular triangular algebra} (denoted also by {\em RTA}).
\begin{enumerate}
\item[(RTA1)]
There exist associative unital $\F$-subalgebras $B^\pm, H_1$ of $A$, such
that the multiplication map $: B^- \otimes_\F H_1 \otimes_\F B^+ \to A$
is a vector space isomorphism (the {\em triangular decomposition}).

\item[(RTA2)]
There exist a unital subalgebra $H_0 \subset H_1$ and a monoid $\calq^+_0
\subset \Aut_{\F-alg}(H_0)$, such that $\calq^+_0 \setminus \{ \id_{H_0}
\}$ is a semigroup, and moreover,
\begin{equation}
B^+ = \bigoplus_{\theta_1 \in \calq^+_1} B^+_{\theta_1}, \qquad
\text{where } \calq^+_1 := (\pi'_{H_0})^{-1}(\calq^+_0) \cap
\Aut_{\F-alg}(H_1).
\end{equation}

\noindent Moreover, $B^+_{\id_{H_0}} = \F \cdot 1$, and $\dim_\F
B^+_{\theta_0} < \infty$ for all $\theta_0 \in \calq^+_0$
(the {\em regularity} assumption).

\item[(RTA3)] There exists an anti-involution $i$ of $A$ (i.e., $i^2|_A =
\id|_A$) that fixes $H_1$, and sends $B^\pm$ into the image under the
multiplication map of $H_1 \otimes B^\mp$.
\end{enumerate}
\end{definition}

\noindent As explained in  Proposition \ref{Pverma}(1) below, the
assumption that $\calq^+_0 \setminus \{ \id_{H_0} \}$ is a semigroup
helps construct a partial order on the set of weights. It also implies
that $\calq^+_0$ is either trivial or infinite.
Moreover, we do not insist that the anti-involution $i : A \to A$ sends
$B^+$ to $B^-$, as is the case for Lie algebras with triangular
decompositions. The reason is that for quantum algebras $i$ may not send
$B^+$ to $B^-$; see Section \ref{Suqg} or Example \ref{Einfhecke1}.

As we will discuss through many examples, most of the traditionally
well-studied RTAs in the literature satisfy two additional restrictions:
(a) $H_0 = H_1$; and (b) $\calq^+_0$ is generated by a finite $\Z$-basis
$\Delta$ of ``simple roots". These restrictions are encoded as follows
for a general RTA.

\begin{definition}\label{Dstrict}
An RTA $A$ (together with $(B^\pm, H_1, H_0, \calq^+_0, i)$) is {\em
strict} if $H_1 = H_0$.
An RTA is {\em based} if there exists a pairwise commuting
$\Z$-linearly independent set $\Delta \subset \calq^+_0$, called the {\em
(base of) simple roots}, such that $\calq^+_0 = \nn \Delta$. In this case
we may also denote the RTA by $(B^\pm, H_1, H_0, \Delta, i)$. The {\em
rank} of a {\em strict}, based RTA is defined to be $|\Delta|$ for the
smallest such $\Delta$ (or $\calq^+_0$).
\end{definition}

\begin{remark}
In Section \ref{Sinfhecke} we will see examples of non-strict RTAs
(called infinitesimal Hecke algebras) which involve simple Lie algebras
of arbitrary Lie rank, but for which it is possible to choose precisely
one simple root to generate $\calq^+_0$. In order to avoid this
discrepancy, we do not talk about the rank of a non-strict, based RTA in
this paper.
\end{remark}

\begin{exam}\label{Emotivating}
Definition \ref{Drta} is quite technical; here is our motivating example
- a finite-dimensional complex semisimple Lie algebra $\lie{g}$ with
triangular decomposition $\lie{g} = \lie{n}^+ \oplus \lie{h} \oplus
\lie{n}^-$. Then $A = U \lie{g} = U \lie{n}^- \otimes \sym(\lie{h})
\otimes U \lie{n}^+$ is a strict, based RTA with:
\begin{itemize}
\item $H_1 = H_0 = U \lie{h}$ the Cartan subalgebra -- this
is a commutative, cocommutative Hopf algebra, so $A = U \lie{g}$ is in
fact a strict \textit{Hopf RTA} (see Section \ref{SHopf});

\item $B^\pm = U \lie{n}^\pm$; and

\item $i$ the anti-involution obtained by composing the Chevalley
involution and the Hopf algebra antipode -- so $i$ sends $\lie{g}_\alpha$
to $\lie{g}_{-\alpha}$ for all roots $\alpha$ (and hence $B^\pm$ to
$B^\mp$), and fixes $\lie{h}$.
\end{itemize}

\noindent Now identify $\Aut_{\F-alg}(H_0)$ with $\widehat{H_0} =
\lie{h}^*$ as follows (also see Proposition \ref{Prho}): for every weight
$\mu \in \lie{h}^*$, define the root $\rho_{H_0}(\mu) : H_0 \to H_0$ via:
$\rho_{H_0}(\mu)(h_1 \cdots h_n) := \prod_{j=1}^n (h_j - \mu(h_j))$ (and
extend by linearity). It follows easily that $\mu \mapsto
\rho_{H_0}(\mu)$ is an isomorphism of additive groups $\rho_{H_0} :
\lie{h}^* \to \Aut_{\F-alg}(H_0)$. Define $\calq^+_0$ to be the monoid
generated by the simple roots $\Delta$ (or more precisely, $\{
\rho_{H_0}(\alpha) : \alpha \in \Delta \}$) -- this is usually denoted in
the literature as $Q^+$, the ``positive" part of the root lattice. Now
note that the above parameters equip $U \lie{g}$ with the structure of a
strict, based RTA of finite rank.
\end{exam}

\begin{remark}\label{R11}\hfill
\begin{enumerate}
\item Henceforth we denote an RTA by $A$ alone, and do not explicitly
write out all of the additional data $(B^\pm, H_1, H_0, \calq^+_0, i)$,
even though it will also be assumed to be fixed.

\item If $A$ is a based RTA and $\Aut_{\F-alg}(H_0)$ is a subgroup of an
$\F$-vector space under addition, then we also require $\F$ to have
characteristic zero, since otherwise $\calq^+_0 \setminus \{ \id_{H_0}
\}$ has torsion and hence cannot be a semigroup. This explains why we
will assume $\ch \F = 0$ for Lie algebras (as in Example
\ref{Emotivating}), but not necessarily for quantum groups.
\end{enumerate}
\end{remark}

We now list some basic properties of regular triangular algebras (RTAs).
These properties will be used henceforth without further reference.

\begin{lemma}\label{Lfirst}
($A$ is an RTA.) Suppose $\calq^+_r$ generates the subgroup
$\tangle{\calq^+_r} \subset \Aut_{\F-alg}(H_r)$ for $r=0,1$.
\begin{enumerate}
\item The groups $\tangle{\calq^+_r}$ act on the sets $\widehat{H_r}$ for
$r=0,1$ via: $\theta_r * \lambda_r := \lambda_r \circ \theta_r^{-1}$ for
$\lambda_r \in \widehat{H_r}, \theta_r \in \tangle{\calq^+_r}$.
The actions are functorial, in that the following square commutes for all
pairs of $\F$-algebras $H_0 \hookrightarrow H_1$:
\[ \begin{CD}
(\pi'_{H_0})^{-1}(\Aut_{\F-alg}(H_0)) \times \widehat{H_1} @>*>>
\widehat{H_1}\\
@V\pi'_{H_0} \times \pi_{H_0} VV @V\pi_{H_0} VV\\
\Aut_{\F-alg}(H_0) \times \widehat{H_0} @>*>> \widehat{H_0}
\end{CD}
\]

\item If $M$ is any $A$-module, then for $r=0,1$, we have:
\begin{equation}
A_{\theta_r} \cdot M_{\lambda_r} \subset M_{\theta_r * \lambda_r} =
M_{\lambda_r \circ \theta_r^{-1}}, \qquad \forall \theta_r \in
\tangle{\calq^+_r},\ \lambda_r \in \widehat{H_r}.
\end{equation}

\item $H_r$ is commutative for $r=0,1$, whence $H_r = (H_r)_{\id_{H_r}} =
(H_r)_{\id_{H_{3-r}} }$.

\item $i(A_{\theta_r}) = A_{\theta_r^{-1}}$ for $r = 0,1$ and $\theta_r
\in \tangle{\calq^+_r}$.
\end{enumerate}
\end{lemma}

\noindent The proofs are straightforward. For instance, part (4) holds
because $i(a_{\theta_r}) h_r = i(h_r a_{\theta_r}) = i(a_{\theta_r}
\theta_r^{-1}(h_r)) = \theta_r^{-1}(h_r) i(a_{\theta_r})$ for all $r=0,1,
\ \theta_r \in \tangle{\calq^+_r},\ a_{\theta_r} \in A_{\theta_r}, h_r
\in H_r$.

In turn, Lemma \ref{Lfirst} helps prove that the subalgebras $B^\pm$ are
``symmetric" in a precise sense:

\begin{prop}\label{Pfirst}
$A$ is an RTA as above.
\begin{enumerate}
\item $B^-$ has a decomposition similar to that of $B^+$ in (RTA2), i.e.,
there exists a monoid $\calq^-_0 \subset \Aut_{\F-alg}(H_0)$, such that
\[ B^- = \bigoplus_{\theta_1 \in \calq^-_1} B^-_{\theta_1}, \qquad
\text{where } \calq^-_1 := \{ \theta_1 \in \Aut_{\F-alg}(H_1) :
\pi'_{H_0}(\theta_1) \in \calq^-_0 \}. \]

\noindent Moreover, $\calq^-_r = -\calq^+_r$ for $r=0,1$,
$B^-_{\id_{H_0}} = \F$, and $\dim_\F B^-_{\theta_r^{-1}} = \dim_\F
B^+_{\theta_r} < \infty\ \forall \theta_r \in \calq^+_r$.

\item $H_r \otimes B^\pm$ (more precisely, their images under
multiplication) are unital $\F$-subalgebras of $A$.

\item For $r=0,1$, $\calq^\pm_r$ are sub-monoids of $\Aut_{\F-alg}(H_r)$,
such that $\calq^\pm_r \setminus \{ \id_{H_r} \}$ are semigroups.
Moreover, $\pi'_{H_0} : \tangle{\calq^+_1} \to \tangle{\calq^+_0}$ is a
group homomorphism that restricts to the monoid maps $: \calq^\pm_1
\twoheadrightarrow \calq^\pm_0$, and $A = \bigoplus_{\theta_r \in
\tangle{\calq^+_r}} A_{\theta_r}$ is $\tangle{\calq^+_r}$-graded for
$r=0,1$.

\item The algebras $B^\pm$ have subalgebras (in fact, augmentation
ideals) defined respectively as
\[ N^\pm := \bigoplus_{\theta_r \in \pm \calq^+_r \setminus \{ \id_{H_r}
\}} B^\pm_{\theta_r}, \qquad r=0,1. \]
\end{enumerate}
\end{prop}

\begin{proof}
First observe that any sum of $H_r$-root subspaces of $B^+$ or $B^-$ is
direct. The statement is part of (RTA2) for $B^+$, and hence follows for
$B^-$ using Lemma \ref{Lfirst}(4). We now proceed with the proof. The
meat of the result lies in proving part (1). Compute using the
$H_1$-root-semisimplicity of $B^+$ and the multiplication map $m_A$ on
$A$:
\begin{align*}
B^- = i(i(B^-)) \subset i( m_A(H_1 \otimes B^+)) = &\ m_A(i(B^+) \otimes
H_1) = m_A(i ( \bigoplus_{\theta_1 \in \calq^+_1} B^+_{\theta_1}) \otimes
H_1)\\
= &\ \bigoplus_{\theta_1 \in \calq^+_1} m_A(H_1 \otimes
B^-_{\theta_1^{-1}} \otimes H_1) = \bigoplus_{\theta_1 \in \calq^+_1}
m_A(H_1 \otimes B^-_{\theta_1^{-1}}).
\end{align*}

\noindent where all decompositions are direct from above, and the last
equality follows by definition of root spaces. It follows by (RTA1) that
$B^-$ decomposes as a direct sum of $H_1$-root spaces with roots in
$\calq^-_1 := -\calq^+_1$. Restricting $\calq^-_1$ to $H_0$ proves the
same assertion for $\calq^-_0 := -\calq^+_0$. Moreover, $i(1_{B^+}) =
1_{H_1 \otimes B^-} = 1_{B^-}$. Thus the remaining assertions in (1)
follow if we show that $\dim_\F B^-_{\theta_r^{-1}} = \dim_\F
B^+_{\theta_r}$ for $\theta_r \in \calq^+_r$.

Before doing so, we first prove (2) using only the aforementioned
$\calq^\pm_r$-root-space decomposition of $B^\pm$. Indeed, observe for
$r=0,1$ that $b_{\theta_r} h_r = \theta_r(h_r) b_{\theta_r} \in m_A(H_r
\otimes B^\pm)$ whenever $h_r \in H_r, \theta_r \in \calq^\pm_r \subset
\tangle{\calq^+_r}, b_{\theta_r} \in B^\pm_{\theta_r}$. Thus (2) follows
from (RTA1).

We now complete the proof of (1), by showing that $\dim_\F
B^-_{\theta_r^{-1}} = \dim_\F B^+_{\theta_r}$ for $\theta_r \in
\calq^+_r$. First suppose $r=0$, and fix an $H_0$-root-basis $b_1, \dots,
b_n$ of $B^+_{\theta_0}$ for fixed $\theta_0 \in \calq^+_0$. Also fix any
finite-dimensional subspace $V \subset B^-_{\theta_0^{-1}}$ such that
$i(b_j) \in m_A(H_1 \otimes V)$ for all $j$. Then using (2),
\[ B^-_{\theta_0^{-1}} = i(i(B^-_{\theta_0^{-1}})) \subset i(m_A(H_1
\otimes B^+_{\theta_0})) \subset m_A(i(B^+_{\theta_0}) \otimes H_1)
\subset m_A(H_1 \otimes V \otimes H_1) = m_A(H_1 \otimes V), \]

\noindent which shows (by (RTA1)) that $B^-_{\theta_0^{-1}} = V$ must be
finite-dimensional for all $\theta_0 \in \calq^+_0$. Now fix $r$ and
$\theta_r \in \calq^+_r$, as well as bases $b_1, \dots, b_n$ of
$B^+_{\theta_r}$ and $v_1, \dots, v_m$ of $B^-_{\theta_r^{-1}}$.
Suppose $\displaystyle i(v_j) = \sum_k m_A(h_{jk} \otimes b_k)$ and
$\displaystyle i(b_k) = \sum_k m_A(v_j \otimes t_{kj})$ for some choices
of elements $h_{jk}, t_{kj} \in H_1$. Then the (possibly rectangular)
matrices $H := (h_{jk}), T := (t_{kj})$ satisfy: $HT, TH$ are identity
matrices. Equating their traces yields $m=n$, i.e., $\dim_\F
B^-_{\theta_r^{-1}} = \dim_\F B^+_{\theta_r}$ as claimed.

The remaining parts are easily shown: (3) is straightforward given (1)
and Lemma \ref{Lfirst}, and (4) follows from (3).
\end{proof}
%}}}

%{{{1 Section 2.1 - Hopf regular triangular algebras
\subsection{Hopf regular triangular algebras}\label{SHopf}

We now analyze regular triangular algebras in the special case when $H_1$
is a Hopf algebra, $H_0$ a Hopf subalgebra, and the Hopf structure is
used to define an adjoint action with respect to which $A$ is semisimple.
This is in itself a very general setup that encompasses many well-studied
examples in the literature, including Kac-Moody Lie algebras and their
quantum groups. To proceed further, it is convenient to fix some
notation.\medskip

\noindent {\bf Notation.}
Let $H$ be a Hopf algebra (not necessarily commutative) over a field
$\F$, and denote by $m_H$ (or $\Delta_H, \eta_H, \vi_H, S_H$) the
multiplication in $H$ (or comultiplication, unit, counit, antipode
respectively) -- see e.g.~\cite{Ka}. We will use Sweedler notation:
$\Delta_H(h) = \sum \one{h} \otimes \two{h}$ for $h \in H$. Now note that
$\widehat{H} \subset H^*$ is precisely the set of grouplike elements in
$H^*$. Also define {\em convolution} on $\widehat{H}$, via $\displaystyle
\tangle{\mu * \la, h} := \tangle{\mu \otimes \la, \Delta_H(h)} = \sum
\tangle{\mu,\one{h}} \tangle{\la, \two{h}}$. Then (\cite[Exercise
III.8.11]{Ka}) $(\widehat{H},*)$ is a group, with unit $\vi_H$, and
inverse given by $\la \mapsto \la \circ S_H$ in $\widehat{H}$.\medskip

We now introduce a Hopf-theoretic framework that encompasses many
well-known algebras in the literature, as we illustrate through examples
later in the paper.

\begin{definition}\label{Dhrta}
Suppose $H$ is a Hopf algebra and $A$ is an $\F$-algebra containing $H$.
Define the {\em adjoint action} $\ad : H \to \End_\F(A)$ via: $(\ad h)(a)
:= \sum \one{h} a S(\two{h})$ for all $h \in H$ and $a \in A$.

Next, a {\em Hopf regular triangular algebra} (denoted also by {\em Hopf
RTA}, or {\em HRTA} in short), is an $\F$-algebra $A$, together with the
data $(B^\pm, H_1, H_0, \calq^{'+}_0, i)$ that satisfies (RTA1), (RTA3),
and the following condition:
\begin{enumerate}
%\item[(HRTA1)]
%There exist associative unital $\F$-subalgebras $B^\pm, H_1$ of $A$,
%such that the multiplication map $: B^- \otimes_\F H_1 \otimes_\F B^+
%\to A$ is a vector space isomorphism.
%
\item[(HRTA2)]
$H_1$ is a Hopf algebra that contains a sub-Hopf algebra $H_0$.
Moreover, there exists a monoid $\calq^{'+}_0 \subset \widehat{H_0}$ such
that $\calq^{'+}_0 \setminus \{ \vi_{H_0} \}$ is a semigroup, which
satisfies:
\[ B^+ = \bigoplus_{\mu_1 \in \calq^{'+}_1} B^+_{\mu_1}, \qquad
\calq^{'+}_1 := \pi_{H_0}^{-1}(\calq^{'+}_0) \subset \widehat{H_1}, \]

\noindent where $\pi_{H_0} : \widehat{H_1} \to \widehat{H_0}$ is the
restriction map, and $B^+_{\mu_1}$ is the $\mu_1$-weight space for the
adjoint action of $H_1$ on $A$. Furthermore, $B^+_{\vi_{H_0}} = \F$, and
$\dim_\F B^+_{\mu_0} < \infty$ for all $\mu_0 \in \calq^{'+}_0$.
%
%\item[(HRTA3)]
%There exists an anti-involution $i$ of $A$ that fixes $H_1$.
\end{enumerate}
\end{definition}

Note that the definition of an HRTA is in some sense parallel to that of
an RTA. However, the conditions (RTA2) and (HRTA2) are significantly
different, in that the monoid $\calq^{'+}_0$ is contained in ``weight
space" $\widehat{H_0}$ and involves the adjoint action of $H_0$ on $A$,
instead of being contained in ``root space" $\Aut_{\F-alg}(H_0)$ as in
the RTA case. 
%Now if $H_1 = H_0$, then the group $\tangle{\calq^{'+}_0}$ is
%automatically contained in the weight space $\widehat{H_1}$, as is the
%case for Lie algebras and quantum groups (but not for non-strict RTAs).
%
In fact, Definition \ref{Dhrta} was primarily designed to incorporate Lie
algebras as well as their quantum analogues into a common framework, and
the properties of HRTAs were extensively studied in previous work
\cite{Kh2}.\footnote{The notion of an RTA was also defined, albeit
``incorrectly", in \cite{Kh2}. The reason it is not ``correct" is that it
is overly restrictive, requiring six technical axioms (besides the
triangular decomposition and anti-involution) and yet not able to
incorporate several of the settings considered in the present paper -
including non-based settings as in Sections \ref{Snonabelian} and
\ref{Snonbased}, as well as generalized Weyl algebras as in Sections
\ref{Sgwa}, \ref{Sgwa2}. However, the notion of a (based) Hopf RTA in
\cite{Kh2} essentially agrees with Definition \ref{Dhrta} in the present
paper.} 
Thus the definition of an HRTA is \textit{a priori} similar, but not
related to the notion of an RTA. However, it turns out that the two are
indeed closely related. To explain their precise connection, additional
notation is required.

\begin{definition}\label{Drtwt}
We say that a HRTA is {\em strict} if $H_1 = H_0$.
A HRTA is {\em based} if there exists a $\Z$-linearly independent set of
weights $\Delta' \subset \calq^{'+}_0$, such that $\calq^{'+}_0 = \nn
\Delta'$. In this case we may also denote the HRTA by $(B^\pm, H_1, H_0,
\Delta', i)$.
Finally, given a Hopf algebra $H$, define two maps:
\begin{itemize}
\item The \textit{weight-to-root map} $\rho_H : \widehat{H} \to
\End_\F(H)$ is defined via: $\rho_H(\mu)(h) := \sum \mu^{-1}(\one{h})
\two{h} = \sum \mu(S(\one{h})) \two{h}$.
\item The \textit{root-to-weight map} $\Psi_\vi : \Aut_{\F-alg}(H) \to
H^*$ is defined via: $\Psi_\vi(\theta) := \vi \circ \theta^{-1}$.
\end{itemize}
\end{definition}

\noindent It is now possible to relate HRTAs to RTAs (and to justify why
we call such algebras Hopf \textit{RTA}s).

\begin{theorem}\label{Thrtarta}
Suppose $A$ is an RTA over a ground field $\F$. Then $A$ is an HRTA if
and only if $H_1 \supset H_0$ are Hopf algebras and there exists a choice
of parameters such that $\calq^+_r \subset \im (\rho_{H_r})$ for $r=0,1$.
In this case, $A$ is a (strict) (based) Hopf RTA if and only if $A$ is a
(strict) (based) RTA.
\end{theorem}

The proof of Theorem \ref{Thrtarta} uses the following preliminary
results.

\begin{prop}\label{Prho}
Suppose $H$ is a Hopf algebra and $A$ is an $\F$-algebra containing $H$.
\begin{enumerate}
\item The root-to-weight map is a surjective group homomorphism $\Psi_\vi
: \Aut_{\F-alg}(H) \to \widehat{H}$. It has right inverse equal to the
weight-to-root map, which is an injective group homomorphism $\rho_H :
\widehat{H} \to \Aut_{\F-alg}(H)$.

\item The assignments $H_0 \mapsto \widehat{H_0}$ and $H_0 \to
\hhh_{\F-alg}(H_0,H)$ are contravariant functors from the category of
sub-Hopf algebras $H_0$ of $H$ and injective Hopf maps, to the categories
of groups and sets respectively. Moreover, the family of weight-to-root
maps $\{ \rho_{H_0} : H_0 \subset H \}$ constitute a natural
transformation $: \widehat{\ } \to \hhh_{\F-alg}(-,H)$.

\item $\im(\rho_H)$ acts freely on $\widehat{H}$ via: $\rho_H(\mu)(\nu) =
\mu * \nu$.

\item $\ad : H \to \End_\F(A)$ is an $\F$-algebra homomorphism.

\item For all $\mu \in \widehat{H}$, the weight space $A_\mu$ (for the
adjoint action of $H$ on $A$) and the root space $A_{\rho_H(\mu)}$ (see
Definition \ref{Dweight}) coincide. In particular, $A_{\vi_H} = A_{\id_H}
= Z_A(H)$.
\end{enumerate}
\end{prop}

\noindent Part (1) says in particular that every Hopf algebra is a module
over its weights. To our knowledge (and that of some experts) it seems,
somewhat surprisingly, to be a new formulation (at least).
%Parts (2) and (3) follow easily by using the definitions. Parts (4),(5)
%are nice exercises in Sweedler notation;
 We also remark that the last assertion in (5) can be found in
 e.g.~\cite[Lemma 1.3.3]{Jos}.

\begin{proof}
Most of the proofs are straightforward; however, we include them for
completeness.
%%%{\color{red}
\begin{enumerate}
\item We begin by studying the properties of the map $\rho_H$. The first
claim is that $\rho_H$ is a group homomorphism. Indeed, given $\mu, \nu
\in \widehat{H}$, one has:
\begin{align*}
\rho_H(\mu) \circ \rho_H(\nu)(h) = &\ \rho_H(\mu) \left( \sum
\nu^{-1}(\one{h}) \two{h} \right) = \sum \nu^{-1}(\one{h})
\mu^{-1}(\two{h}) \three{h}\\
= &\ \sum (\nu^{-1} * \mu^{-1})(\one{h}) \two{h} = \rho_H(\mu *
\nu)(h),\\
\rho_H(\vi_H)(h) = &\ \sum \vi_H(\one{h}) \two{h} = h.
\end{align*}

Next, we check that each $\rho_H(\mu)$ is an algebra map (it is
necessarily an automorphism, since it has inverse $\rho_H(\mu^{-1})$):
\begin{align*}
\rho_H(\mu)(hh') = &\ \sum \mu^{-1}(\one{(hh')}) \two{(hh')} = \sum
\mu^{-1}(\one{h} \one{h'}) \two{h} \two{h'}\\
= &\ \sum \mu^{-1}(\one{h}) \two{h} \cdot \sum \mu^{-1}(\one{h'})
\two{h'} = \rho_H(\mu)(h) \rho_H(\mu)(h'),
\end{align*}

\noindent where the penultimate equality holds because $\mu$ is an
algebra map. (That $\rho_H(\mu)(1) = 1$ is obvious.) Also note that
$\rho_H$ is injective because if $\rho_H(\mu) = \id_H$, then applying
$\vi_H$ to both sides yields: $\vi_H(h) = \vi_H(\rho_H(\mu(h))) =
\mu^{-1}(h)$ for all $h \in H$. It follows that $\mu^{-1} \equiv \vi_H$
on $H$, whence $\mu = \vi_H$ as desired.

Finally, the root-to-weight map $\Psi_\vi$ clearly has image in
$\widehat{H}$. That $\Psi_\vi$ is a surjection follows if we show that
$\rho_H$ is its right-inverse; but this is a straightforward computation:
\[ \Psi_\vi(\rho_H(\mu))(h) = \vi \circ (\rho_H(\mu)^{-1})(h) =
\vi(\rho_H(\mu^{-1})(h)) = \sum \vi(\mu(\one{h}) \two{h}) = (\mu *
\vi)(h) = \mu(h)\ \forall h \in H. \]

\item The categorical statement follows by observing that the following
square commutes, given Hopf algebras $H_0 \hookrightarrow H_1
\hookrightarrow H$:
\begin{equation}\label{Esquare}
\begin{CD}
\widehat{H_1} @>\rho_{H_1}>> \Aut_{\F-alg}(H_1) \cap
(\pi'_{H_0})^{-1}(\Aut_{\F-alg}(H_0))\\
@V\pi_{H_0} VV @V\pi'_{H_0} VV\\
\widehat{H_0} @>\rho_{H_0}>> \Aut_{\F-alg}(H_0)
\end{CD}
\end{equation}

\item This assertion follows from the definitions.

\item We compute, using that the comultiplication (or antipode) is
(anti)multiplicative in $H$:
\begin{align*}
(\ad hh')(a) = &\ \sum \one{(hh')} a S(\two{(hh')}) = \sum \one{h}
\one{h'} a S(\two{h'}) S(\two{h})\\
= &\ \sum \one{h} (\ad h'(a)) S(\two{h}) = \ad h ( \ad h' (a)).
\end{align*}

\noindent We conclude this part by computing: $(\ad 1)(a) = 1 \cdot a
\cdot 1^{-1} = a$ for all $a \in A$.

\item We show both inclusions. First if $a \in A_\mu$ and $h \in H$, then
compute:
\begin{align*}
a \rho_H(\mu^{-1})(h) = &\ \sum a \mu(\one{h}) \two{h} = \sum (\ad
\one{h})(a) \two{h} = \sum \one{h} a S(\two{h}) \three{h}\\
= &\ \sum \one{h} \vi_H(\two{h}) a = ha =
\rho_H(\mu)(\rho_H(\mu^{-1})(h)) a.
\end{align*}

\noindent Since $\rho_H(\mu^{\pm 1})$ is an automorphism of $H$, it
follows that $a \in A_{\rho_H(\mu)}$. Conversely, if $a \in
A_{\rho_H(\mu)}$ and $h \in H$, then
\begin{align*}
\ad h(a) = &\ \sum \one{h} a S(\two{a}) = \sum a
\rho_H(\mu^{-1})(\one{h}) S(\two{h})\\
= &\ a \sum \mu(\one{h}) \two{h} S(\three{h}) = a \sum \mu(\one{h})
\vi_H(\two{h}) = (\mu * \vi_H)(h) a = \mu(h) a.
\end{align*}
\end{enumerate}
%%%}
\end{proof}

It is now possible to show how HRTAs relate to RTAs.

\begin{proof}[Proof of Theorem \ref{Thrtarta}]
In proving the first assertion, we focus only on the conditions (RTA2)
and (HRTA2). Suppose first that $H_1 \supset H_0$ are Hopf algebras and
$\calq^+_r \subset \im(\rho_{H_r})$ for $r=0,1$.
Define $\calq^{'+}_r := \rho_{H_r}^{-1}(\calq^+_r)$. Then $\calq^{'+}_1 =
\pi_{H_0}^{-1}(\calq^{'+}_0)$ by \eqref{Esquare}.
Moreover, if the root space $B^+_{\theta_1} \neq 0$ for some $\theta_1 =
\rho_{H_1}(\mu_1) \in \calq^+_1$, then $\pi'_{H_0}(\theta_1) =
\rho_{H_0}(\pi_{H_0}(\mu_1)) \in \calq^+_0$ by \eqref{Esquare}. But then
$\pi_{H_0}(\mu_1) \in \calq^{'+}_0$. This shows the decomposition in
condition (HRTA2). That condition (HRTA2) holds now follows by using
Proposition \ref{Prho}. Hence $A$ is an HRTA.

Conversely, suppose $A$ is an HRTA. Then $H_1 \supset H_0$ are clearly
Hopf algebras. Now choose $\calq^+_r \subset \Aut_{\F-alg}(H_r)$ to be
$\rho_{H_r}(\calq^{'+}_r) \subset \im(\rho_{H_r})$ for $r=0,1$ (via
Proposition \ref{Prho}). Moreover, Proposition \ref{Prho} and the
decomposition in condition (HRTA2) imply that the decomposition in (RTA2)
holds as well.

Finally, the last assertion is easily verified, if we set $\Delta :=
\rho_{H_0}(\Delta')$ when $A$ is a based HRTA. Note that since
$\rho_{H_0}$ is injective, the two possible notions of the rank of a
strict, based HRTA coincide.
\end{proof}

\begin{remark}
When $H_1 \supset H_0$ are Hopf algebras, Proposition \ref{Prho}(1) shows
how the weight-to-root and root-to-weight maps help identify roots with
weights. For general RTAs, the maps $\vi, \Psi_\vi, \rho_{H_r}$ need not exist,
and so roots and weights necessarily lie in different spaces that need
not be identifiable with one another.
\end{remark}
%}}}

\section{The BGG Category $\calo$}\label{Sverma}

Having introduced the general framework of interest, the next step is to
define and study the Bernstein-Gelfand-Gelfand Category $\calo$ for an
RTA $A$. In this section we develop the theory of Category $\calo$ for
regular triangular algebras. The main results in this section are
described in Section \ref{Sos}. Following the theory, in subsequent
sections we discuss how the results in this section apply to a large
number of examples, traditional as well as modern, classical as well as
quantum.

\begin{definition}
Given an RTA $A$, the {\em BGG Category $\calo$} is the full subcategory
of all finitely generated $H_1$-semisimple $A$-modules with
finite-dimensional $H_1$-weight spaces, on which $B^+$ acts locally
finitely. (Henceforth by a weight space we mean an $H_1$-weight space,
unless specified otherwise.)
\end{definition}

%{{{1 Section 3.1 - Verma modules; weights fixed by roots
\subsection{Verma modules; weights fixed by roots}

Category $\calo$ was introduced by Bernstein, Gelfand, and Gelfand in
their seminal paper \cite{BGG1} in the setting of complex semisimple Lie
algebras. Since then, similar categories of modules have been studied in
the literature in a wide variety of other settings, including Kac-Moody
Lie algebras, quantum groups, and several other algebras with triangular
decomposition. In studying $\calo$ for these algebras, a common theme is
to carefully examine the structure of a distinguished family of objects
called Verma modules. We now introduce this and other notions in the
general setting of regular triangular algebras.

\begin{definition}
$A$ is an RTA.
\begin{enumerate}
\item Given $\lambda \in \widehat{H_1}$, the corresponding {\em Verma
module} is $M(\lambda) := A / (A \cdot N^+ + A \cdot \ker \lambda)$.

\item The {\em Harish-Chandra projection} is $\xi : A = H_1 \oplus (N^-
\cdot A + A \cdot N^+) \twoheadrightarrow H_1$.
\end{enumerate}
\end{definition}

We now begin to develop the theory of Category $\calo$ via a careful
study of Verma modules and related objects in $\calo$. An attractive
feature of our framework of RTAs is that it is robust enough that much of
the ``traditional" development of $\calo$ in more classical settings goes
through for RTAs as well. More precisely, several of the results in this
section can be proved by adapting the arguments in \cite{MP,Kh2} to RTAs.
Thus, the proofs in this section will occasionally be omitted for
brevity. This applies in particular to the following result.

\begin{prop}\label{Pbasic}
Fix an RTA $A$ and a weight $\la \in \widehat{H_1}$.
\begin{enumerate}
\item Every submodule and quotient of an $H_1$-semisimple module $M$ is
also $H_1$-semisimple.

\item $M(\lambda)$ is an $H_1$-weight module generated by a
one-dimensional subspace of its $\lambda$-weight space. It is a free rank
one $B^-$-module.

\item The center $Z(A)$ acts by a central character $\chi_\lambda$ on the
Verma module $M(\lambda)$.

\item On $Z(A)$, the Harish-Chandra projection $\xi$ is an algebra map
that commutes with the anti-involution $i$ (i.e., $\xi \circ i = \xi$),
and $\chi_\lambda = \lambda \circ \xi$.
\end{enumerate}
\end{prop}

However, in the general setting of RTAs, one encounters certain technical
issues involving Verma modules. More specifically, it is not always true
that all Verma modules lie in Category $\calo$. We now present such an
example, which falls outside the traditional Hopf setting but is an RTA
(and hence can be studied using the methods developed in this paper).

\begin{exam}\label{EJZ}
Motivated by quantum algebras associated to Hecke R-matrices, Jing and
Zhang \cite{JZ} introduced and studied a family of noncommutative and
non-cocommutative bialgebras that $q$-deform $U(\lie{gl}_2)$. (These
algebras were also studied later by Tang \cite{Ta1} from the viewpoint of
hyperbolic algebras.) More precisely, given $q \in \F^\times$ and $\ch \F
\neq 2$, the algebra $U'_q(\lie{gl}_2)$ is defined to be generated by
$u,d,h,a$, with relations:
\[ qhu - uh = 2u, \qquad hd - qdh = -2d, \qquad ud - qdu = a + h +
\frac{1-q}{4} h^2, \]

\noindent where $a$ is central. The algebra $U'_q(\lie{sl}_2)$ is defined
to be the quotient of $U'_q(\lie{gl}_2)$ by the central ideal $(a)$. Note
that setting $q=1$ yields the usual enveloping algebras of $\lie{sl}_2$
and $\lie{gl}_2$ respectively. Now it is not hard to show the following
result.

\begin{prop}
Suppose $q \in \F^\times$ is not a root of  unity, and $\ch \F \neq 2$.
Then $U'_q(\lie{gl}_2), U'_q(\lie{sl}_2)$ are strict, based RTAs of rank
one -- but not Hopf RTAs -- with $H_1 = H_0$ equal to $\F[a,h]$ and
$\F[h]$ respectively, and
\[ B^- = \F[d], \quad B^+ = \F[u], \quad \Delta = \{ \theta \}, \quad
\theta(h) = qh-2, \quad \theta(a) = a, \quad i(u) = d. \]

\noindent Moreover, $M(\lambda) \in \calo$ if and only if $\lambda \neq
-2/(1-q)$, and $M(-2/(1-q))$ is an $H_1$-weight module, with exactly one
weight space of infinite dimension.
\end{prop}

\begin{proof}
The only nontrivial property to check is (RTA1) -- this is shown in far
greater generality in Lemma \ref{Lgwa}. The remaining properties are easy
to verify -- e.g., that $\theta$ generates an infinite cyclic subgroup of
$\Aut_{\F-alg}(H_0)$ follows from the fact that $q$ is not a root of
unity.
\end{proof}
\end{exam}

With this motivating example in mind, we introduce the following
notation.

\begin{definition}\label{Dposet}
Suppose $A$ is an RTA.
\begin{enumerate}
\item Define $\hzfree := \{ \lambda_0 \in \widehat{H_0} :
\tangle{\calq^+_0} \text{ acts freely on } \lambda_0 \}$ and
$\hofree := \pi_{H_0}^{-1}(\hzfree)$.
If $A$ is a strict RTA, we will denote this common set by $\hfree$.

\item Define {\em partial orders} on the following four spaces:
\begin{itemize}
\item Define $\geq_{\calq_0}$ on $\Aut_{\F-alg}(H_0)$ via:
$\theta_0 \geq \theta'_0$ if there exists $\theta''_0 \in \calq^+_0$ such
that $\theta_0 = \theta''_0 * \theta'_0$.

\item Define $\geq_{\calq_1}$ on $\Aut_{\F-alg}(H_1)$ via:
$\theta_1 \geq \theta'_1$ if either $\pi_{H_0}(\theta_1) >
\pi_{H_0}(\theta'_1)$, or $\theta_1 = \theta'_1$.

\item Define $\geq_0$ on $\hzfree$ by: $\mu_0 \geq \mu'_0$ if there
exists $\theta''_0 \in \calq^+_0$ such that $\mu_0 = \theta''_0 *
\mu'_0$.

\item Define $\geq_1$ on $\hofree$, via: $\mu_1 \geq \mu'_1$ if
$\pi_{H_0}(\mu_1) > \pi_{H_0}(\mu'_1)$ in $\widehat{H_0}$, or $\mu_1 =
\mu'_1$ in $\widehat{H_1}$.
\end{itemize}

\item A {\em maximal vector} of weight $\la$ in an $A$-module $M$, is $m
\in M_\la \cap \ker N^+$.
\end{enumerate}
\end{definition}

\noindent In the remainder of the paper, we will often use $\geq$ without
specifying which of the four aforementioned partial orders is being used,
when this is clear from context.

\begin{remark}\label{Rfree}
If $A$ is an HRTA with $\calq^+_r = \rho_{H_r}(\calq^{'+}_r)$ for
$r=0,1$, then $\widehat{H_r}^{free} = \widehat{H_r}$ by Proposition
\ref{Prho}(3). For this reason, in many examples in the literature (and
below) one works with all of $\calo$, since all Verma modules lie in
$\calo$.
\end{remark}

In the rest of the paper, we work with Verma modules (and their
quotients) with highest weights in the set $\hofree$ -- these modules are
objects in $\calo$. The following result summarizes the basic properties
of Verma modules and their unique simple quotients.

\begin{prop}\label{Pverma}
Fix an RTA $A$ and a weight $\la \in \hofree$.
\begin{enumerate}
\item The relations $\geq$ in $\widehat{H_r}$ are partial orders when
restricted to $\widehat{H_r}^{free}$ for $r=0,1$. The map $\pi_{H_0}$ is
an order-preserving map when restricted to $\hofree$.

\item $M(\la)$ is an indecomposable object of $\calo$, generated by its
one-dimensional $\la$-weight space. All other weight spaces have weights
$\mu < \la$ with $\mu \in \hofree$.

\item Every proper submodule of $M(\la)$ is $H_1$-semisimple and has zero
$\la$-weight space.

\item $M(\la)$ has a unique maximal submodule $\Rad M(\lambda)$, and a
unique simple quotient $L(\la)$.

\item $M(\la)$ is the ``universal" cyclic module of highest weight $\la$.

\item If $v \in M(\la)_\mu$ is maximal, then $\mu \leq \la$ in
$\widehat{H}$, and $[M(\la) : L(\mu)] > 0$.

\item The simple objects in $\calo$ with at least one weight in $\hofree$
are precisely $L(\lambda)$ for some $\lambda \in \hofree$. All such
modules are pairwise non-isomorphic.
\end{enumerate}
\end{prop}

\begin{proof}
Most of the proofs are similar to those in \cite{MP,Kh2}, and are hence
not included for brevity, except for the last part. In that part, fix a
simple module $V$ in $\calo$ with nonzero weight space $V_\mu$ for some
weight $\mu \in \hofree$. Then the vector space $B^+ V_\mu$ is
finite-dimensional and $H_1$-semisimple by the assumptions on $\calo$.
Thus it contains a weight vector $v_\la$ of maximal $H_1$-weight
$\lambda$ in the partial order $\geq_{H_1}$. Since $V$ is simple, it is
generated by $v_\lambda$, whence $V \cong L(\lambda)$ by part (4).
\end{proof}

\begin{remark}
It is also possible to introduce the \textit{Shapovalov form} $Sh : A
\times A \to H_1$ for a general RTA $A$, by defining: $Sh(x,y) := \xi(i(x)
y)$ for $x,y \in A$. One verifies that the Shapovalov form satisfies the
following properties for RTAs, which it satisfies for $A = U \lie{g}$ for
semisimple $\lie{g}$:
(a) The Shapovalov form is bilinear and symmetric.
(b) $Sh(A_{\theta_r}, A_{\theta'_r}) = 0$ unless $\theta_r = \theta'_r$.
(c) The Shapovalov form induces a symmetric bilinear form $Sh_\lambda$ on
every Verma module $M(\lambda)$ via: $Sh_\lambda(b_1 m_\lambda, b_2
m_\lambda) := \lambda(Sh(b_1, b_2))$ for $b_1, b_2 \in B^-$ and
$m_\lambda$ a nonzero highest weight vector in $M(\lambda)_\lambda$.
(d) If $\lambda \in \hofree$ then $\ker (Sh_\lambda) = \Rad
(M(\lambda))$.
(e) Given $\lambda \in \hofree$ and $\theta_1 \in \rt_{H_1}(B^-)$,
consider the restriction of the form $Sh(-,-)$ to the root space
$B^-_{\theta_1}$. Then if one applies $\lambda$ to each entry of the
matrix of this bilinear form (with respect to any fixed basis of
$B^-_{\theta_1}$), the resulting matrix has rank equal to $\dim
L(\lambda)_{\theta_1 * \lambda}$.
\end{remark}

\begin{remark}\label{Rnotions}
Various other notions from the theory of semisimple Lie algebras also
have analogues for general RTAs. For instance, the \textit{Kostant
partition function} has analogues $\calp_r : \calq^+_r \to \nn$ defined
via $\calp_r(\theta_r) := \dim_\F B^+_{\theta_r}$, for $r=0,1$.
Next, \textit{highest weight modules} $\mathbb{V}^\lambda$ are quotients
of Verma modules $M(\lambda)$; if $\lambda \in \hofree$ then
$\mathbb{V}^\lambda \in \calo$ since $\calo$ is closed under quotienting.

Now suppose $A$ is a strict, based RTA with a base of simple roots
$\Delta$ of smallest possible size. One can then define the
\textit{height} of a ``restricted root" $\theta_0 = \sum_{\theta \in
\Delta} n_\theta \theta \in \Z \Delta = \tangle{\calq^+_0} \subset
\Aut_{\F-alg}(H_0)$, to be $\hgt(\theta_0) := \sum_{\theta \in \Delta}
n_\theta$. Similarly, define \textit{parabolic/Levi regular triangular
subalgebras} as follows: for any subset $\Delta_0 \subset \Delta$, define
$B^\pm_{\Delta_0} := \bigoplus_{\theta_1 \in (\pi'_{H_0})^{-1}(\Z
\Delta_0)} B^\pm_{\theta_1}$, and
\[ \lie{P}^\pm_{\Delta_0} := B^\pm \otimes H_1 \otimes B^\mp_{\Delta_0}
\quad \supset \quad \lie{L}^\pm_{\Delta_0} := B^\pm_{\Delta_0} \otimes
H_1 \otimes B^\mp_{\Delta_0}, \]

\noindent or more precisely, (the subalgebras of $A$ generated by) their
images under the multiplication map. Thus one can study notions such as
\textit{parabolic/generalized Verma modules}, as well as analogues of
``parabolic" induction over based RTAs.
\end{remark}
%}}}

%{{{1 Section 3.2 - Duality and extensions
\subsection{Duality and extensions}

We next construct a duality functor on finite length objects in $\calo$.
In light of Proposition \ref{Pverma} and Example \ref{EJZ}, henceforth we
only work with objects in $\calo$ whose weights lie in $\hofree$. The
following notation is required for this purpose.

\begin{definition}
Define $\calo[\hofree]$ and $\calo_\N$ to respectively be the full
subcategories of objects in $\calo$ whose weights lie in $\hofree$ and
which are of finite length. Also define $\calo_\N[\hofree] :=
\calo[\hofree] \cap \calo_\N$. Next, define the {\em formal character} of
$M \in \calo$ to be: $\ch M := \sum_{\lambda \in \widehat{H_1}} \dim
M_\lambda \cdot e^\lambda$, where $e^\lambda$ is a formal variable for
each $\lambda \in \widehat{H_1}$.
Finally, given an object $M$ in $\calo$, use the anti-involution $i : A
\to A$ to define its {\em restricted dual} $F(M) := \bigoplus_{\lambda
\in \wt M} M_\lambda^*$, with $A$-module structure given by: $(a m^*)(m)
:= m^*(i(a) m)$.
\end{definition}

\noindent In particular, it follows from Proposition \ref{Pverma} that
the simple objects in $\calo[\hofree]$ are parametrized by $\hofree$. As
discussed above, we work henceforth only in $\calo[\hofree]$; however,
the next result holds in all of $\calo$. The proof is as in the special
case when $A = U \lie{g}$ for semisimple $\lie{g}$; see \cite{MP,Kh2}.

\begin{prop}\label{Pdual}
$\calo_\N$ and $\calo_\N[\hofree]$ are abelian categories, and $F :
\calo_\N \to \calo_\N$ is an exact, contravariant duality functor that
sends each simple object $L(\lambda)$ for $\lambda \in \hofree$ to
itself. More generally, $F$ preserves the length and formal character of
all objects in $\calo_\N$ and $\calo_\N[\hofree]$ respectively.
\end{prop}

The above results allow us to now consider \textit{extensions}. A key
result involves classifying all non-split objects in $\calo[\hofree]$ of
length two.

\begin{theorem}\label{Text}
Fix $\lambda, \lambda' \in \hofree$. Then $E(\lambda, \lambda') :=
\Ext^1_\calo(L(\lambda), L(\lambda'))$ is nonzero if and only if $\Rad
M(\lambda) \twoheadrightarrow L(\lambda')$, or $\Rad M(\lambda')
\twoheadrightarrow L(\lambda)$. Moreover, $F$ induces an isomorphism $ :
E(\lambda,\lambda') \leftrightarrow E(\lambda',\lambda)$. Finally,
$\Ext^1_\calo(M,N)$ is finite-dimensional for $M,N \in
\calo_\N[\hofree]$.
\end{theorem}

\begin{proof}
That $F$ is an isomorphism $: E(\lambda,\lambda') \to
E(\lambda',\lambda)$ follows from Proposition \ref{Pdual} (and standard
arguments), since $F$ is contravariant and exact. Now if $\Rad M(\lambda)
\twoheadrightarrow L(\lambda')$ with kernel $V$, then
\[ 0 \to L(\lambda') = (\Rad M(\lambda) / V) \to M(\lambda) / V \to
L(\lambda) \to 0, \]

\noindent and this is non-split, else $M(\lambda) \twoheadrightarrow
M(\lambda) / V \twoheadrightarrow L(\lambda')$, whence $\lambda =
\lambda'$ and $\dim M(\lambda)_\lambda \geq 2$, which is false.
Conversely, suppose $0 \to L(\lambda') \to M \to L(\lambda) \to 0$ is
nonsplit in $\calo$, and let $v_\lambda \in L(\lambda)_\lambda,
v_{\lambda'} \in L(\lambda')_{\lambda'}$ be nonzero highest weight
vectors in the first and third terms of the short exact sequence
respectively. Also fix any lift $m_\lambda \in M_\lambda$ of $v_\lambda$,
so that $N^+ m_\lambda \subset L(\lambda')$ (where $N^+$ is the
augmentation ideal in $B^+$). There are three cases:

First if $\pi_{H_0}(\lambda) = \pi_{H_0}(\lambda') =: \lambda_0$, say,
then by Proposition \ref{Pverma} $M_{\lambda_0}$ is a two-dimensional
$H_0$-weight space, spanned by $v_{\lambda'}$ and $m_\lambda$. Now $B^-
m_\lambda$ is a nonzero submodule of $M$, and it has trivial intersection
with the simple $A$-module $L(\lambda')$ since otherwise $v_{\lambda'}
\in B^- m_\lambda$. Therefore the short exact sequence splits, which is
impossible.

The second case is if $N^+ m_\lambda = 0$ and $\pi_{H_0}(\lambda) \neq
\pi_{H_0}(\lambda')$. Then $M(\lambda) \twoheadrightarrow B^- m_\lambda
\twoheadrightarrow L(\lambda)$, so if the extension is nonsplit then $B^-
m_\lambda$ is not simple and hence $L(\lambda') \subset B^- m_\lambda$.
But then $B^- m_\lambda$ has length 2, hence $M = B^- m_\lambda =
M(\lambda) / V$, say. It follows that $L(\lambda') = (\Rad M(\lambda)) /
V$, proving the assertion.
Furthermore, the nonsplit extension class is completely determined by
$\theta$ and $b_- \in B^-_\theta$ such that $\theta * \lambda = \lambda'$
and $b_- m_\lambda = v_{\lambda'}$. Thus using condition (RTA2),
\[ \dim \Ext^1_\calo(L(\lambda),L(\lambda')) \leq \dim
B^-_{\pi'_{H_0}(\theta)} < \infty. \]

Finally, suppose $0 \neq N^+ m_\lambda \subset L(\lambda')$, so that
$\lambda < \lambda'$. In this case we use the duality functor $F$ to
reduce to the previous case. This proves the first two assertions of the
theorem. The final assertion now follows by using Proposition \ref{Pdual}
and standard homological arguments in $\calo_\N[\hofree]$.
\end{proof}
%}}}

%{{{1 Section 3.3 - Blocks in $\calo$, Conditions (S), and main results
\subsection{Blocks in $\calo$, Conditions (S), and main
results}\label{Sos}

We now describe the two main results in this paper, on Category $\calo$
over an arbitrary RTA. The results provide sufficient conditions under
which a large subcategory of $\calo$ -- in fact of $\calo[\hofree]$ --
acquires an increasing number of desirable homological properties. To
state and prove these results requires the following notation.

\begin{definition}\label{DcondS}
Suppose $A$ is an RTA.
\begin{enumerate}
\item For each weight $\la \in \hofree$, define the following four sets:
\begin{itemize}
\item $S^4(\la) := \{ \mu \in \widehat{H_1} : \chi_\mu \equiv \chi_\la
\mbox{ on } Z(A) \}$ (where $\chi_\la$ denotes the central character
defined in Proposition \ref{Pbasic}).

\item $S^3(\la)$ is the equivalence closure of $\{ \la \}$ in
$\widehat{H_1}$, under the relation:
\[ \mu \to \la \text{ if and only if } L(\mu) \text{ is a subquotient of
} M(\la). \]

\item $S^2(\la) := \{ \pi_{H_0}(\mu) : \mu \in S^3(\la) \}$.

\item $S^1(\la) := \{ \pi_{H_0}(\mu) : \mu \in S^3(\la), \mu \leq \la
\}$.
\end{itemize}

\item For $1 \leq m \leq 4$, define
\begin{equation}
S^m(A) := \{ \lambda \in \hofree : S^m(\lambda) \text{ is finite} \}
\subset \hofree.
\end{equation}

\noindent (Note that $S^1(A), S^2(A) \subset \hofree$ although
$S^1(\lambda), S^2(\lambda) \subset \widehat{H_0}$.) We say that the
algebra $A$ satisfies {\em Condition (S1), (S2), (S3)}, or {\em (S4)} if
the corresponding set $S^m(A)$ equals $\hofree$.

\item Given $T \subset \hofree$, define $\calo[T]$ to be the full
subcategory of $\calo$, such that every simple subquotient of every
object is of the form $L(\lambda)$ for some $\lambda \in T$. (This is
consistent with the definition of $\calo[\hofree]$.)
Now given $\lambda \in \hofree$, define the corresponding {\em block} of
$\calo$ to be $\calo[S^3(\lambda)]$.
\end{enumerate}
\end{definition}

The idea behind the conditions (S) is that each of them implies
increasingly desirable homological and representation-theoretic
properties for $\calo$. (For instance, the sets $S^4(\lambda)$ are
related to central characters, while $S^3(\lambda)$ are concerned with
linkage.) Thus, in some sense $\calo[S^m(A)]$ is \textit{the part of
Category $\calo$ that satisfies these (desirable) properties}. This is
made clearer in our ``first main result", Theorem \ref{Tfirst} below.
In later sections, we show that a large number of well-explored settings
in representation theory are all examples of RTAs, and explore whether or
not these algebras satisfy the various Conditions (S).
We are also motivated by settings such as \cite{FeFr}, in which it is
often the case that distinguished subcategories/sums of blocks in $\calo$
are shown to have desirable properties or a tractable analysis.

\begin{remark}
The $S$-sets should not be confused with the antipode map on $H$ in the
event that $H$ is a Hopf algebra. In fact we do not use the antipode in
the remainder of the paper, except in Proposition \ref{Pgcm-qg}.
\end{remark}

In order to state our main results, we need a further piece of notation.

\begin{definition}
Suppose $A$ is an RTA.
Define
\begin{equation}
\overline{S^2}(A) := \{ \lambda \in S^2(A) : \mu'_0 \leq \pi_{H_0}(\mu)
\leq \mu''_0 \text{ and } \mu'_0, \mu''_0 \in S^2(\lambda) \implies \mu
\in S^1(A) \}.
\end{equation}
%\overline{S^3}(A) := &\ \{ \lambda \in S^3(A) : \mu' < \mu < \mu''
%\text{ and } \mu', \mu'' \in S^3(\lambda) \implies \mu \in S^1(A) \}.
%\end{align}
%
%\begin{equation}
%\widetilde{S^1}(A) := \{ \lambda \in \hofree : S^1(\mu) \text{ is finite
%for all } \mu \in \tangle{\calq^+_1} * \lambda \} \subset S^1(A).
%\end{equation}

\noindent Also say that an RTA is {\em discretely graded} if for all
$\theta_0 \in \calq^+_0$, the interval $[0, \theta_0]$ (in the partial
order $\leq$ on $\calq_0^+$) is finite.
\end{definition}

Note that $\overline{S^2}(A)$ is precisely the set of weights $\lambda
\in S^2(A)$ such that $\pi_{H_0}^{-1}([S^2(\lambda)]_\leq) \subset
S^1(A)$, where $[T]_\leq$ denotes the closure of $T \subset
\widehat{H_0}^{free}$ in the partial order induced by $\calq_0^+$.
%The set $\overline{S^3}(A)$ has a similar interpretation, consisting of
%all $\lambda \in S^3(A)$ such that $[S^3(\lambda)]_{\leq_H} \subset
%S^1(A)$, in the partial order $\leq_H$ on $\hofree$.

\begin{remark}
The assumption of being discretely graded is weaker than most algebras
studied in the literature, which are moreover based with a finite set of
simple roots. In fact based RTAs with an \textit{infinite} base of simple
roots are also discretely graded. There are other examples of non-based
but discretely graded RTAs that arise from mathematical physics, such as
generalized Heisenberg algebras for discrete, totally ordered groups. See
Section \ref{Snonbased} for more details.
\end{remark}

We now discuss some results on the $S$-sets and the Conditions (S).
First, the nomenclature is inspired by the ``T"-properties of
separation/Hausdorffness in topology, in the following sense.

\begin{lemma}\label{Lcs}
$S^3(\la) \subset S^4(\la) \cap (\tangle{\calq^+_1} * \lambda)$ for all
$\la \in \widehat{H}$, so $S^4(A) \subset S^3(A) \subset S^2(A) \subset
S^1(A)$. Therefore the following implications hold among the Conditions
(S): $(S4) \Rightarrow (S3) \Rightarrow (S2) \Rightarrow (S1)$.

Moreover, if $S^2(A) = \hofree$ then $\overline{S^2}(A) = \hofree$ as
well.
\end{lemma}

Additionally (like the separation properties), the $S^m$-sets/conditions
yield increasingly (in $m$) useful homological information about Category
$\calo$. The following is one of the two main results involving the
$S$-sets for a general regular triangular algebra.

\begin{utheorem}\label{Tfirst}
Suppose $A$ is a discretely graded (e.g.~based) RTA.
\begin{enumerate}
\item $\calo[S^1(A)]$ is finite length, and hence splits into a direct
sum of blocks $\calo[S^1(A) \cap S^3(\la)]$, each of which is abelian and
self-dual.
%Thus, $\calo$ has a block decomposition
%\[ \calo = \calo[\widehat{H} \setminus \widetilde{S^1}(A)] \oplus
%\bigoplus_{\lambda \in \widetilde{S^1}(A) / S^3} \calo[S^3(\lambda)]. \]

\item Suppose $\lambda \in \overline{S^2}(A)$.
Then the block $\calo[S^3(\la)]$ is abelian and self-dual with enough
projectives, each with a filtration whose subquotients are Verma modules.

\item Suppose $\lambda \in S^3(A) \cap \overline{S^2}(A)$.
Then the block $\calo[S^3(\la)]$ is equivalent to the category
$(\Mod$-$B_\la)^{fg}$ of finitely generated right modules over a
finite-dimensional $\F$-algebra $B_\la$. Moreover, $\calo[S^3(\la)]$ is a
highest weight category; equivalently, the algebra $B_\la$ is
quasi-hereditary.
%Many different notions of block decomposition all coincide, including
%the one that uses the categorical definition of linking.
\end{enumerate}

\noindent In particular, if $A$ satisfies condition $(Sm)$ for some $m$,
then the corresponding assertion above (numbered $\min(m,3)$) holds for
all of $\calo[\hofree]$. Thus if $(S3)$ holds, we obtain a block
decomposition
\[ \calo = \calo[\widehat{H_1} \setminus \hofree] \oplus
\bigoplus_{\lambda \in \hofree / S^3} \calo[S^3(\lambda)]. \]
\end{utheorem}

\begin{remark}
(For the definition of a highest weight category, see \cite{CPS1}.) Thus,
if $A$ satisfies (S3), then Theorem \ref{Tfirst} implies that each block
$\calo[S^3(\lambda)]$ has enough projectives (each filtered with Verma
subquotients), finite cohomological dimension, {\em tilting modules}
(i.e., modules simultaneously filtered in $\calo$ by standard as well as
costandard subquotients -- see \cite{Rin,Don}), and the property of {\em
BGG reciprocity}. These properties then transfer to all of
$\calo[\hofree]$.
Thus, Theorem \ref{Tfirst} implies that the algebras $B_\la$ are {\em BGG
algebras} (see \cite{Irv}). We do not discuss these results in great
detail as they are homological properties valid in all highest weight
categories; however, some of these results are stated in Theorem
\ref{Thwc} to give the reader a flavor of highest weight categories. We
also refer the interested reader to the comprehensive program developed
by Cline, Parshall, and Scott for more on such categories.
\end{remark}

We state our second main result about Category $\calo$ and the Conditions
(S) over regular triangular algebras: these constructions are all
\textit{functorial}.

\begin{utheorem}\label{Tfunct}
Suppose $A_j = B^-_j \otimes H_{1j} \otimes B^+_j$ (with $H_{1j} \supset
H_{0j}$) is a (Hopf) RTA for $1 \leq j \leq n$.
\begin{enumerate}
\item Then so is $A := \otimes_{j=1}^n A_j$. Moreover, $A$ is strict
and/or based (and discretely graded), if and only if so is $A_j$ for all
$j$.

\item A module $V \in \calo[\hofree]$ is simple if and only if $V =
\otimes_{j=1}^n V_j$, with $V_j$ simple (and unique up to isomorphism) in
$\calo[\hofree_j]$ for all $j$.

\item Each of the Conditions (S) holds for $A$ if and only if it holds
for all $A_j$. More generally, for $1 \leq m \leq 4$,
\begin{equation}
S^m_A((\la_1, \dots, \la_n)) = \times_{j=1}^n S^m_{A_j}(\la_j), \qquad
S^m(A) = \times_{j=1}^n S^m(A_j), \qquad
\forall 1 \leq j \leq n,\ \lambda_j \in \hofree_j
\end{equation}

\noindent as subsets of $\widehat{H_r} = \times_{j=1}^n \widehat{H_{rj}}$
for suitable $r=0,1$. Furthermore, $\overline{S^2}(A) = \times_{j=1}^n
\overline{S^2}(A_j)$.

\item Complete reducibility for finite-dimensional modules holds in
$\calo[\hofree]$ if and only if it holds in $\calo[\hofree_j]$ for all $1
\leq j \leq n$.
\end{enumerate}
\end{utheorem}

In other words, it is possible to relate all of these notions for a
tensor product $A = \otimes_{j=1}^n A_j$ of commuting RTAs $A_i$, with
their counterparts for the individual tensor factors $A_i$. This is akin
to (and more general than) relating representations of semisimple Lie
algebras with those of the individual simple ideals.
In fact, Theorem \ref{Tfunct} provides a useful approach to take in
studying Category $\calo$ over newly introduced and studied classes of
RTAs. For example, this was the approach adopted in \cite{Zhi}, where
Zhixiang showed the algebra of interest to be a strict, based Hopf RTA of
rank one that satisfies Condition (S4). See Example \ref{Ezhi} for more
details.
%}}}

%{{{1 Section 3.4 - Proofs of main results
\subsection{Proofs of main results}\label{Sproofs}

The remainder of this section is devoted to proving Theorems \ref{Tfirst}
and \ref{Tfunct}. We will sketch those arguments which are along the
lines of similar results in \cite{MP,Kh2}; but we will spell out the
details when illustrating how the more general structure of a (discretely
graded) regular triangular algebra is used.

We begin with results in the spirit of the original paper \cite{BGG1},
which help explicitly construct projective modules in Category $\calo$
over discretely graded RTAs.
%In fact Theorem \ref{Tfirst}(2),(3) use a stronger assumption than
%%Theorem \ref{Tfirst}(1), because we explicitly construct projectives in
%blocks $\calo[S^3(\lambda)]$. 
%In contrast, we use an existence argument to prove Theorem
%\ref{Tfirst}(3), so that part is stated for all $\lambda \in S^3(A)$.
 To do so, we introduce the following notation.

\begin{definition}
Suppose $A$ is an RTA. Given a subset $\Theta_0 \subset \calq_0^+$ and
$\lambda \in \widehat{H_1}$, define
\begin{equation}
B_{\Theta_0 +} := \sum_{\theta \in \calq_0^+, \ \theta \not\leq \theta_0\
\forall \theta_0 \in \Theta_0} B^+_\theta, \qquad
P(\lambda,\Theta_0) := A / (A \cdot B_{\Theta_0 +}+ A \cdot \ker
\lambda).
\end{equation}

\noindent Also define $\calo(\lambda, \Theta_0 +)$ to be the full
subcategory of $\calo$ consisting of the objects $M$ for which
$B_{\Theta_0 +} M_\lambda = 0$.
Finally, an $A$-module $M$ is said to have a {\em standard filtration}
(respectively, a {\em highest weight filtration}) if $M$ has a finite
descending chain of $A$-submodules such that the successive quotients are
Verma modules (respectively, quotients of Verma modules).
\end{definition}

\noindent Note that if $\Theta_0$ is finite, then $B_{\Theta_0 +} =
B_{\max(\Theta_0) +}$, where $\max(\Theta_0)$ denotes the $\leq$-maximal
elements of $\Theta_0$.
We now list some properties of the modules $P(\lambda, \Theta_0 +)$ that
are used to prove Theorem \ref{Tfirst}.

\begin{prop}\label{Pproj}
Suppose $A$ is a discretely graded RTA, $\Theta_0 \subset \calq_0^+$ is
finite, and $\lambda \in \hofree$.
\begin{enumerate}
\item The subspace $B_{\Theta_0 +}$ is a left ideal in $B^+$ of finite
codimension. Moreover, $B_{\{ \id_{H_0} \}, +} = N^+$ and $P(\lambda,\{
\id_{H_0} \} +) = M(\lambda)$.

\item $P(\lambda, \Theta_0 +)$ is an object of $\calo(\lambda, \Theta_0
+) \subset \calo$. Moreover, $\hhh_\calo(P(\lambda, \Theta_0 +), M) =
\dim M_\lambda$ for all objects $M$ in $\calo(\lambda, \Theta_0 +)$.

\item $P(\lambda, \Theta_0 +)$ has a standard filtration in $\calo$, and
surjects onto $M(\lambda)$. If $\rt_{H_1}(B^+ / B_{\Theta_0 +}) =
S_{\Theta_0}$ as multisets, then the multiset of Verma subquotients of
$P(\lambda, \Theta_0 +)$ equals $\{ M(\theta * \lambda) : \theta \in
S_{\Theta_0} \}$.

\item An $H_1$-semisimple module $M$ is in $\calo[\hofree]$ if and only
if $M$ is a quotient of a finite direct sum of modules of the form
$P(\lambda, \Theta_0 +)$ for $\lambda \in \hofree$, if and only if $M$
has a highest weight filtration with highest weights in $\hofree$.

\item Given objects $M_1, M_2 \in \calo[\hofree]$, $M_1 \oplus M_2$ has a
standard filtration if and only if each of $M_1$ and $M_2$ has a standard
filtration.
\end{enumerate}
\end{prop}

\noindent We omit the proof as the arguments in \cite{BGG1},
\cite[Appendix A]{Don}, and \cite{Kh2} can be suitably modified to work
for all discretely graded RTAs. Note that when the discretely graded RTA
is $U \lie{g}$ for semisimple $\lie{g}$, we set $H_1 = H_0 := \lie{h}^*$
and $\calq_0^+ := \nn \Delta = \nn \rho_{H_0}(\Delta')$ to lie in the
simple root lattice, and work with the modules $P(\lambda,l) :=
P(\lambda, \Theta_l +)$ for $l \in \nn$, where $\Theta_l := \{ \theta_0
\in \nn \Delta : \hgt(\theta_0) = l \}$, with $\hgt(\theta_0)$ defined in
Remark \ref{Rnotions}.
Indeed, this was the approach adopted in the seminal work \cite{BGG1} to
explicitly construct projective objects in blocks of $\calo$.

It is now possible to prove our first main theorem.

\begin{proof}[Proof of Theorem \ref{Tfirst}]
Along the way to showing the assertions, we prove some intermediate steps
that are useful facts in their own right. The first claim is that part
(1) already holds for $\calo_\N$, i.e.,
\textit{for all $T \subset \hofree$, $\calo_\N[T]$ has a block
decomposition}:
\begin{equation}\label{Efindecomp}
\calo_\N[T] = \bigoplus_{\lambda \in T / S^3} (\calo_\N[T] \cap
\calo[S^3(\lambda)]) = \bigoplus_{\lambda \in T / S^3} \calo_\N[T \cap
S^3(\lambda)],
\end{equation}

\noindent where we sum over distinct blocks, and where each summand is an
abelian, finite-length, and self-dual Serre subcategory of
$\calo_\N[T] \subset \calo_\N[\hofree]$.
Indeed, most of the claim follows by Proposition \ref{Pdual} and standard
arguments, once we show the direct sum decomposition for all finite
length objects in $\calo_\N[T]$. That there are no morphisms or
extensions between objects of distinct blocks follows from the same
statement for \textit{simple} objects of distinct blocks, by using
Theorem \ref{Text} and the long exact sequence of $\Ext_{\calo_\N}$s.

It remains to prove the direct sum decomposition of $\calo_\N[T]$ into
blocks. This is done by induction on the length $l$ of the object in
$\calo_\N[T]$. For $l=0,1,2$, the result is immediate or follows from
Theorem \ref{Text}. Now suppose the result holds for some object $N =
\bigoplus_{\lambda \in T / S^3} N[S^3(\lambda)]$, and we have $0 \to N
\to M \to L(\mu) \to 0$ for some $\mu \in T$. Now use the following
general fact that holds in any abelian category $\mathscr{C}$: if $0 \to
A \oplus B' \to C \to B'' \to 0$ and $\Ext^1_{\mathscr{C}}(B'',A) = 0$,
then the sequence $0 \to A \to C \to C/A \to 0$ splits, and we have $0
\to B' \to C/A \to B'' \to 0$. Write $N = N' \oplus N[S^3(\mu)]$, and set
$A := N', B' := N[S^3(\mu)], C := M, B'' := L(\mu)$.
Applying the above general fact yields $M = N' \oplus M[S^3(\mu)]$, where
$0 \to N[S^3(\mu)] \to M[S^3(\mu)] \to L(\mu) \to 0$; thus Equation
\eqref{Efindecomp} follows. We now prove the various parts of the
theorem.
\begin{enumerate}
\item Given $M \in \calo[S^1(A)]$, observe by Proposition \ref{Pproj}
that $M$ has a highest weight filtration. Moreover, the corresponding
highest weights $\lambda_1, \dots, \lambda_k$ can be shown to lie in
$S^1(A)$. Now for each $\mu_0 \in S^1(\lambda_j)$, by Proposition
\ref{Pverma} there exists a unique $\theta_{j,\mu_0} \in -\calq_0^+$ such
that $\theta_{j,\mu_0} * \pi_{H_0}(\lambda_j) = \mu_0$. Thus,
\[ l(M) \leq \sum_{j=1}^k l(M(\lambda_j)) \leq \sum_{j=1}^k \sum_{\mu_0
\in S^1(\lambda_j)} \dim_\F B^-_{\theta_{j,\mu_0}} < \infty, \]

\noindent since every simple subquotient of a Verma module $M(\lambda)$
is generated by (a lift of) its highest weight vector, whose $H_0$-weight
lies in $S^1(\lambda)$. It follows that $\calo[S^1(A)]$ is finite length.
Now use the above analysis (before this first part) to complete the
proof.

%Finally, it is not hard to show that if two objects in $\calo$ have
%supports (i.e., sets of $H_1$-weights) in disjoint
%$\tangle{\calq^+_1}$-orbits in $\widehat{H_1}$, then all morphisms and
%extensions in $\calo$ between such modules are trivial. It follows that
%$\calo = \calo[\widehat{H_1} \setminus \widetilde{S^1}(A)] \oplus
%\calo[\widetilde{S^1}(A)]$. Now the second summand decomposes into
%blocks $\calo[S^3(\lambda)]$ as above.

\item We first introduce some notation. Fix $\lambda \in S^2(A)$, with
$S^2(\lambda) = \{ \lambda_1, \dots, \lambda_k \}$. Given $\mu \in
\hofree$, define $\theta_j(\mu)$ to be the unique element of $\calq_0^+$
such that $\theta_j(\mu) * \pi_{H_0}(\mu) = \lambda_j$ if there exists
such a $\theta_j(\mu) \in \tangle{\calq^+_0}$, else set $\theta_j(\mu) :=
0 = \id_{H_0}$. Now define $\displaystyle \Theta_\mu := \bigcup_{j=1}^k
[\id_{H_0}, \theta_j(\mu)]_\leq$, where $\leq$ is the partial order
induced on $\hzfree$ by $\calq_0^+$.
(Alternatively, we may define $\Theta_\mu$ to be the set $\{
\theta_j(\mu) \}$, discounting repetitions.) Note that $\Theta_\mu$ is a
finite subset for all $\mu \in \hzfree$ since $\calq_0^+$ is discretely
graded.

We now prove the result. Suppose $\lambda \in \overline{S^2}(A)$ and $M
\in \calo[S^3(\lambda)]$. By Proposition \ref{Pproj}(4), $M$ is generated
by the lifts to $M$ of the highest weight vectors in each of its highest
weight module subquotients. Each of these highest weights $\mu_l$ lies in
$\pi_{H_0}^{-1}(S^2(\lambda))$; thus, $M$ is generated by its
$H_0$-weight spaces of weights $\lambda_j$ for $1 \leq j \leq k$.
It follows by Proposition \ref{Pproj} and the previous paragraph that
$P_M := \bigoplus_{l=1}^N P(\mu_l, \Theta_{\mu_l} +) \twoheadrightarrow
M$, where $\pi_{H_0}(\mu_l) \in \{ \lambda_1, \dots, \lambda_k \}\
\forall l$.
Now use Proposition \ref{Pproj}(4) as well as the definition of
$\overline{S^2}(A)$ to show that $\calo[S^3(\lambda)] \subset
\calo(\mu_l, \Theta_{\mu_l} +) \cap \calo[S^1(A)]$ for all $1 \leq l \leq
N$. Moreover, $P(\mu_l, \Theta_{\mu_l} +)$ is an object of
$\calo[S^1(A)]$ by Proposition \ref{Pproj}(3) and the definition of
$\overline{S^2}(A)$. Denote its summand in the block $\calo[S^1(A) \cap
S^3(\lambda)] = \calo[S^3(\lambda)]$ by $P_l$, say. Then
$\hhh_\calo(P_l,-)$ is exact in $\calo[S^3(\lambda)]$ by Proposition
\ref{Pproj}(2), whence $\oplus_l P_l$ is projective in
$\calo[S^3(\lambda)]$ and surjects onto $M$. This shows that the block
$\calo[S^3(\lambda)]$ has enough projectives.

It remains to show that each indecomposable projective $P$ in
$\calo[S^3(\lambda)]$ has a standard filtration. Since $P$ is finite
length, it has a simple quotient $L(\mu)$ for some $\mu \in
S^3(\lambda)$. Now $P(\mu, \Theta_\mu +) \twoheadrightarrow L(\mu)$ from
above, so its $\calo[S^3(\lambda)]$-summand $P_\lambda$ surjects onto
$L(\mu)$. By universality, this surjection factors through a nonzero map
$: P_\lambda \to P \twoheadrightarrow L(\mu)$. Now replace $P_\lambda$ by
some indecomposable (projective) summand $P' \in \calo[S^3(\lambda)]$ to
obtain nonzero maps $: (P' \leftrightarrow P) \twoheadrightarrow L(\mu)$.
Then standard arguments involving Fitting's Lemma show that $P, P'$ are
both isomorphic to the projective cover in $\calo[S^3(\lambda)]$ of
$L(\mu)$. Since $P'$ is a summand of $P(\mu, \Theta_\mu +)$ and $P \cong
P'$, it follows by Proposition \ref{Pproj}(5) that $P$ has a standard
filtration.

%\item (This does not use the previous part about $S^2(A)$.)
%First note that since $\lambda \in S^3(A) \subset S^1(A)$,
%$\calo[S^3(\lambda)]$ is finite length by a previous part, and hence
%Proposition \ref{Pcostand} holds.

\item Suppose $\lambda \in S^3(A) \cap \overline{S^2}(A)$. 
%We now repeat the proof in the previous part, this time using that $M \in \calo[S^3(\lambda)]$ is generated by its $H_1$-weight spaces of weights $\lambda_j$, where $S^3(\lambda) = \{ \lambda_1, \dots, \lambda_k \} \subset \hofree$. Now given $\mu \in \hofree$, define $\Theta'_\mu := \bigcup_{j=1}^k [\id_{H_1}, \theta'_j(\mu)]$ where $\theta'_j(\mu) * \mu = \lambda_j$ if such a $\theta'_j(\mu)$ exists, else $\theta'_j(\mu) = \id_{H_1}$.
First note by Proposition \ref{Pproj}(3) and the RTA axioms that in the
set of Verma subquotients in any standard filtration of $P(\mu, \Theta_0
+)$ (for any $\mu \in \hofree$ and finite subset $\Theta_0 \subset
\calq^+_0$), the multiplicity of $M(\mu)$ is always $1$, and any Verma
module with nonzero multiplicity is of the form $M(\theta * \mu)$ for
some $\theta \in \rt_{H_1}(B^+ / B_{\Theta_0 +})$. Hence the same applies
to the projective cover $P(\lambda)$ of $L(\lambda)$ in
$\calo[S^3(\lambda)]$, for all $\lambda \in S^3(A) \cap
\overline{S^2}(A)$.

Now continue the analysis in the previous part and recall that
$\calo[S^3(\lambda)]$ has only finitely many simple objects up to
isomorphism. Thus, standard category-theoretic and homological arguments
using Fitting's Lemma show that the set of indecomposable projectives in
$\calo[S^3(\lambda)]$ is precisely the set of projective covers $P(\mu)$
for $\mu \in S^3(\lambda)$ (and a dual statement holds for injective
hulls as well).
Now define $P := \bigoplus_{\mu \in S^3(\lambda)} P(\mu)^{\oplus n_\mu}$
for any choice of integers $n_\mu > 0$. Since $S^3(\lambda)$ is finite,
$P \in \calo[S^3(\lambda)]$ is a projective generator of the block
$\calo[S^3(\lambda)]$.
Moreover, if $B_\lambda := \End_\calo(P)$, then by standard computations
in \cite[Appendix A]{Don} or \cite{CPS1}, the functor $\hhh_\calo(P, -) =
\hhh_{\calo[S^3(\lambda)]}(P,-)$ is an equivalence from
$\calo[S^3(\lambda)]$ into the category of finitely generated right
$B_\lambda$-modules. That $B_\lambda$ is finite-dimensional follows from
a more general result:
\[ \dim_\F \hhh_\calo(P(\lambda),M) = [M : L(\lambda)], \qquad \forall M
\in \calo[S^1(A)], \ \lambda \in \overline{S^2}(A). \]

\noindent Now the first paragraph in this part implies that each block is
a highest weight category.
\end{enumerate}

\noindent Finally, suppose $A$ satisfies Condition $(Sm)$ for some $1
\leq m \leq 4$. Then the corresponding assertion (numbered $\min(m,3)$)
holds on $\calo[\hofree]$ because $S^m(A) = \hofree$.
\end{proof}

The proof of our second main result (Theorem \ref{Tfunct}) is of a very
different flavor. Before proceeding to this proof, we first write down
some additional facts in order to give the reader a flavor of highest
weight categories. More precisely, we list various desirable properties
for the blocks of Category $\calo$ over regular triangular algebras
satisfying Condition (S3). See \cite{Rin,Don,Kh2} for proofs.

\begin{theorem}\label{Thwc}
Suppose $A$ is an RTA, and $\lambda, \mu \in S^3(A) \cap
\overline{S^2}(A)$. Then $\cala := \calo[S^3(A) \cap \overline{S^2}(A)]$
has the following properties:
\begin{enumerate}
\item (BGG Reciprocity.)
The multiplicity of $M(\mu)$ in any standard filtration of $P(\lambda)$
in $\calo[S^3(\lambda)]$ (or in $\cala$) equals the multiplicity of
$L(\lambda)$ in any Jordan-Holder series for $M(\mu)$.

\item (Neidhardt's theorem.)
If $B^-$ is an integral domain, then every nonzero map of Verma modules
with highest weights in $\hofree$ is an embedding. Moreover,
$\hhh_\cala(M(\mu), M(\lambda)) \neq 0$ if and only if $M(\lambda)$ has a
subquotient $L(\mu)$.

\item $\Ext^n_\cala(L(\lambda), L(\mu)) = 0$ for all $n > 2
|S^3(\lambda)|$. In particular, $\calo[S^3(\lambda)]$ has finite
cohomological (or global) dimension, bounded above by $2 |S^3(\lambda)|$.

\item $\Ext^n_\cala(M,N)$ is finite-dimensional for all $n \geq 0$ and
$M,N \in \cala$.

\item If $X,Y \in \cala$ have standard filtrations, then
\begin{equation}\label{Ecostand}
\dim_\F \Ext^n_\cala(X,F(Y)) =
\begin{cases}
\sum_{\lambda \in S^3(A) \cap \overline{S^2}(A)} [X : M(\lambda)][Y :
M(\lambda)], &\text{if $n=0$;}\\
0,           &\text{otherwise.}
\end{cases}
\end{equation}

\noindent In particular, if $\Ext^n_\cala(M(\lambda), F(M(\mu)))$ is
nonzero, then $n=0$ and $\lambda = \mu$.
\end{enumerate}
\end{theorem}

\noindent Note that some of the assertions hold more generally;
moreover, if Condition (S3) holds, then $S^3(A) \cap \overline{S^2}(A) =
\hofree$ by Lemma \ref{Lcs}.

We end this section with the proof of our second main result.
The proof repeatedly uses the following standard result.

\begin{lemma}\label{Lsummand}
Given a ring $R$, every simple (sub)quotient of a direct sum of
$R$-modules is automatically a simple (sub)quotient of some summand.
\end{lemma}

\begin{proof}[Proof of Theorem \ref{Tfunct}]\hfill
\begin{enumerate}
\item This part involves some (relatively straightforward) bookkeeping.
In particular, define
\[ \calq^+_r = \times_{j=1}^n \calq^+_{rj}, \qquad \Delta :=
\coprod_{j=1}^n \Delta_j, \qquad \Delta' := \coprod_{j=1}^n \Delta'_j,
\qquad B^\pm := \otimes_{j=1}^n B^\pm_j, \qquad H_r = \otimes_{j=1}^n
H_{rj}. \]

\noindent Then $\widehat{H_r} = \times_{j=1}^n \widehat{H_{rj}}$ for
$r=0,1$; moreover, $(B^\pm_j)_{\theta_j} \subset B^\pm_{\id_{H_1}, \dots,
\id_{H_{j-1}}, \theta_j, \id_{H_{j+1}}, \dots, \id_{H_n}}$ for all $1
\leq j \leq n$ and $\theta_j \in \Aut_{\F-alg}(H_j)$.
Conversely if $A$ is based, then defining $\Delta_j := \Delta \cap
\calq^+_{0j}$ for all $j$ shows that $A_j$ is also based. The assertion
about the equivalence of discrete gradings follows from the fact that
$[\id_{H_0}, (\theta_j)_{j=1}^n] = \times_{j=1}^n [\id_{H_{0j}},
\theta_j]$.

\item By Proposition \ref{Pverma}(7), simple modules in $\calo[\hofree]$
are characterized by $\hofree$. Now verify using the previous part that
$\hofree = \times_{j=1}^n \hofree_j$.
Moreover, given $\lambda_j \in \hofree_j$ for all $j$, set $\lambda :=
(\lambda_1, \dots, \lambda_n) \in \hofree$. Then $\otimes_j
L_j(\lambda_j)$ is generated by its one-dimensional $\lambda$-weight
space, which is spanned by a maximal vector. Moreover, it is easily
verified that $\otimes_j L_j(\lambda_j)$ is a simple highest weight
$A$-module in $\calo_A$, whence it is isomorphic to $L(\lambda)$.
Finally, the uniqueness of the $\lambda_j$ (given some $\lambda \in
\hofree$) follows because $L(\lambda)$ is isomorphic to a direct sum of
copies of $L_j(\lambda_j)$ for any fixed $j$, so by Lemma \ref{Lsummand},
$\lambda_j$ is uniquely determined from $L(\lambda)$ as well.

\item We first claim that $S^3_A(\lambda) = \times_{j=1}^n
S^3_{A_j}(\lambda_j)$, where $\lambda_j \in \hofree_j$ for all $j$ and
$\lambda = (\lambda_1, \dots, \lambda_n)$ as above. To do so, first note
that $M(\lambda) = \otimes_{j=1}^n M_j(\lambda_j)$. Now if $M(\lambda_j)$
has a simple subquotient $L_j(\mu_j)$, then there exist submodules $N_j
\subset M_j \subset M(\lambda_j)$ such that $M_j / N_j \cong L_j(\mu_j)$.
But then by the previous part,
\[ (\otimes_{j=1}^n M_j) / N \cong L(\mu) = \otimes_{j=1}^n L_j(\mu_j),
\qquad N := \sum_{j=1}^n \left( N_j \otimes \otimes_{k \neq j} M_k
\right). \]

\noindent Moreover, suppose (exactly) one of $[M(\lambda_j) : L(\mu_j)]$,
$[M(\mu_j), L(\lambda_j)]$ is nonzero for each $j$. Then by the previous
paragraph, the simple $A$-module $\otimes_{j=1}^n L_j(\min(\lambda_j,
\mu_j))$ occurs as a subquotient of both $M(\lambda) = \otimes_{j=1}^n
M_j(\lambda_j)$ and $M(\mu) = \otimes_{j=1}^n M_j(\mu_j)$. From this
analysis it follows that $\times_{j=1}^n S^3_{A_j}(\lambda_j) \subset
S^3_A(\lambda)$.

To prove the reverse inclusion, suppose $[M(\lambda) : L(\mu)] > 0$ in
$\calo[\hofree]$. Fix $1 \leq j \leq n$, and consider both modules over
their restriction to $A_j$. Thus (a direct sum of copies of) $L_j(\mu_j)$
occurs as a subquotient of a direct sum of copies of $M_j(\lambda_j)$. It
follows using Lemma \ref{Lsummand} that $[M_j(\lambda_j) : L_j(\mu_j)] >
0$ for all $j$. Now it is not hard to show that $S^3_A(\lambda) \subset
\times_{j=1}^n S^3_{A_j}(\lambda_j)$.
%Moreover, the $S^3$-sets partition $\hofree$ into equivalence classes
%for any regular triangular algebra $A$. Now the equivalence classes in
%$\times_j \hofree_j$ are precisely Cartesian products of equivalence
%classes. This proves the claim for $S^3_A(\lambda)$.

In turn, applying $\pi_{H_0} = \times_{j=1}^n \pi_{H_{0,j}}$ shows that
$S^2_A(\lambda) = \times_{j=1}^n S^2_{A_j}(\lambda_j)$. Moreover, the
partial order on $\widehat{H_r}^{free}$ holds precisely when it holds in
each component (i.e., $\widehat{H_{rj}}^{free}$). This implies that
$S^1_A(\lambda) = \times_{j=1}^n S^1_{A_j}(\lambda_j)$. Finally, one
verifies that $Z(A) = \otimes_{j=1}^n Z(A_j)$, which implies the
assertion for the $S^4$-sets. The assertion involving $S^m(A)$ now
follows easily. Finally, verify that
\begin{align*}
\overline{S^2}(A) = &\ \{ \lambda = (\lambda_j) \in S^2(A) = \times_j
S^2(A_j) : \pi_{H_0}^{-1}([S^2(\lambda)]_{\leq}) \subset S^1(A) =
\times_j S^1(A_j) \}\\
= &\ \{ \lambda = (\lambda_j) \in \times_j S^2(A_j) : \times_j
\pi_{H_{0j}}^{-1}([S^2(\lambda_j)]_{\leq}) \subset \times_j S^1(A_j) \}\\
= &\ \times_{j=1}^n \{ \lambda_j \in S^2(A_j) :
\pi_{H_{0j}}^{-1}([S^2(\lambda_j)]_{\leq}) \subset S^1(A_j) \} =
\times_{j=1}^n \overline{S^2}(A_j).
\end{align*}

\item The proof is similar to that of \cite[Theorem 15.2]{Kh2} and is
therefore omitted.
\end{enumerate}
\end{proof}
%}}}

\section{Existence results for RTAs: semidirect product constructions,
non-abelian root lattice}\label{Snonabelian}

Having studied the structure of Category $\calo$ over a general RTA,
we turn to the construction and study of examples of RTAs.
We begin by presenting examples of RTAs that are ``completely different"
from all examples considered to date in the literature, in the following
sense: all previously studied RTAs have the property that the ``root
lattice" $\calq^+_0$ is abelian. In fact, with the exception of
stratified Virasoro algebras (see Section \ref{Snonbased}), all
previously studied RTAs are moreover based with a finite set of simple
roots.
Thus a natural question to ask (and whose answer is \textit{a priori} not
yet known) is: do there exist examples of RTAs for which $\calq^+_0$ is
not abelian?
In this section we provide a positive answer to this question, which
further reinforces the generality of our framework.

Before constructing a concrete example of an RTA with non-abelian monoid
$\calq^+_0$ of positive roots, we first introduce the following notation.

\begin{definition}\label{Drtm}
A monoid $Q^+_0$ is said to be a \textit{regular triangular monoid (RTM)}
if it satisfies the following properties:
\begin{enumerate}
\item[(RTM1)] $Q^+_0 \setminus \{ 1_{Q^+_0} \}$ is a semigroup, which
generates a group $\tangle{Q^+_0}$.

\item[(RTM2)] There exists a left-action $\ltimes$ of $Q^+_0$ on
$-Q^+_0$, which fixes $1_{Q^+_0}$ and satisfies the following ``cocycle
conditions":
\begin{align}
\theta_1 \cdot \theta_2^{-1} = &\ (\theta_1 \ltimes \theta_2^{-1}) \cdot
(\theta_2 \ltimes \theta_1^{-1})^{-1}, \label{Ecocycle1}\\
\theta_1 \ltimes (\theta_2^{-1} \cdot \theta_3^{-1}) = &\ (\theta_1
\ltimes \theta_2^{-1}) \cdot ((\theta_2 \ltimes \theta_1^{-1})^{-1}
\ltimes \theta_3^{-1}).\label{Ecocycle2}
\end{align}
%
%\begin{align}
%\theta_1^+ \cdot \theta_2^- = &\ (\theta_1^+ \ltimes \theta_2^-) \cdot
%(\theta_1^+ \rtimes \theta_2^-),\notag\\
%\theta_1^+ \ltimes (\theta_2^- \cdot \theta_3^-) = &\ (\theta_1^+ \ltimes
%\theta_2^-) \cdot ((\theta_1^+ \rtimes \theta_2^-) \ltimes \theta_3^-),\\
%(\theta_1^+ \cdot \theta_2^+) \rtimes \theta_3^- = &\ (\theta_1^+ \rtimes
%(\theta_2^+ \ltimes \theta_3^-)) \cdot (\theta_2^+ \rtimes
%\theta_3^-).\notag
%\end{align}
\end{enumerate}

\noindent Now say that a monoid $M$ acts {\em admissibly} on a regular
triangular monoid $Q^+_0$ if there exists a monoid map $: M \to
\End_{monoid}(\tangle{Q^+_0}) \bigcap \End_{monoid}(Q^+_0)$, such that
\begin{equation}
m(\theta_1 \ltimes \theta_2^{-1}) = m(\theta_1) \ltimes m(\theta_2)^{-1},
\qquad \forall m \in M,\ \theta_1, \theta_2 \in Q^+_0.
\end{equation}
\end{definition}

\begin{remark}
Note that the condition (RTM2) is equivalent to defining a
\textit{matched pairing} of the monoids $Q^+_0$ and $-Q^+_0$, as in
\cite[Section 3]{GM}.
This is because (RTM2) can be reformulated in terms of a right action
$\rtimes$ of $-Q^+_0$ on $Q^+_0$, via similar looking ``cocycle
conditions" as \eqref{Ecocycle1},\eqref{Ecocycle2}. The relationship
between these two actions is: $\theta_1 \rtimes \theta_2^{-1} = (\theta_2
\ltimes \theta_1^{-1})^{-1}$.
Further note that the first cocycle condition \eqref{Ecocycle1} is
unchanged under taking inverses. Moreover as in \cite[Section 3]{GM}, the
matched pairing/RTM structure above shows that $\tangle{Q^+_0} = (-Q^+_0)
\bowtie Q^+_0$, the bicrossed product of the two monoids. However, the
matched pairing is not strong, because the multiplication map $: -Q^+_0
\times Q^+_0 \to \tangle{Q^+_0}$ is not a bijection unless $Q^+_0$ is a
singleton.

For completeness, we also note that matched pairs of monoids were defined
and studied by Gateva-Ivanova and Majid in connection with solutions of
the Yang-Baxter Equation. More generally, an example of the bicrossed
product of two Hopf algebras is the Drinfeld quantum double, which is a
braided Hopf algebra and hence provides solutions of the Yang-Baxter
Equation. There are further connections to Hopf algebras and Lie theory
(see \cite{GM,Ka} and their references), which we do not pursue further
in this paper.
\end{remark}

In this section we prove the following existence theorem for RTAs over
RTMs. The theorem shows that RTAs with certain additional properties
exist over monoids $Q^+_0$, if and only if these are regular triangular
monoids:

\begin{theorem}[RTM-RTA Correspondence]\label{Tmonoid}
Given a regular triangular monoid $Q^+_0$, there exists a strict RTA $A
:= \cala(Q^+_0)$ such that
(a) $\calq^+_0 = Q^+_0$;
(b) $B^+_{\theta_0} \neq 0\ \forall \theta_0 \in \calq^+_0$;
(c) $B^\pm$ do not contain zerodivisors;
(d) the multiplication map $m_A : B^+ \otimes H_0 \otimes B^- \to A$ is
also a vector space isomorphism; and
(e) for each $\theta_1, \theta_2 \in \calq^+_0$, the image of
$m_A(B^+_{\theta_1} \otimes \F 1_{H_0} \otimes B^-_{\theta_2^{-1}})$ is
contained in $B^-_{\theta^-} \otimes H_0 \otimes B^+_{\theta^+}$ for
unique $\theta^\pm \in \pm \calq^+_0$.

Conversely, if $A$ is a strict RTA that satisfies (b)--(e), then
$\calq^+_0$ is a regular triangular monoid.
\end{theorem}

\begin{remark}
Thus the notion of an RTM is intimately related to the notion of an RTA.
In fact, RTMs provide a natural answer (via Theorem \ref{Tmonoid}) to the
question: ``Given a (sufficiently nice) RTA, what structure can one
expect from its underlying semigroup $\calq^+_0$?"
Moreover, the existence result in Theorem \ref{Tmonoid} shows that the
RTM/RTA frameworks are not ``unnecessarily broad".
\end{remark}

We prove Theorem \ref{Tmonoid} below; for now, we observe that Theorem
\ref{Tmonoid} immediately shows the existence of RTAs over arbitrary
abelian monoids $Q^+_0$:

\begin{cor}\label{Cabelian}
Suppose $Q^+_0$ is a commutative monoid such that $Q^+_0 \setminus \{
1_{Q^+_0} \}$ is a semigroup. Then there exists a strict RTA $A =
\cala(Q^+_0)$ such that $\calq^+_0 = Q^+_0$.
\end{cor}

\begin{proof}
Note that every commutative monoid $Q^+_0$ satisfying (RTM1) is a regular
triangular monoid with $\theta_1 \ltimes \theta_2^{-1} := \theta_2^{-1}$
for all $\theta_1, \theta_2 \in Q^+_0$. Now apply Theorem \ref{Tmonoid}.
\end{proof}

This section is organized as follows. In Section \ref{Srtmrta} we prove
Theorem \ref{Tmonoid}. In Section \ref{Segrtm}, we then provide several
recipes that yield RTMs (which in turn lead to RTA constructions).
Finally, in Section \ref{Sexample} we carry out the aforementioned
construction of an RTA $A$ with non-abelian span of roots, and study its
associated Category $\calo$.

%{{{1 Section 4.1 - Existence theorem for RTAs over regular triangular monoids
\subsection{Existence theorem for RTAs over regular triangular
monoids}\label{Srtmrta}

This subsection is devoted to proving Theorem \ref{Tmonoid}. We begin by
showing a more general result.

\begin{theorem}[RTA Existence Theorem]\label{Trtm}
Suppose $Q^+_0$ is a regular triangular monoid, and $\Z^k$ acts
admissibly on $\tangle{Q^+_0}$ for some $k \in \nn$. Then for all ${\bf
c} \in \F^k$, there exists a strict RTA $A = \cala(Q^+_0, {\bf c})$ such
that:
(i) $\calq^+_0$ equals the regular triangular monoid $(\nn)^k \ltimes
Q^+_0$;
(ii) $\cala(Q^+_0, {\bf c})$ satisfies properties (b)--(d) in Theorem
\ref{Tmonoid}, as well as property (e) for all $\theta_1, \theta_2 \in
Q^+_0$; and
(iii) there exist nonzero elements $\{ x^\pm_r : 1 \leq r \leq k \}$ in
$\cala(Q^+_0, {\bf c})$ such that
$[x^+_q, x^+_r] = [x^-_q, x^-_r] = 0$ and
$[x^+_q, x^-_r] = \delta_{q,r} c_r$ for all $1 \leq q,r \leq k$.
\end{theorem}

\begin{proof}
Suppose $\{ \vi_j : 1 \leq j \leq k \}$ denotes the standard $\Z$-basis
of $\Z^k$. Define the $\F$-algebra $\cala(Q^+_0, {\bf c})$ to be
generated by the algebra $H_0 := \F[\Z^k \ltimes \tangle{Q^+_0}] = (\F
(\Z^k \ltimes \tangle{Q^+_0}))^*$ of $\F$-valued functions on the group
$\Z^k \ltimes \tangle{Q^+_0}$ (with coordinatewise addition and
multiplication), together with elements $\{ t^{\theta_0^{\pm 1}} :
\theta_0 \in Q^+_0 \}$ and $\{ x_j^\pm : 1 \leq j \leq k \}$, modulo the
following relations:
\begin{align}\label{Ertm1}
&\ t^{\theta_1^{\pm 1}} t^{\theta_2^{\pm 1}} = t^{\theta_1^{\pm 1} \cdot
\theta_2^{\pm 1}}, \qquad t^{1_{Q^+_0}} = 1, \qquad
t^{\theta_1} t^{\theta_2^{-1}} = t^{\theta_1 \ltimes \theta_2^{-1}}
t^{(\theta_2 \ltimes \theta_1^{-1})^{-1}},\notag\\
&\ t^{\theta_1^{\pm 1}} f(-) = f((0,\theta_1^{\mp 1}) \cdot -)
t^{\theta_1^{\pm 1}}, \qquad 
x_j^\pm f(-) = f((\mp \vi_j,0) \cdot -) x_j^\pm,\\
&\ [x_j^\pm, x_{j'}^+] = [x_j^\pm, x_{j'}^-] = 0, \qquad [x_j^+, x_j^-] =
c_j, \qquad x_j^\pm t^{\theta_0} = t^{\vi_j^{\pm 1}(\theta_0)}
x_j^\pm,\notag
\end{align}

\noindent for all $1 \leq j \neq j' \leq k$, $\theta_1, \theta_2 \in
Q^+_0$, $\theta_0 \in \tangle{Q^+_0}$, and $f : \Z^k \ltimes
\tangle{Q^+_0} \to \F$ in $H_0$.

We now claim that $\cala(Q^+_0, {\bf c})$ is a strict RTA satisfying the
aforementioned properties. The meat of the proof lies in showing that the
algebra $\cala(Q^+_0, {\bf c})$ satisfies (RTA1). More precisely, we
\textbf{claim} that the following holds:\medskip

\textit{The $\F$-algebra $\cala(Q^+_0, {\bf c})$ satisfies (RTA1), with
$B^-$ having an $\F$-basis of the form}
\begin{equation}\label{Ertm2}
X_{irr}^- := \{ t^{\theta_0^{-1}} \prod_{j=1}^k (x_{k+1-j}^-)^{n_j} : n_j
\in \nn, \theta_0 \in Q^+_0 \} \quad (\text{with } t^{1_{Q^+_0}} = 1),
\end{equation}

\textit{and $B^+$ having an $\F$-basis $X_{irr}^+ := \{ \prod_{j=1}^k
(x_j^+)^{n_j} \cdot t^{\theta_0} : n_j \in \nn, \theta_0 \in Q^+_0
\}$.}\medskip

\noindent We show the claim after proving the remaining assertions. Given
the above claim, the decomposition in (RTA2) also holds, with $\calq^+_0
= (\nn)^k \ltimes Q^+_0$. Moreover, every weight space of $B^+$ is
one-dimensional, of the form $B^+_{({\bf n},\theta_0)} = \F \prod_{j=1}^k
(x_j^+)^{n_j} t^{\theta_0}$. This includes the space $B^+_{\id_{H_0}} =
B^+_{({\bf 0}, 1_{Q^+_0})}$. Next, the symmetric form of the algebra
relations \eqref{Ertm1} implies that (RTA3) also holds with the
anti-involution sending $t^{\theta_0}, x_j^+$ to $t^{\theta_0^{-1}},
x_j^-$ respectively, for all $\theta_0 \in Q^+_0$ and $1 \leq j \leq k$.
(In fact, (RTA3) can be verified at the very outset, even before/without
verifying \eqref{Ertm2}.)
Now property (i) from the statement is clear from the algebra relations
and the above claim, except for the fact that $(\nn)^k \ltimes Q^+_0$ is
an RTM. This last fact is proved more generally in Theorem
\ref{Tsemidirect}(1),(4) below. Property (iii) is immediate from
Equations \eqref{Ertm1},\eqref{Ertm2}. To show property (ii), the algebra
relations \eqref{Ertm1} yield:
\[ m_A(B^+_{\theta_1} \otimes H_0 \otimes B^-_{\theta_2^{-1}}) \subset
B^-_{\theta_1 \ltimes \theta_2^{-1}} \otimes H_0 \otimes B^+_{(\theta_2
\ltimes \theta_1^{-1})^{-1}}, \qquad \forall \theta_1, \theta_2 \in
Q^+_0. \]

\noindent Next, that $B^\pm$ do not contain zerodivisors follows since
$B^\pm$ are $H_0$-root-semisimple, each root space is one-dimensional,
and by Equation \eqref{Ertm1} and (RTA1) there exist no root vectors that
are zerodivisors. Finally, that the multiplication map $: B^+ \otimes H_0
\otimes B^- \to \cala(Q^+_0, {\bf c})$ is a vector space isomorphism is
proved similarly as the proof (below) of the claim \eqref{Ertm2}.

It remains to show that (RTA1) holds, or more precisely, the claim
\eqref{Ertm2} is true. For this we use the Diamond Lemma from \cite{Be}. 
More precisely, Bergman has shown a variant for rings of the following
result from graph theory: if a directed graph satisfies (a) the
\textit{descending chain condition} (every directed path has finite
length) and (b) the \textit{diamond condition} (any two distinct directed
edges with common source extend to directed paths with common target),
then every connected graph component has a unique ``maximal" vertex. We
now apply Bergman's result to $\cala(Q^+_0,{\bf c})$, to prove the above
claim \eqref{Ertm2} as follows:
\begin{itemize}
\item Let $\{ h_i : i \in I \}$ be a fixed $\F$-basis of $H_0$ with $h_0
= 1_{H_0}$ for a fixed element $0 \in I$. Then a set of generators of
$\cala(Q^+_0, {\bf c})$ is given by:
\begin{equation}
X := \{ x_j^\pm : 1 \leq j \leq k \} \coprod \{ t^{\theta_0} : \theta_0
\in Q^+_0 \setminus \{ 1_{Q^+_0} \} \} \coprod \{ h_i : i \in I \}.
\end{equation}

\noindent Then one has relations $h_q h_r = \sum_{s \in I} c^s_{q,r} h_s$
that encode the multiplication in $H_0$, as well as other relations that
we rewrite for reasons explained below:
\begin{align}\label{Ertm3}
&\ t^{\theta_1^{\pm 1}} t^{\theta_2^{\pm 1}} = t^{\theta_1^{\pm 1} \cdot
\theta_2^{\pm 1}}, \qquad t^{\theta_1} t^{\theta_2^{-1}} = t^{\theta_1
\ltimes \theta_2^{-1}} t^{(\theta_2 \ltimes \theta_1^{-1})^{-1}},\notag\\
&\ t^{\theta_1} h_i(-) = ((0,\theta_1)(h_i))(-) \cdot t^{\theta_1},
\qquad h_i(-) t^{\theta_1^{-1}} = t^{\theta_1^{-1}} \cdot
((0,\theta_1)(h_i))(-),\notag\\
&\ x_j^+ h_i(-) = ((\vi_j, 1_{Q^+_0})(h_i))(-) \cdot x_j^+, \qquad
h_i(-) x_j^- = x_j^- \cdot ((\vi_j, 1_{Q^+_0})(h_i))(-),\\
&\ x_{j_1}^+ x_{j_2}^- = x_{j_2}^- x_{j_1}^+, \qquad 
x_j^+ x_j^- = x_j^- x_j^+ + c_j, \qquad
x_j^- x_{j'}^- = x_{j'}^- x_j^-, \qquad
x_{j'}^+ x_j^+ = x_j^+ x_{j'}^+, \notag\\
&\ t^{\theta_1} x_j^\pm = x_j^\pm t^{\mp \vi_j(\theta_1)}, \qquad x_j^\pm
t^{\theta_1^{-1}} = t^{\pm \vi_j(\theta_1^{-1})} x_j^\pm, \qquad h_0 x =
x h_0 = x,\notag
\end{align}

\noindent for all $i \in I$, $1 \leq j_1 \neq j_2 \leq k$, $1 \leq j < j'
\leq k$, $\theta_1, \theta_2 \in Q^+_0 \setminus \{ 1_{Q^+_0} \}$, and $x
\in X$. (Note that some of these relations are obtained from the
presentation \eqref{Ertm1} of $\cala(Q^+_0, {\bf c})$, via the fact that
$H_0$ is a contragredient representation of the group $\Z^k \ltimes
\tangle{Q^+_0}$.) We label the set of all these relations by the index
set $\Sigma_X$.\smallskip

\item The next ingredient is to define a total order on $\tangle{X}$. To
do so, first define and fix a total order $\prec$ on the set of
generators $X$ as follows. Use the Axiom of Choice (more specifically,
the Ultrafilter Theorem -- or equivalently, the Compactness Theorem for
first-order logic), to construct a total order $\prec$ on $Q^+_0$, which
extends the partial order $\theta_0 \prec \theta'_0 \cdot \theta_0\
\forall \theta_0, \theta'_0 \in Q^+_0$. Also fix a basis $\{ h_i : i \in
I \}$ of $H_0$ that includes $h_0 := 1_{H_0}$ (with $0 \in I$). Next,
define a total ordering on $I$ such that $0 < i\ \forall i \in I$. Now
use the following total order on $X$:
\begin{align}
h_0 \prec \quad t^{\theta_0^{-1}}\ (\downarrow \theta_0 \in Q^+_0
\setminus \{ 1_{Q^+_0} \}) \quad \prec &\ x_k^- \prec \cdots \prec x_1^-
\prec \quad h_i\ (\uparrow i \in I \setminus \{ 0 \})\\
\prec &\ x_1^+ \prec \cdots \prec x_k^+ \prec \quad t^{\theta_0}\
(\uparrow \theta_0 \in Q^+_0 \setminus \{ 1_{Q^+_0} \}),\notag
\end{align}

\noindent where
$t^{\theta_0}\ (\uparrow \theta_0 \in Q^+_0 \setminus \{ 1_{Q^+_0} \})$
means that $\{ t^{\theta_0} : \theta_0 \in Q^+_0 \setminus \{ 1_{Q^+_0}
\} \}$ inherits the total order from $Q^+_0$ via the map $\theta_0
\mapsto t^{\theta_0}$, and similarly in the other cases (where
$\downarrow$ indicates order-reversing).

This order then extends to a semigroup partial order on the monoid
$\tangle{X}$ generated by $X$ (which is a basis for the tensor algebra on
the $\F$-span of $X$), which satisfies: if $x > x'$ in $X$ then $A x B >
A x' B$ for all $A,B \in \tangle{X}$. To do so, we in fact write down a
\textit{total order} on $\tangle{X}$ as follows: compare two words by
setting larger words to be greater, and via the lexicographic order
induced by $\prec$ on two words of equal lengths.\smallskip

\item In place of directed edges in the graph-theoretic version of the
Diamond Lemma, Bergman uses the algebra relations to work with
\textit{reductions}. Namely, we need to verify that every relation
discussed above can be written as
$w_\sigma \to f_\sigma\ \forall \sigma \in \Sigma_X$,
with $w_\sigma \in \tangle{X}, f_\sigma \in \F \tangle{X} =
T_\F({\rm span}_\F X)$, and such that every monomial in $f_\sigma$ is
strictly less than $w_\sigma$ in the semigroup partial (in fact total)
order above. It is easy to verify that this procedure applies to every
relation in \eqref{Ertm3} by replacing the equality by $\to$.\smallskip

\item The previous step verifies that the semigroup partial order on
$\tangle{X}$ is compatible with the reductions. This provides us with a
directed path structure on the graph whose nodes are the monomials in
$\tangle{X}$. In this structure, ``maximal" vertices are those from which
no directed path starts, i.e., monomials that are left unchanged by every
reduction. In other words, maximal vertices are precisely the
``irreducible" monomials, given by $\{ x^- \cdot h_i \cdot x^+ : x^\pm
\in X_{irr}^\pm, i \in I \}$ (via Equation \eqref{Ertm2}).

\item It remains to check that the descending chain condition (DCC) and
the diamond condition are satisfied in our setting. A standard tool used
to verify the DCC is a \textit{misordering index} $\mis(w_\sigma)$, where
$\mis : \tangle{X} \to \nn$ is zero on all irreducible words, and we show
that $w \succ w'$ in $\tangle{X}$ implies $\mis(w) > \mis(w')$. Thus each
reduction reduces the $\mis(\cdot)$-value, proving the DCC.

In our setting, define $\mis : \tangle{X} \to \nn$ via: $\mis(s_1 \cdots
s_n) := T'_+ + T'_- + H' + N'$, where $T'_\pm, H'$ denote respectively
the number of letters $s_j$ of type $t^{\theta_0^{\pm 1}}, h_i$ for
$\theta_0 \in Q^+_0 \setminus \{ 1_{Q^+_0} \}$ and $i \in I$; and $N$
denotes the number of pairs $(j,j')$ such that: (a) $1 \leq j < j' \leq
n$ and (b) $s_j \succ s_{j'}$ with not both $s_j,s_{j'}$ of the same type
($T'_+, T'_-$, or $H'$). We claim that every reduction of $w \in
\tangle{X}$ strictly decreases $\mis(w)$. This is not hard to verify
using the presentation of the algebra $\cala(Q^+_0, {\bf c})$ given in
\eqref{Ertm3}.\smallskip

\item Finally, we verify the diamond condition. By \cite{Be}, one only
needs to work with directed paths of reductions starting from monomials;
and one only needs to resolve ``minimal ambiguities". Moreover, there are
no ``inclusion ambiguities" in our setting (i.e., for no $\sigma,
\sigma'$ is it true that $w_\sigma \in \tangle{X} w_{\sigma'}
\tangle{X}$). Thus it suffices to show that the diamond condition holds
for all overlap ambiguities $ABC$, where $AB = w_\sigma$ and $BC =
w_{\sigma'}$ for some $\sigma, \sigma' \in \Sigma_X$. In the present case
this involves computations with words of length precisely $3$, in the
$t^{\theta_0^{\pm 1}}, x_j^\pm, h_i$.
Some of these computations are straightforward using the algebra
relations and so we do not write them all down; for instance, overlap
ambiguities involving all three alphabets being of the same ``type" --
$t$, or $x$, or $h$ -- are trivially resolved using the algebra relations
and the cocycle conditions \eqref{Ecocycle1},\eqref{Ecocycle2}. The other
overlap ambiguities are also not hard to work out; for illustrative
purposes we carry out a few of the verifications in the equations
\eqref{Ertm4}, using the RTM-axioms and the hypotheses of the theorem. We
will also use $0$ instead of ${\bf 0}$ for the zero element in $\Z^k$.

{\allowdisplaybreaks
\begin{align}\label{Ertm4}
(\theta_1 > 1_{Q^+_0}) & \quad (t^{\theta_1} x_j^+) h_i(-) \to \quad
x_j^+ (t^{(-\vi_j)(\theta_1)} h_i(-)) \to (x_j^+
(0,(-\vi_j)(\theta_1))(h_i)(-)) t^{(-\vi_j)(\theta_1)}\notag\\
& \qquad \to ((\vi_j, 1_{Q^+_0}) (0, (-\vi_j)(\theta_1)) h_i)(-) \cdot
x_j^+ t^{(-\vi_j)(\theta_1)},\notag\\
& \quad t^{\theta_1} (x_j^+ h_i(-)) \to \quad (t^{\theta_1} ((\vi_j,
1_{Q^+_0}) h_i)(-)) x_j^+ \to ((0,\theta_1) (\vi_j, 1_{Q^+_0}) h_i)(-)
\cdot (t^{\theta_1} x_j^+)\notag\\
& \qquad \to ((0,\theta_1) (\vi_j, 1_{Q^+_0}) h_i)(-) \cdot x_j^+
t^{(-\vi_j)(\theta_1)}.\notag\\
(\theta_1 > 1_{Q^+_0}) & \quad (t^{\theta_1} h_i(-)) x_j^- \to \quad
((0,\theta_1) h_i)(-) \cdot (t^{\theta_1} x_j^-) \to (((0,\theta_1)
h_i)(-) x_j^-) t^{\vi_j(\theta_1)} \notag\\
& \qquad \to x_j^- \cdot ((\vi_j, 1_{Q^+_0}) (0, \theta_1) h_i)(-) \cdot
t^{\vi_j(\theta_1)},\\
& \quad t^{\theta_1} (h_i(-) x_j^-) \to (t^{\theta_1} x_j^-) ((\vi_j,
1_{Q^+_0})h_i)(-) \to x_j^- (t^{\vi_j(\theta_1)} ((\vi_j, 1_{Q^+_0})
h_i)(-))\notag\\
& \qquad \to x_j^- \cdot ((0,\vi_j(\theta_1))(\vi_j, 1_{Q^+_0}) h_i)(-)
\cdot t^{\vi_j(\theta_1)}.\notag\\
(\theta_l > 1_{Q^+_0}) & \quad (t^{\theta_1} h_i(-)) t^{\theta_2^{-1}}
\to \quad ((0,\theta_1) h_i)(-) (t^{\theta_1} t^{\theta_2^{-1}}) \to
(((0,\theta_1) h_i)(-) t^{\theta_1 \ltimes \theta_2^{-1}}) t^{(\theta_2
\ltimes \theta_1^{-1})^{-1}}\notag\\
& \qquad \to t^{\theta_1 \ltimes \theta_2^{-1}} \cdot ((0,\theta_1
\ltimes \theta_2^{-1})^{-1} (0,\theta_1) h_i)(-) \cdot t^{(\theta_2
\ltimes \theta_1^{-1})^{-1}}, \notag\\
& \quad t^{\theta_1} (h_i(-) t^{\theta_2^{-1}}) \to \quad (t^{\theta_1}
t^{\theta_2^{-1}}) ((0,\theta_2) h_i)(-) \to t^{\theta_1 \ltimes
\theta_2^{-1}} (t^{(\theta_2 \ltimes \theta_1^{-1})^{-1}} ((0,\theta_2)
h_i)(-))\notag\\
& \qquad \to t^{\theta_1 \ltimes \theta_2^{-1}} \cdot ((0, (\theta_2
\ltimes \theta_1^{-1})^{-1}) (0,\theta_2) h_i)(-) \cdot t^{(\theta_2
\ltimes \theta_1^{-1})^{-1}}. \notag\\
(\theta_l > 1_{Q^+_0}) & \quad (t^{\theta_1} t^{\theta_2}) x_j^\pm
\to t^{\theta_1 \theta_2} x_j^\pm \to x_j^\pm t^{(\mp \vi_j)(\theta_1
\theta_2)}, \notag\\
& \quad t^{\theta_1} (t^{\theta_2} x_j^\pm) \to (t^{\theta_1} x_j^\pm)
t^{(\mp \vi_j)(\theta_2)} \to x_j^\pm (t^{(\mp \vi_j)(\theta_1)} t^{(\mp
\vi_j)(\theta_2)}) \to x_j^\pm t^{(\mp \vi_j)(\theta_1
\theta_2)}.\notag\\
(\theta_l > 1_{Q^+_0}) & \quad (t^{\theta_1} x_j^\pm) t^{\theta_2^{-1}}
\to x_j^\pm (t^{(\mp \vi_j)(\theta_1)} t^{\theta_2^{-1}}) \to (x_j^\pm
t^{(\mp \vi_j)(\theta_1) \ltimes \theta_2^{-1}}) t^{(\theta_2 \ltimes
(\mp \vi_j)(\theta_1^{-1}))^{-1}} \notag\\
& \qquad \to t^{\theta_1 \ltimes (\pm \vi_j)(\theta_2^{-1})} x_j^\pm
t^{(\theta_2 \ltimes (\mp \vi_j)(\theta_1^{-1}))^{-1}}, \notag\\
& \quad t^{\theta_1} (x_j^\pm t^{\theta_2^{-1}}) \to (t^{\theta_1}
t^{(\pm \vi_j)(\theta_2^{-1})}) x_j^\pm \to t^{\theta_1 \ltimes (\pm
\vi_j)(\theta_2^{-1})} \cdot (t^{(((\pm \vi_j)(\theta_2^{-1})^{-1})
\ltimes \theta_1^{-1})^{-1}} x_j^\pm)\notag\\
& \qquad = t^{\theta_1 \ltimes (\pm \vi_j)(\theta_2^{-1})} (t^{((\pm
\vi_j)(\theta_2) \ltimes \theta_1^{-1})^{-1}} x_j^\pm) \to t^{\theta_1
\ltimes (\pm \vi_j)(\theta_2^{-1})} x_j^\pm t^{(\mp \vi_j)((\pm
\vi_j)(\theta_2) \ltimes \theta_1^{-1})^{-1}}\notag\\
& \qquad = t^{\theta_1 \ltimes (\pm \vi_j)(\theta_2^{-1})} x_j^\pm
t^{(\theta_2 \ltimes (\mp \vi_j)(\theta_1^{-1}))^{-1}}. \notag\\
(\theta_1 > 1_{Q^+_0}) & \quad (t^{\theta_1} x_j^+) x_l^- \to x_j^+
(t^{(-\vi_j)(\theta_1)} x_l^-) \to (x_j^+ x_l^-) t^{(\vi_l -
\vi_j)(\theta_1)} \to (x_l^- x_j^+ + \delta_{jl} c_j) t^{(\vi_l -
\vi_j)(\theta_1)}, \notag\\
& \quad t^{\theta_1} (x_j^+ x_l^-) \to (t^{\theta_1} x_l^-) x_j^+ +
\delta_{jl} c_j t^{\theta_1} \to x_l^- (t^{\vi_l(\theta_1)} x_j^+) +
\delta_{jl} c_j t^{\theta_1}\notag\\
& \qquad \to x_l^- x_j^+ t^{(\vi_l - \vi_j)(\theta_1)} + \delta_{jl} c_j
t^{\theta_1}.\notag\\
(i \in I) & \quad (x_j^+ h_i(-)) x_l^- \to \quad ((\vi_j, 1_{Q^+_0})
h_i)(-) (x_j^+ x_l^-) \to ((\vi_j, 1_{Q^+_0}) h_i)(-) (x_l^- x_j^+ +
\delta_{jl} c_j)\notag\\
& \qquad \to x_l^- \cdot ((\vi_j + \vi_l, 1_{Q^+_0}) h_i)(-) \cdot x_j^+
+ \delta_{jl} c_j \cdot ((\vi_j, 1_{Q^+_0}) h_i)(-),\notag\\
& \quad x_j^+ (h_i(-) x_l^-) \to \quad (x_j^+ x_l^-) ((\vi_l, 1_{Q^+_0})
h_i)(-) \to (x_l^- x_j^+ + \delta_{jl} c_j) ((\vi_l, 1_{Q^+_0})
h_i)(-)\notag\\
& \qquad \to x_l^- \cdot ((\vi_j + \vi_l, 1_{Q^+_0}) h_i)(-) \cdot x_j^+
+ \delta_{jl} c_j \cdot ((\vi_l, 1_{Q^+_0}) h_i)(-).\notag\\
(j > l) & \quad (x_j^+ x_l^+) x_i^- \to x_l^+ (x_j^+ x_i^-) \to x_l^+
(x_i^- x_j^+ + \delta_{ij} c_j) \to x_i^- x_l^+ x_j^+ + \delta_{il} c_l
x_j^+ + \delta_{ij} c_j x_l^+, \notag\\
& \quad x_j^+ (x_l^+ x_i^-) \to \quad x_j^+ (x_i^- x_l^+ + \delta_{il}
c_l) \to x_i^- (x_j^+ x_l^+) + \delta_{ij} c_j x_l^+ + \delta_{il} c_l
x_j^+. \notag\\
(\theta_1 > 1_{Q^+_0}) & \quad (t^{\theta_1} h_q(-)) h_r(-) \to
((0,\theta_1)h_q)(-) (t^{\theta_1} h_r(-)) \to ((0,\theta_1)h_q)(-) \cdot
((0,\theta_1)h_r)(-) t^{\theta_1},\notag\\
& \quad t^{\theta_1} (h_q(-) h_r(-)) \to \quad \sum_s c_{q,r}^s
t^{\theta_1} h_s(-) \to \sum_s c_{q,r}^s ((0,\theta_1)h_s)(-) \cdot
t^{\theta_1}.\notag
\end{align}}
\end{itemize}

\noindent Both computations in the last reduction yield the same quantity
because $H_0$ is a module-algebra over the group $\Z^k \ltimes
\tangle{Q^+_0}$ (via its contragredient representation), and this imposes
a compatibility constraint on the structure constants for $H_0$. Further
note that several overlap ambiguities that are not listed in
\eqref{Ertm4} can be resolved without any further computation, by
applying the anti-involution $i : \cala(Q^+_0, {\bf c}) \to \cala(Q^+_0,
{\bf c})$ from (RTA3). For instance, the first overlap ambiguity resolved
in \eqref{Ertm4} implies that the overlap ambiguity $h_i(-) x_j^-
t^{\theta_1^{-1}}$ can also be resolved for $\theta_1 \in Q^+_0 \setminus
\{ 1_{Q^+_0} \}$.
\end{proof}

\begin{remark}
Our construction of the algebra $\cala(Q^+_0, {\bf c})$ can be thought as
a generalization of continuous Cherednik algebras (see \cite{EGG}).
Note that the algebras $B^\pm$ are ``dual" to one another in some sense;
however, they need not be polynomial algebras as in \cite{EGG}.
\end{remark}

\begin{remark}
In light of the algebra relations \eqref{Ertm1} in $\cala(Q^+_0, {\bf
c})$, the first cocycle condition \eqref{Ecocycle1} can be thought of as
a group/monoid version of a so-called ``straightening identity" in the
flavor of Garland \cite{Gar}, Beck-Chari-Pressley \cite{BCP}, and several
other works in the literature. For more on the subject, see \cite{BC} and
the references therein.
\end{remark}

Equipped with the above theorem, we now prove our initial result in this
section.

\begin{proof}[Proof of Theorem \ref{Tmonoid}]
The meat of the proof lies in proving the existence of an RTA
$\cala(Q^+_0)$ over an arbitrary RTM $Q^+_0$; but this is the special
case of Theorem \ref{Trtm} where $k=0$.
Conversely, suppose there exists a strict RTA $A$ satisfying properties
(b)--(e). Define $\theta_1 \ltimes \theta_2^{-1}, \theta_1 \rtimes
\theta_2^{-1}$ via:
\[ m_A (B^+_{\theta_1} \otimes \F \cdot 1_{H_0} \otimes
B^-_{\theta_2^{-1}}) \subset B^-_{\theta_1 \ltimes \theta_2^{-1}} \otimes
H_0 \otimes B^+_{\theta_1 \rtimes \theta_2^{-1}}. \]

\noindent Note here that both sides are nonzero subspaces of $A$.
It follows by considering their $H_0$-roots that $\ltimes$ is an action
of $\calq^+_0$ on $-\calq^+_0$, and that $\theta_1 \cdot \theta_2^{-1} =
(\theta_1 \ltimes \theta_2^{-1} ) \cdot (\theta_1 \rtimes
\theta_2^{-1})$. Now applying the anti-involution $i$ to the above
subspace (via Lemma \ref{Lfirst}) and again considering the $H_0$-roots,
we obtain (via a slight abuse of notation):
\[ (\theta_1 \rtimes \theta_2^{-1})^{-1} = i(\theta_1 \rtimes
\theta_2^{-1}) = i(\theta_2^{-1}) \ltimes i(\theta_1) = \theta_2 \ltimes
\theta_1^{-1}, \qquad \forall \theta_1, \theta_2 \in \calq^+_0. \]

\noindent This shows the first cocycle condition \eqref{Ecocycle1} for
$\calq^+_0$. Next, consider $\theta_1, \theta_2, \theta_3 \in \calq^+_0$,
and compute $m_A(B^+_{\theta_1} \otimes B^-_{\theta_2^{-1}} \otimes
B^-_{\theta_3^{-1}})$ in two ways by using the associativity of $m_A$ and
properties (b)--(e). This is an easy computation that yields the second
cocycle condition \eqref{Ecocycle2} for $\calq^+_0$.
\end{proof}

We end this subsection with two further results on the algebras
$\cala(Q^+_0)$. The first discusses Casimir operators.

\begin{prop}\label{Prtmo1}
Fix a regular triangular monoid $Q^+_0$ and a subset $\calq^- \subset
-Q^+_0$ such that $\theta_1 \ltimes -$ is a bijection on $\calq^-$ for
all $\theta_1 \in Q^+_0$.
\begin{enumerate}
\item Then a suitable completion of the algebra $\cala(Q^+_0)$ contains a
central ``Casimir" operator $\Omega(\calq^-) := \sum_{\theta_0 \in
\calq^-} t^{\theta_0} t^{\theta_0^{-1}}$.

\item The operators $\Omega(\calq^-)$ act on all objects in
$\calo[\hzfree]$. They kill every highest weight module in $\calo$, hence
act nilpotently on $\calo[\hzfree]$.
\end{enumerate}
\end{prop}

\noindent Examples of subsets $\calq^-$ include any subset of $-Q^+_0$
when $Q^+_0$ is abelian; as well as $\calq^- = \{ 1_{Q^+_0} \}$, which
corresponds to $\Omega(\calq^-) = 1$.

\begin{proof}
For the first part, observe using the first cocycle condition
\eqref{Ecocycle1} that $(\theta_2 \ltimes \theta_1^{-1})^{-1} \cdot
\theta_2 = (\theta_1 \ltimes \theta_2^{-1})^{-1} \cdot \theta_1$, for all
$\theta_1, \theta_2 \in Q^+_0$. Now fix $\theta_2 \in -\calq^-$ and
compute using the algebra relations:
\[ t^{\theta_1} \cdot t^{\theta_2^{-1}} t^{\theta_2} = t^{\theta_1
\ltimes \theta_2^{-1}} t^{(\theta_2 \ltimes \theta_1^{-1})^{-1}}
t^{\theta_2} = t^{\theta_1 \ltimes \theta_2^{-1}} t^{(\theta_1 \ltimes
\theta_2^{-1})^{-1}} t^{\theta_1}. \]

\noindent By the assumptions on $\calq^-$, it follows that $t^{\theta_1}$
commutes with $\Omega(\calq^-)$ for all $\theta_1 \in Q^+_0$. In turn,
this implies (using the anti-involution $i : t^{\theta_0}
\leftrightarrow t^{\theta_0^{-1}}$ on $\cala(Q^+_0)$) that
$\Omega(\calq^-)$ is central.

To prove the second part, first observe as in the Kac-Moody setting, that
the ``Casimir" operator $\Omega(\calq^-)$ acts on arbitrary objects of
Category $\calo[\hzfree]$. Moreover, $\Omega(\calq^-)$ kills the highest
vector in any highest weight module in the respective Categories $\calo$,
since the Harish-Chandra projection to $H_0$ kills all such operators. It
follows that these central elements annihilate the entire module.
The final assertion now follows from Proposition \ref{Pproj}(4).
\end{proof}

We also discuss the Conditions (S) for the algebras $\cala(Q^+_0)$.

\begin{prop}\label{Prtmo2}
Suppose $Q^+_0$ is a nontrivial regular triangular monoid, whose action
$\ltimes$ on $-Q^+_0$ stabilizes $-Q^+_0 \setminus \{ 1_{Q^+_0} \}$.
Then the algebra $\cala(Q^+_0)$ satisfies none of the Conditions (S),
because $S^3(\lambda) = \tangle{Q^+_0} * \lambda\ \forall \lambda \in
\hzfree$, so that $\dim L(\lambda) = 1$.
\end{prop}

\noindent Note that $\hzfree$ is nonempty because $\tangle{Q^+_0} \subset
\hzfree$.

\begin{proof}
Given $\lambda \in \hzfree$, we first claim that every nonzero weight
vector $t^{\theta_2^{-1}} m_\lambda$ of the Verma module $M(\lambda)$ is
maximal (for $\theta_2 \in Q^+_0$). To show the claim, compute using the
assumptions and the algebra relations \eqref{Ertm1}, for $\theta_1 \in
Q^+_0 \setminus \{ 1_{Q^+_0} \}$:
\[ t^{\theta_1} \cdot t^{\theta_2^{-1}} m_\lambda = t^{\theta_1 \ltimes
\theta_2^{-1}} t^{(\theta_2 \ltimes \theta_1^{-1})^{-1}} m_\lambda \in
t^{\theta_1 \ltimes \theta_2^{-1}} \cdot N^+ m_\lambda = 0. \]

\noindent This proves the claim. In particular, $N^- m_\lambda \subset
M(\lambda)$ is a codimension $1$ submodule, whence $\F \cong M(\lambda) /
N^- m_\lambda \twoheadrightarrow L(\lambda)$. Therefore $\dim L(\lambda)
= 1$.
Finally, recall by the first cocycle condition \eqref{Ecocycle1} that an
arbitrary element $\theta \in \tangle{Q^+_0}$ can be written as $\theta =
\theta_+ \theta_-$, where $\theta_+, \theta_-^{-1} \in Q^+_0$. Now note
using the above claim:
\[ [M(\lambda) : L(\theta_- * \lambda)] > 0, \qquad [M(\theta * \lambda)
: L(\theta_+^{-1} * (\theta * \lambda))] > 0 \quad \implies \quad \theta
* \lambda \in S^3(\lambda)\ \forall \theta \in \tangle{Q^+_0}. \]

\noindent This proves the statement about $S^3(\lambda)$ since
$S^3(\lambda) \subset \tangle{\calq^+_0} * \lambda$ for every RTA and all
$\lambda \in \hzfree$. Moreover, $\lambda = \pi_{H_0}(\lambda) >
\theta_0^{-n}*\lambda$ for all $n \in \N$ and $\theta_0 \in Q^+_0
\setminus \{ 1_{Q^+_0} \}$, so that $|S^1(\lambda)| = \infty\ \forall
\lambda \in \hzfree$. This concludes the proof.
\end{proof}
%}}}

%{{{1 Section 4.2 - Examples of regular triangular monoids
\subsection{Examples of regular triangular monoids}\label{Segrtm}

Having proved the RTM-RTA Correspondence (Theorem \ref{Tmonoid}) and the
more general RTA Existence Theorem \ref{Trtm}, we now describe several
recipes to construct examples of RTMs, which in turn admit RTA
constructions.

\begin{theorem}\label{Tsemidirect}
Suppose $k \in \nn$, and for each $j = 0, \dots, k$, $Q^+_j$ is a monoid
contained in a group $\tangle{Q^+_j}$ such that $Q^+_j \setminus \{
1_{Q^+_j} \}$ is a semigroup.
\begin{enumerate}
\item If $Q^+_0$ is abelian, then it is an RTM with $\theta_1 \ltimes
\theta_2^{-1} = \theta_2^{-1}$ for $\theta_1, \theta_2 \in Q^+_0$.

\item If all $Q^+_j$ are RTMs, then so is $\times_{j=0}^k Q^+_j$.

\item Suppose $k=0$ and $Q^+_0$ is an RTM. Suppose $Q^+_0$ contains a
submonoid $Q^+$ whose action $\ltimes$ on $-Q^+_0$ stabilizes $-Q^+$.
Then $Q^+$ is also an RTM.

\item Suppose $k=1$, $Q^+_0, Q^+_1$ are RTMs, and $\tangle{Q^+_0}$ acts
admissibly on $Q^+_1$. Then $Q^+_0 \cdot Q^+_1$ is an RTM (where
$\tangle{Q^+_0} \cdot \tangle{Q^+_1}$ denotes the semidirect product
group).
\end{enumerate}
\end{theorem}

\noindent Note that parts (1),(4) are used in the proof of Theorem
\ref{Trtm}.

\begin{proof}
The first two parts are easily verified; note for the second part that we
define
\[ (\theta_j^+)_{j=0}^k \ltimes (\theta_j^-)_{j=0}^k := (\theta_j^+
\ltimes \theta_j^-)_{j=0}^k, \qquad \forall \theta_j^\pm \in \pm Q^+_j.
\]

\noindent To prove the third part, note that $\tangle{Q^+} \subset
\tangle{Q^+_0}$ is a group; moreover, the cocycle conditions
\eqref{Ecocycle1},\eqref{Ecocycle2} hold in $-Q^+$ because they hold in
$-Q^+_0$.

It remains to prove the last part; for this we will write every element
of $\tangle{Q^+_0} \cdot \tangle{Q^+_1}$ as $(\theta_1, \theta_0) =
\theta_1 \cdot \theta_0$, with $\theta_j \in \tangle{Q^+_j}$ for $j=0,1$.
That $Q^+_0 \cdot Q^+_1$ satisfies (RTM1) is not hard to verify, so we
only verify here that (RTM2) holds. For this, define
\[ (\theta_1^+, \theta_0^+) \ltimes (\theta_1^-, \theta_0^-) :=
(\theta_1^+ \ltimes \theta_0^+(\theta_1^-), \theta_0^+ \ltimes
\theta_0^-), \qquad \forall \theta_j^\pm \in \pm Q^+_j. \]

\noindent Note that this is a natural definition to propose, given the
actions $\ltimes$ of $Q^+_j$ on $-Q^+_j$ for $j=0,1$ and the semidirect
product structure of $Q^+_0 \cdot Q^+_1$. Now compute:
\begin{align*}
(\nu_1^+, \nu_0^+) \ltimes ((\theta_1^+, \theta_0^+) \ltimes
(\theta_1^-, \theta_0^-)) = &\ (\nu_1^+, \nu_0^+) \ltimes (\theta_1^+
\ltimes \theta_0^+(\theta_1^-), \theta_0^+ \ltimes \theta_0^-)\\
= &\ (\nu_1^+ \ltimes \nu_0^+(\theta_1^+ \ltimes \theta_0^+(\theta_1^-)),
\nu_0^+ \ltimes (\theta_0^+ \ltimes \theta_0^-)),\\
((\nu_1^+, \nu_0^+) \cdot (\theta_1^+, \theta_0^+)) \ltimes (\theta_1^-,
\theta_0^-) = &\ ((\nu_1^+ \nu_0^+(\theta_1^+)) \ltimes (\nu_0^+
\theta_0^+)(\theta_1^-), (\nu_0^+ \theta_0^+) \ltimes \theta_0^-).
\end{align*}

\noindent Using the admissibility of the $Q^+_0$-action on $Q^+_1$, as
well as the actions $\ltimes$ of $Q^+_j$ on $-Q^+_j$ for $j=0,1$, it
follows that both of the above quantities are equal. Therefore $\ltimes$
is indeed an action of $Q^+_0 \cdot Q^+_1$ on $(-Q^+_0) \cdot (-Q^+_1)$.
The action fixes $(1_{Q^+_1}, 1_{Q^+_0})$ because $Q^+_0, Q^+_1$ are both
RTMs.

It remains to verify the two cocycle conditions
\eqref{Ecocycle1},\eqref{Ecocycle2}. In what follows, denote $a \rtimes b
:= (b^{-1} \ltimes a^{-1})^{-1}$ for suitable $a,b$. Now to show the
first cocycle condition for $Q^+_0 \cdot Q^+_1$, we compute:
\begin{align}\label{Esemidirect}
&\ (\theta_1^+, \theta_0^+) \cdot (\theta_1^-, \theta_0^-) = \theta_1^+
\theta_0^+ \theta_1^- \theta_0^- = \theta_1^+ \theta_0^+(\theta_1^-)
\cdot (\theta_0^+ \ltimes \theta_0^-) (\theta_0^+ \rtimes \theta_0^-)\\
= &\ (\theta_1^+ \ltimes \theta_0^+(\theta_1^-)) \cdot (\theta_0^+
\ltimes \theta_0^-) \cdot (\theta_0^+ \ltimes
\theta_0^-)^{-1}(\theta_1^+ \rtimes \theta_0^+(\theta_1^-)) \cdot
(\theta_0^+ \rtimes \theta_0^-).\notag
\end{align}

\noindent Thus it suffices to prove that
\[ \left( (\theta_1^-, \theta_0^-)^{-1} \ltimes (\theta_1^+,
\theta_0^+)^{-1} \right) \cdot (\theta_0^+ \ltimes
\theta_0^-)^{-1}(\theta_1^+ \rtimes \theta_0^+(\theta_1^-)) \cdot
(\theta_0^+ \rtimes \theta_0^-) = (1_{Q^+_1}, 1_{Q^+_0}) = 1, \]

\noindent i.e., that
\[ ((\theta_0^-)^{-1}(\theta_1^-)^{-1} \ltimes (\theta_0^+
\theta_0^-)^{-1}(\theta_1^+)^{-1}, (\theta_0^-)^{-1} \ltimes
(\theta_0^+)^{-1}) \cdot ((\theta_0^+ \ltimes \theta_0^-)^{-1}(\theta_1^+
\rtimes \theta_0^+(\theta_1^-)), (\theta_0^+ \rtimes \theta_0^-)) = 1, \]

\noindent i.e., that (using the first cocycle condition \eqref{Ecocycle1}
for $Q^+_0$):
\[ (\theta_0^-)^{-1}(\theta_1^-)^{-1} \ltimes (\theta_0^+
\theta_0^-)^{-1}(\theta_1^+)^{-1} \cdot ((\theta_0^-)^{-1} \ltimes
(\theta_0^+)^{-1}) \left[ (\theta_0^+ \ltimes \theta_0^-)^{-1}(\theta_1^+
\rtimes \theta_0^+(\theta_1^-)) \right] = (1_{Q^+_1}, 1_{Q^+_0}) = 1. \]

\noindent But now the first cocycle condition \eqref{Ecocycle1} and
action on $Q^+_1$ for $Q^+_0$ show that the second factor on the
left-hand side equals $(\theta_0^+ \theta_0^-)^{-1} \left( \theta_1^+
\rtimes \theta_0^+(\theta_1^-) \right)$. Similarly, the admissibility of
the $Q^+_0$-action on $Q^+_1$ shows that the first factor on the
left-hand side equals $(\theta_0^+ \theta_0^-)^{-1} \left(
\theta_0^+(\theta_1^-)^{-1} \ltimes (\theta_1^+)^{-1} \right)$.
Multiplying these two factors, we are now done by using the first cocycle
condition for $Q^+_1$.

Similarly, the second cocycle condition is verified as follows, using
\eqref{Esemidirect} and the cocycle conditions for $Q^+_j$:
\[ ((\theta_1^+, \theta_0^+) \ltimes (\theta_1^-, \theta_0^-)) \cdot
\left(((\theta_1^+, \theta_0^+) \rtimes (\theta_1^-, \theta_0^-)) \ltimes
(\nu_1^-, \nu_0^-) \right) = a \cdot b \cdot \left[ (b^{-1}(c),d) \ltimes
(\nu_1^-, \nu_0^-) \right], \]

\noindent where $a := \theta_1^+ \ltimes \theta_0^+(\theta_1^-)$, $b :=
\theta_0^+ \ltimes \theta_0^-$, $c := \theta_1^+ \rtimes
\theta_0^+(\theta_1^-)$, and $d := \theta_0^+ \rtimes \theta_0^-$. In
turn, this expression equals
\[ = a \cdot b \cdot (b^{-1}(c) \ltimes d(\nu_1^-)) \cdot (d \ltimes
\nu_0^-) = a \cdot (c \ltimes (bd)(\nu_1^-)) \cdot b \cdot (d \ltimes
\nu_0^-). \]

\noindent Recall by the cocycle condition for $Q^+_0$ that $bd =
\theta_0^+ \theta_0^-$. Now we compute the other side of the second
cocycle condition:
\begin{align*}
&\ (\theta_1^+, \theta_0^+) \ltimes ((\theta_1^-, \theta_0^-) \cdot
(\nu_1^-, \nu_0^-)) = (\theta_1^+ \ltimes \theta_0^+(\theta_1^- \cdot
\theta_0^-(\nu_1^-)), \theta_0^+ \ltimes (\theta_0^- \nu_0^-))\\
= &\ \left( (\theta_1^+ \ltimes \theta_0^+(\theta_1^-)) \cdot \left[
(\theta_1^+ \rtimes (\theta_0^+(\theta_1^-))) \ltimes
(\theta_0^+(\theta_0^-(\nu_1^-))) \right], (\theta_0^+ \ltimes
\theta_0^-) \cdot ((\theta_0^+ \rtimes \theta_0^-) \ltimes \nu_0^-)
\right)\\
= &\ a \cdot (c \ltimes (bd)(\nu_1^-)) \cdot b \cdot (d \ltimes \nu_0^-),
\end{align*}

\noindent and this proves the second cocycle condition, as desired.
\end{proof}

We now describe an application of Theorem \ref{Tsemidirect}(1),(4), which
yields a natural class of solvable examples of RTMs:

\begin{cor}\label{Csemidirect}
Suppose $G$ is a group with abelian subgroups $G_0, \dots, G_n$ such
that:
\begin{itemize}
\item $G_j$ acts on $G_k$ for $0 \leq j \leq k \leq n$ by group
homomorphisms, in a compatible manner such that $G = (\cdots ((G_n
\rtimes G_{n-1}) \rtimes G_{n-2}) \rtimes \cdots G_1) \rtimes G_0$; and

\item For every $0 \leq k \leq n$, $G_k$ contains a sub-monoid $G_k^+$
stable under the action of $(\cdots ((G_k \rtimes G_{k-1}) \rtimes
G_{k-2}) \rtimes \cdots G_1) \rtimes G_0$, such that $G_k^+ \setminus \{
1_{G_k^+} \}$ is a semigroup that generates $G_k$.
\end{itemize}

\noindent Then $G^+ := (\cdots ((G_n^+ \rtimes G_{n-1}^+) \rtimes
G_{n-2}^+) \rtimes \cdots G_1^+) \rtimes G_0^+$ is a regular triangular
monoid.
\end{cor}

\noindent Note that Corollary \ref{Cabelian} is a particular special case
with $n=0$. Similarly, if $n=1$ then this result implies Theorem
\ref{Tsemidirect}(4) when $G_1^+$ is an abelian RTM with the usual
(trivial) $\ltimes$-action.

\begin{proof}
The proof is by induction on $n$. For $n=0$ the result follows from
Theorem \ref{Tsemidirect}(1). Now suppose the result holds for $n-1$,
whence $M^+ := (\cdots ((G_n^+ \rtimes G_{n-1}^+) \rtimes G_{n-2}^+)
\rtimes \cdots G_1^+)$ is an RTM. Define the action map $\ltimes$ of
$G^+$ on $-G^+$ as follows:
\begin{equation}
(g_n^+, \dots, g_0^+) \ltimes (g_n^-, \dots, g_0^-) := \left( (g_{n-1}^+
\cdots g_0^+)(g_n^-), \dots, (g_1^+ g_0^+)(g_2^-), g_0^+(g_1^-), g_0^-
\right),
\end{equation}

\noindent where $g_k^\pm \in \pm G_k^+$ for $0 \leq k \leq n$. It is now
a straightforward calculation to verify that $G_0$ acts admissibly on the
regular triangular monoid $M^+$, whence we are done by induction via
Theorem \ref{Tsemidirect}(4).
\end{proof}

\begin{remark}
Theorem \ref{Csemidirect} holds for all groups that can be expressed as a
tower of semidirect products. Clearly such groups $G$ include all abelian
groups; each such group $G$ is solvable; and if all $G_k$ are finitely
generated, then $G$ is polycyclic. A natural question to explore is if
every solvable group with a given set of abelian Jordan-Holder factors
generated by RTMs, is also generated by an RTM.
\end{remark}
%}}}

%{{{1 Section 4.3 - Non-based example with non-abelian span of roots
\subsection{Non-based example with non-abelian span of
roots}\label{Sexample}

We conclude this section by studying an RTA $\cala(Q^+_0, {\bf c})$
(constructed in the RTA Existence Theorem \ref{Trtm}) for a specific
non-abelian monoid $Q^+_0$, as well as its Category $\calo$.

Fix $k \in \N$, $\boldsymbol{\zeta} \in (0,\infty)^k$, ${\bf c} \in
\F^k$, and a nontrivial additive subgroup $\E \subset (\R, +)$ such that
$\zeta_j^{\pm 1} \E \subset \E$ for $1 \leq j \leq k$. Then $\E \cap
[0,\infty)$ is an abelian RTM with the trivial $\ltimes$-action on $\E
\cap (-\infty,0]$. Set $\boldsymbol{\zeta}^{\bf n} := \prod_{j=1}^k
\zeta_j^{n_j}$ and ${\bf n}(e) := \boldsymbol{\zeta}^{\bf n} e$ for ${\bf
n} \in \Z^k$ and $e \in \E$. Then $\Z^k$ acts admissibly on $\E$ (by
Corollary \ref{Csemidirect}), which allows us to define the $\F$-algebra
$\alga(\E, {\bf c}) := \cala((\nn)^k \ltimes_{\boldsymbol \zeta} (\E \cap
[0,\infty)), {\bf c})$ as in the proof of Theorem \ref{Trtm}. Here we
use $\ltimes_{\boldsymbol \zeta}$ to denote the semidirect product of the
groups $\Z^k$ and $\E$, in order to differentiate it from the RTM action
$\ltimes$.

\begin{theorem}\label{Tdiamond}
Fix $k \in \N,\ {\boldsymbol \zeta} \in (0,\infty)^n,\ \E \subset \R$,
and ${\bf c} \in \F^k$ as above.
\begin{enumerate}
\item The algebra $\alga(\E, {\bf c})$ is a strict RTA with $\calq^+_0 =
(\nn)^k \ltimes_{\boldsymbol \zeta} (\E \cap [0,\infty))$.

\item The algebra $\alga(\E, {\bf c})$ is based if and only if it is
discretely graded, if and only if $\E = \eta \Z$ for some $\eta \in
\R^\times$ and $\zeta_j = 1$ for all $j$.
\end{enumerate}
\end{theorem}

\noindent Thus, to our knowledge the algebras $\alga(\E, {\bf c})$ with
${\boldsymbol \zeta} \neq (1, \dots, 1)$ provide the first explicitly
constructed examples of regular triangular algebras with non-abelian
group of roots $\tangle{\calq^+_0}$. These algebras cannot be studied by
using existing theories of Category $\calo$ in the literature (e.g.~as in
\cite{H2,Kh2,MP}), because the ``root lattice" is not abelian. In fact
the monoid $\calq^+_0$ is abelian if and only if $\zeta_j = 1$ for all
$j$.

\begin{proof}
The first part follows directly from Theorem \ref{Trtm}. To show the
second part, note that if $\alga(\E, {\bf c})$ is based then it is
discretely graded. In turn, this implies that the interval $[(0,0),
(0,e)]$ is finite for every $0 < e \in \E \subset \R$. Thus $\E$ is a
lattice $\eta \Z$ for $\eta \neq 0$, which contains $\zeta_j^\Z \eta \Z$
for all $j$. It follows that $\zeta_j = 1\ \forall j$. Finally, if $\E =
\eta \Z$ and $\zeta_j = 1\ \forall j$, then $A$ is indeed based with
$\Delta := \{ \vi_1, \dots, \vi_k, (0,\eta) \}$.
\end{proof}

We now study Category $\calo$ over the algebra $\alga(\E, {\bf c})$,
including computing the center and its action on Verma modules.

\begin{prop}
Fix ${\boldsymbol \zeta} \in (0,\infty)^n$, $\E \subset \R$, and ${\bf c}
\in \F^k$ as above. Define $J := \{ j \in [1,k] : c_j = 0 \}$.
\begin{enumerate}
\item $\alga(\E, {\bf c})$ contains a central subalgebra $Z_0 := \F[\{
x_j^- x_j^+ : j \in J \}]$ that is isomorphic to a polynomial algebra in
$|J|$ variables. Now suppose $\ch \F = 0$ if $J \subsetneq \{ 1, \dots, k
\}$. Then the center of $\alga(\E, {\bf c})$ equals:
\[ Z(\alga(\E, {\bf c})) = \begin{cases}
Z_0 \otimes_\F {\rm span}_\F \{ t^{-e} t^e : e \in \E \cap [0,\infty) \},
& \text{ if } {\boldsymbol \zeta} = (1, \dots, 1);\\
Z_0, &\text{ otherwise}.
\end{cases} \]

\item Suppose ${\boldsymbol \zeta} \neq (1, \dots, 1)$. Then there are
exactly $k+1$ isomorphism classes of algebras among the family $\{
\alga(\E, {\bf c}) : {\bf c} \in \F^k \}$ (where we assume $\ch \F = 0$
if $J \subsetneq \{ 1, \dots, k \}$).

\item The algebras $B^\pm$ do not contain zerodivisors. Thus every
nonzero map of Verma modules is an embedding.

\item Define $K := \{ j \in [1,k] : \zeta_j \neq 1 \}$. Then a suitable
completion of $\alga(\E, {\bf c})$ contains central ``Casimir" operators
of the form
\[ T(e) := \sum_{{\bf n} \in \Z^K} t^{-e \prod_{j \in K} \zeta_j^{n_j}}
t^{e \prod_{j \in K} \zeta_j^{n_j}}, \qquad \forall e \in \E \cap
(0,\infty), \]

\noindent where $T(e) = t^{-e} t^e \in \alga(\E, {\bf c})$ if $K$ is
empty. Then the operators $T(e)$ act on all objects in $\calo[\hzfree]$.
Moreover, $T(e)$ and $Z(\alga(\E, {\bf c}))$ kill every highest weight
module in $\calo$, hence act nilpotently on $\calo[\hzfree]$.

\item $\tangle{\calq_0^+} = \zze \subset \hzfree$, and the algebra
$\alga(\E, {\bf c})$ satisfies none of the Conditions (S) because $\Z^J
\ltimes_{\boldsymbol \zeta} \E$ is in each block. More precisely,
$S^3(\lambda) \supset (\Z^J \ltimes_{\boldsymbol \zeta} \E) * \lambda\
\forall \lambda \in \hzfree$.
\end{enumerate}
\end{prop}

\begin{proof}\hfill
\begin{enumerate}
\item The first assertion in this part is easily verified using the
algebra relations. Next, the center is contained in the centralizer of
$H_0$: $Z(\alga(\E, {\bf c})) \subset Z_{\alga(\E, {\bf c})}(H_0)$, and
Theorem \ref{Tdiamond} can be used to show that
\begin{equation}\label{Ediamond5}
Z_{\alga(\E, {\bf c})}(H_0) = H_0[x_1^- x_1^+, \dots, x_k^- x_k^+]
\otimes_\F {\rm span}_\F \{ t^{-e} t^e : e \in \E \cap [0,\infty) \}.
\end{equation}

\noindent Now given ${\bf m} \in (\nn)^k$, define ${\bf x}^{\bf m} :=
\prod_{j=1}^k (x_j^- x_j^+)^{m_j}$. Then write an arbitrary element $z
\in Z(\alga(\E, {\bf c})) \subset Z_{\alga(\E, {\bf c})}(H_0)$ using
\eqref{Ediamond5}:
\[ z = \sum_{{\bf m'} \in (\nn)^J} {\bf x}^{\bf m'} \sum_{{\bf m} \in
(\nn)^{J^c}} {\bf x}^{\bf m} \sum_{i=1}^{N({\bf m'}, {\bf m})} h_i
t^{-e_i} t^{e_i}, \]

\noindent for a suitable choice of elements $e_i \in \E \cap [0,\infty),
0 \neq h_i \in H_0$, and where $J^c := \{ 1, \dots, k \} \setminus J$.
Now note by (RTA1) that $z$ is central if and only if the inner double
summation in the previous equation is central for each fixed ${\bf m'}$.
Thus, assume without loss of generality that
\[ z = \sum_{\bf m} {\bf x}^{\bf m} \sum_{i=1}^n h_{i,{\bf m}}
t^{-e_{i,{\bf m}} } t^{e_{i, {\bf m}} } \]

\noindent for suitable $h_{i, {\bf m}}, e_{i, {\bf m}}$. Note by the
algebra relations that the variables $\{ x_j^\pm : j \not\in J \}$ each
generate a copy of the Weyl algebra in $\alga(\E, {\bf c})$, and this has
trivial center since $\ch \F = 0$. It follows via (RTA1) that the central
element $z$ in the above form has only one nonzero term in the outer sum,
namely, the term corresponding to ${\bf m} = {\bf 0}$. Thus we may assume
that $z = \sum_{i=1}^n h_i t^{-e_i} t^{e_i}$. Now define ${\bf x}_+^{\bf
m} := \prod_{j=1}^k (x_j^+)^{m_j}$ for ${\bf m} \in (\nn)^k$, and
compute:
\[ 0 = {\bf x}_+^{\bf m} z - z {\bf x}_+^{\bf m} = \sum_{i=1}^n h_i(({\bf
m},0) \cdot -) t^{- {\boldsymbol \zeta}^{\bf m} e_i} t^{{\boldsymbol
\zeta}^{\bf m} e_i} {\bf x}_+^{\bf m} - \sum_{i=1}^n h_i(-) t^{-e_i}
t^{e_i} {\bf x}_+^{\bf m}, \qquad \forall {\bf m} \in (\nn)^k. \]

\noindent Now for a fixed ${\bf m}$, since both sums involve finitely
many terms, there is a unique largest positive exponent for $t$ in both
sums. For the two sums to be equal, either $e_i = 0$ for all $i$, or
${\boldsymbol \zeta}^{\bf m} = 1$. There are now two cases:
\begin{itemize}
\item If $\zeta_j = 1\ \forall j$, then $Z$ is easily seen to equal
$(Z(\alga(\E, {\bf c})) \cap H_0)[\{ x_j^- x_j^+ : j \in J \}] \otimes_\F
{\rm span}_\F \{ t^{-e} t^e : e \in \E \cap [0,\infty) \}$. Moreover, it
is not hard to show that $Z(\alga(\E, {\bf c})) \cap H_0 = \F$, which
proves this case.

\item Otherwise there exists ${\bf m}$ such that ${\boldsymbol
\zeta}^{\bf m} \neq 1$. In this case, the above computation must
necessarily have one term, corresponding to $e_1 = 0$. But then we are
once again reduced to computing $Z(\alga(\E, {\bf c})) \cap H_0$, which
is $\F$.
\end{itemize}

\item First note by rescaling the $x_j^-$, say, that $\alga(\E, {\bf c})$
is an associative $\F$-algebra that is isomorphic to the algebra
$\alga(\E, {\bf d})$, where $d_j := 1 - \delta_{c_j, 0}\ \forall j$.
Further observe that for any permutation $\sigma \in S_k$, we have an
obvious isomorphism of algebras $\alga(\E, {\bf c}) \cong
\cala_{\sigma({\boldsymbol \zeta})}(\E, \sigma({\bf c}))$. It now remains
to count the number of possible nondecreasing $0,1$-valued sequences of
length $k$, and there are precisely $k+1$ of them. Since ${\boldsymbol
\zeta} \neq (1, \dots, 1)$, the previous part shows that the centers of
these $k+1$ algebras are polynomial rings with pairwise distinct
transcendence degrees over $\F$. Thus the $k+1$ algebras in question are
pairwise non-isomorphic.

\item That $B^\pm$ do not contain zerodivisors holds more generally by
Theorem \ref{Trtm}. The statement about Verma module embeddings is now
standard.

%\item This part is also easily verified using the algebra relations. As
%in the Kac-Moody setting, the ``Casimir" operators $T(e)$ act on
%arbitrary objects of Category $\calo[\hzfree]$. Moreover, both $T(e)$ and
%$Z(\alga(\E, {\bf c}))$ (by the previous parts) kill the highest vector
%in any highest weight module in the respective Categories $\calo$, since
%the Harish-Chandra projection to $H_0$ kills all of these operators. It
%follows that these central elements annihilate the entire module.
%
\item This part is proved similarly to Proposition \ref{Prtmo1}.

\item Note that $\tangle{\calq^+_0} = \zze$, which embeds into $\hzfree$
via the evaluation maps. Next, given $\lambda \in \hzfree$ and we first
\textbf{claim} that every nonzero weight vector of the Verma module
$M(\lambda)$ with weight in $(\Z^J \ltimes_{\boldsymbol \zeta} \E) *
\lambda$ is maximal. To show the claim, it suffices to show that $b_-
m_\lambda$ is maximal, for every monomial word $b_- = t^{-e} \prod_{j \in
J} (x_j^-)^{n_j} \in X^-_{irr}$. But now we compute using the algebra
relations that $t^{e'} \cdot b_- m_\lambda = b_- t^{\prod_{j \in J}
\zeta_j^{n_j} e'} m_\lambda = 0$; and similarly, $x_j^+ b_- m_\lambda =
0$ for all $1 \leq j \leq k$. This proves the claim. Finally, 
%
%
%an arbitrary element $\theta \in \Z^J \ltimes_{\boldsymbol \zeta} \E$
%can be written as $\theta = g_+ g_-$, where $g_+, g_-^{-1} \in (\nn)^J
%\ltimes_{\boldsymbol \zeta} (\E \cap [0,\infty))$. Thus we obtain using
%the above claim (in this part): \[ [M(\lambda) : L(g_- * \lambda)] > 0,
%\qquad [M(\theta * \lambda) : L(g_+^{-1} * (\theta * \lambda))] > 0
%\quad \implies \quad \theta * \lambda \in S^3(\lambda)\ \forall \theta
%\in \Z^J \ltimes_{\boldsymbol \zeta} \E. \]
%
%\noindent This concludes the proof because $\lambda = \pi_{H_0}(\lambda)
%> (0,-e)*\lambda$ for all $e \in \E \cap (0,\infty)$, so that
%$|S^1(\lambda)| = \infty\ \forall \lambda \in \hfree$.
%
similar to Proposition \ref{Prtmo2} we obtain that $S^3(\lambda) \supset
\Z^J \ltimes_{\boldsymbol \zeta} \E$ and $S^1(\lambda) \supset(\E \cap
(-\infty,0)) * \lambda$ for all $\lambda \in \hzfree$.
\end{enumerate}
\end{proof}

\begin{remark}
In \cite{GGOR}, one finds a homological treatment of Category $\calo$
over a very general class of algebras $A$. We point out that the
framework in the present paper cannot be subsumed under that paradigm,
because of non-based examples such as $A = \alga(\E, {\bf c})$ above. In
such non-based cases, there does not exist an inner grading by any
subgroup of $\R$ (e.g.~via taking the commutator with some element
$\partial \in A$, as discussed in \cite{GGOR}).
\end{remark}
%}}}

\section{Based and non-based Lie algebras with triangular
decomposition}\label{Sexam}

%{{{1 Section 5.0
The remainder of the paper focusses on applying the theory from Section
\ref{Sverma} to a large class of algebras studied in the literature -- as
well as novel examples including stratified Virasoro algebras and certain
triangular generalized Weyl algebras. The examples are presented in
``decreasing order of familiarity" in the following sense: this section
and the next discuss two ``well-known" families of strict, based Hopf
RTAs of finite rank: Lie algebras and quantum groups. The reader who
wishes to skip these examples and focus immediately on non-strict or
non-Hopf RTAs, can jump ahead to (a) non-based RTAs in Section
\ref{Snonbased}; (b) infinitesimal Hecke algebras in Section \ref{Smore}
(rank one) and Section \ref{Sinfhecke} (higher rank); or (c) generalized
Weyl algebras in Sections \ref{Sgwa} and \ref{Sgwa2}.

%In the present and subsequent sections, the goal is to show that the
%methods developed in the general setting of an RTA are sufficient to
%obtain a great deal of homological information about Verma modules and
%blocks of Category $\calo$, in all of these examples.

We begin by discussing the case of $A = U \lie{g}$ for $\lie{g}$ a Lie
algebra with regular triangular decomposition. Such Lie algebras are
defined and explored in great detail in \cite{RCW,MP}, so this section is
restricted to briefly mentioning some examples, after defining such Lie
algebras.
We also observe at the very outset that by Remark \ref{Rfree}, it is
possible to work with all of $\calo = \calo[\hofree]$ when $A$ is a HRTA.
This is the case in the present section as well as the next two.

\begin{definition}
Assume $\ch \F = 0$.
A Lie algebra $\lie{g}$, together with the following data, is a {\em Lie
algebra with triangular decomposition} (also called a {\em regular
triangular Lie algebra} or {\em RTLA}):
\begin{enumerate}
\item $\lie{g} = \lie{g}^- \oplus \lie{h} \oplus \lie{g}^+$, where all
summands are nonzero Lie subalgebras of $\lie{g}$, and $\lie{h}$ is
abelian.

\item $\lie{g}^+$ is an $\ad \lie{h}$-semisimple module with
finite-dimensional $\lie{h}$-weight spaces.

\item All $\ad \lie{h}$-weights for $\lie{g}^+$ lie in $Q^+ \setminus \{
0 \}$, where $Q^+$ denotes a free abelian monoid with finite basis
$\Delta' := \{ \alpha_j \}_{j \in J}$; this basis consists of linearly
independent vectors in $\lie{h}^*$.

\item There exists an anti-involution $\omega$ of $\lie{g}$ that sends
$\lie{g}^+$ to $\lie{g}^-$ and preserves $\lie{h}$ pointwise.
\end{enumerate}
\end{definition}

\noindent In contrast, a general, non-based RTA does not require
$\calq^+_0$ to be $\nn \Delta'$ for finite -- or infinite -- $\Delta'$.
Also note that we require $\ch \F = 0$ in order that the abelian monoid
$\calq^+_0 = Q^+ = \Z^+ \Delta'$ is an RTM (i.e., satisfies Condition
(RTM1)).

The following result summarizes the main (functorial) properties of such
Lie algebras, and is not hard to show.

\begin{prop}\label{Prtla}\hfill
\begin{enumerate}
\item If $\lie{g}$ is an RTLA, then $U \lie{g}$ is a strict, based Hopf
RTA with base of simple roots $\Delta'$.

\item If $\lie{g}_i$ is an RTLA for $1 \leq i \leq n$, and $\lie{h}'$ is
an abelian Lie algebra, then $\lie{h'} \oplus \bigoplus_{i=1}^n
\lie{g}_i$ is an RTLA as well (with pairwise commuting summands).

\item If $\lie{g}$ is an RTLA and $V \subset Z(\lie{g})$ is any subspace,
then $\lie{g}/V$ is an RTLA.
\end{enumerate}
\end{prop}

\noindent Note here that the adjoint action of $H_1 = H_0 = U \lie{h} =
\sym \lie{h}$ is given by $\ad h(x) = hx - xh$ for $x \in A = U \lie{g}$.
Moreover, Condition (RTA1) holds because of the Poincare-Birkhoff-Witt
Theorem for $U \lie{g}$, and $\widehat{H_1} = \widehat{H_0} = \lie{h}^*$.

\subsection{Examples of RTLAs}\label{Srtla}

For completeness, we mention a large number of well-studied examples of
RTLAs in the literature (which yield strict Hopf RTAs).

\begin{exam}[{\em Symmetrizable Kac-Moody Lie algebras}]

See \cite{Kac2} for the definition and basic properties of $\lie{g} =
\lie{n}^- \oplus \lie{h} \oplus \lie{n}^+$. Note that if $\lie{g}$ is
complex semisimple (and finite-dimensional), then Harish-Chandra's
theorem implies that $S^4(\lambda) = W \bullet \lambda\ \forall \lambda
\in \lie{h}^*$ (the twisted Weyl group orbit). Thus, all Conditions (S)
hold by Theorem \ref{Tfirst}, and all blocks $\calo[S^3(\lambda)]$ of
$\calo = \calo[\hofree]$ are highest weight categories with BGG
Reciprocity.
\end{exam}

We now mention two generalizations of Kac-Moody Lie algebras, which are
also RTLAs.

\begin{exam}[{\em Contragredient Lie algebras}]\label{Ekk}

These are a family of Lie algebras defined in \cite{KK}, which can be
verified to be RTLAs (and for which Kac and Kazhdan proved the Shapovalov
determinant formula).
\end{exam}

\begin{exam}[{\em Some (symmetrizable) Borcherds algebras and central
extensions}]

These Lie algebras are defined and studied in \cite{Bo1,Bo2}; we remark
that they are also RTLAs under certain additional assumptions, but not in
general.
\end{exam}

\begin{exam}[{\em The Virasoro and Witt algebras}]

The Witt algebra is the centerless Virasoro algebra. Both of these Lie
algebras are RTLAs; see \cite{FeFu,KR}, for example.
\end{exam}

\begin{exam}[{\em Heisenberg algebras extended by derivations}]

Both these and the (centerless) Virasoro algebras can be found in
\cite{MP}, for instance. It is not hard to show that all Conditions (S)
fail to hold for (centerless) extended Heisenberg algebras if $V \neq 0$.
\end{exam}

\begin{exam}[{\em Certain quotients of preprojective algebras of
loop-free quivers}]

Let $Q$ be a finite acyclic quiver (i.e., containing no loops or oriented
cycles) with path algebra $\F Q = \oplus_{n \geq 0} (\F Q)_n$, where each
summand has a basis consisting of (oriented) paths in $Q$ of length $n$.
Thus $(\F Q)_0$ and $(\F Q)_1$ have bases $I$ of vertices $e_i$ and $E$
of edges $a$ respectively. Assume $I,E \neq \emptyset$. Now construct the
{\it double} $\qbar$ of $Q$, by adding an ``opposite" edge $a^*$ for each
$a \in E$.

The sub-quiver $Q^*$ is defined with vertices $I$ and edges $a^*$. Now
define $\lie{g} = \F \qbar / (a'a^*, a^* a' : a' \in (\F Q)_1, a^* \in
(\F Q^*)_1)$. This is an associative algebra, and a quotient of the {\em
preprojective algebra} introduced in \cite{GP}, namely, $\F \qbar /
(\sum_{a \in E} [a,a^*])$. One uses the associative algebra structure to
show that $\lie{g}$ is an RTLA, using: $\lie{g}^+ := \bigoplus_{n>0} (\F
Q)_n,\ \lie{h} := (\F Q)_0$, and $\lie{g}^- := \bigoplus_{n>0} (\F
Q^*)_n$. Moreover, $[\lie{g}^+, \lie{g}^-] = 0$, using which it can be
shown that all Conditions (S) fail to hold.
\end{exam}

\begin{remark}[{\em Toroidal Lie algebras}]
These Lie algebras are defined (see \cite[Section 0]{BM}) to be the
universal central extensions of $R_n \otimes \lie{g}$, where $\lie{g}$ is
a simply laced Lie algebra and $R_n = \F[T_1^{\pm 1}, \dots, T_n^{\pm
1}]$. The central extension is by $\lie{Z} := \Omega^1 R_n / d R_n$.

Clearly, the regularity condition fails here, so that toroidal Lie
algebras are not RTLAs. We can, however, look at a related algebra,
namely $U \lie{g} \otimes R_n$. By the above result, this is a strict
Hopf RTA. If the central extension above splits, then $U(\lie{g} \oplus
\lie{Z}) \otimes R_n$ is also a strict Hopf RTA.
\end{remark}
%}}}

%{{{1 Section 5.1 - Non-based Lie algebras with triangular decomposition
\subsection{Non-based Lie algebras with triangular
decomposition}\label{Snonbased}

We now discuss examples of non-based RTAs arising from Lie algebras
(which are necessarily not RTLAs). Such Lie algebras have emerged from
mathematical physics and are the subject of active study.

\begin{exam}[{\em Generalized Virasoro algebras}]

A relatively modern construction (which is among the RTLAs not covered in
\cite{MP}, say) consists of generalized Virasoro algebras ${\rm Vir}[G]$.
These algebras were defined in \cite{PZ} and have been the subject of a
large body of literature; see e.g.~\cite{HWZ,LZ} and the references
therein. They involve working over a field $\F$ of characteristic zero
and a group $0 \neq G \subset (\F,+)$. Then ${\rm Vir}[G]$ is a
$G$-graded Lie algebra with similar relations to the usual Virasoro
algebra. Now suppose $\F = \R \supset G$. If $G = \alpha \Z$ for some
$\alpha \neq 0$ then ${\rm Vir}[G]$ is discretely graded (and based);
otherwise for $G \neq \alpha \Z$, the algebra is not discretely graded
(and hence not based). In the case when the group $G$ has a total
ordering compatible with addition, ${\rm Vir}[G]$ has a triangular
decomposition -- in fact, $U({\rm Vir}[G])$ turns out to be a (possibly
non-based) strict Hopf RTA -- and its Category $\calo$ has been studied
in great depth; see \textit{loc.~cit.}
\end{exam}

\begin{remark}
Note in the theory developed above that the group $\tangle{\calq^+_r}$
usually does not equal the disjoint union $(\calq^-_r \setminus \{
\id_{H_r} \}) \coprod \calq^+_r$ for $r=0,1$ (notation as in Lemma
\ref{Lfirst}). For instance, this is the case for semisimple Lie algebras
(and more generally, for all RTLAs) of rank at least $2$. However,
sometimes it does happen that $\tangle{\calq^+_r} = \calq^+_r \cup
\calq^-_r$. One example is precisely the higher rank/generalized Virasoro
algebras over a totally ordered group $G$.
\end{remark}

\begin{exam}[{\em Generalized Schr\"odinger-Virasoro algebras}]

Another modern construction of a strict RTA not found in \cite{MP} is the
Schr\"odinger-Virasoro algebra. This is a Lie algebra whose construction
is motivated by the free Schr\"odinger equation in $(1+1)$-variables, and
involves extending the centerless Virasoro Lie algebra by a 2-step
nilpotent Lie algebra formed by bosonic currents. The larger class of
generalized Schr\"odinger-Virasoro algebras $\lie{gsv}[G]$ over totally
ordered groups $G \subset (\F, +)$, as well as their Verma modules were
studied in \cite{TZ} (see also \cite{LS}). Once again, their universal
enveloping algebras provide examples of Hopf RTAs that are possibly
non-based.
\end{exam}

\begin{exam}[{\em Twisted Heisenberg-Virasoro algebra}]
This algebra was introduced and studied by Billig in \cite{Bi}. It is not
hard to show that its universal enveloping algebra is a Hopf RTA.
\end{exam}

\subsection*{General construction: stratified Virasoro algebras}

In light of the above ``generalized" Virasoro-type examples, it is
natural to ask if there is a unified framework of a general Lie algebra
$\lie{g}$, which encompasses all of the above examples (i.e., in Section
\ref{Snonbased}). We now provide a positive answer to this question, over
an arbitrary field $\F$ of characteristic zero:
\begin{enumerate}
\item $\lie{g}$ is a Lie algebra for which there exist nonnegative
integers $M,N \in \nn$ such that
\[ \lie{g} = Z \oplus \bigoplus_{j=0}^N \lie{g}_j, \qquad \lie{g}_0 =
\bigoplus_{k=0}^M \lie{g}_0[k], \]

\noindent with all summands being vector spaces, and $Z$ central in
$\lie{g}$.

\item There exists an additive subgroup $G^0_k \subset (\F,+)$ for each
$0 \leq k \leq M$, such that $G^0_k + G^0_{k'} \subset G^0_{k+k'}$
whenever $k+k' \leq M$.

\item For each $1 \leq j \leq N$, there exists a subset $G_j^+ \subset
(\F, +)$ satisfying:
(a) $G_j^+$ is an additive subgroup of $(\F, +)$, or else $\tangle{G_j^+}
\setminus G_j^+$ is an additive subgroup of $(\F, +)$ and $G_j^+$ is a
coset of it;
(b) $G^+_j + G^+_{j'} \subset G^+_{j+j'}$ whenever $j+j' \leq N$; and
(c) $G^0_k + G^+_j \subset G^+_j$ for all $j,k$.

\item There exists a total ordering on the subgroup of $\F$ spanned by
all $G_k^0, G_j^+$.

\item For all $0 < j \leq N$, the vector space $\lie{g}_j$ is spanned by
an $\F$-basis $\{ L_{j,\alpha}^+ : \alpha \in G^+_j \}$. Similarly, for
all $0 \leq k \leq M$, the vector space $\lie{g}_0[k]$ is spanned by an
$\F$-basis $\{ L_{k,\alpha}^0 : \alpha \in G^0_k \}$. Moreover, these
basis vectors satisfy the relations:
\begin{align}
[L^+_{j,\alpha}, L^+_{j',\beta}] = &\ {\bf 1}_{j+j' \leq N}
f^{++}_{j,j'}(\alpha,\beta) L^+_{j+j',\alpha+\beta} + {\bf 1}_{\alpha +
\beta = 0} g^{++}_{j,j'} (\alpha, \beta) z_{j,j'}^{++}, \notag\\
[L^0_{k,\alpha}, L^0_{k',\beta}] = &\ {\bf 1}_{k+k' \leq M}
f^{00}_{k,k'}(\alpha,\beta) L^0_{k+k',\alpha+\beta} + {\bf 1}_{\alpha +
\beta = 0} g^{00}_{j,j'} (\alpha, \beta) z_{k,k'}^{00},\\
[L^0_{k,\alpha}, L^+_{j,\beta}] = &\ f^{0+}_{k,j}(\alpha,\beta)
L^0_{j,\alpha+\beta} + {\bf 1}_{\alpha + \beta = 0} g^{0+}_{k,j} (\alpha,
\beta) z_{k,j}^{0+}, \notag
\end{align}

\noindent for suitable functions $f^{00}_{j,j'}, f^{++}_{k,k'},
f^{0+}_{k,j}$ and similarly for the $g$-functions, and with (suitable)
central elements $z^{++}_{j,j'}, z^{00}_{k,k'}, z^{0+}_{k,j} \in Z$.
\end{enumerate}

The aforementioned construction yields a Lie algebra whose universal
enveloping algebra is an RTA, provided the $f,g$-functions and central
elements satisfy certain compatibility conditions arising for the
following reasons:
\begin{itemize}
\item the anti-symmetry of the Lie bracket;
\item the Jacobi identity; and
\item the anti-involution $i$, which should send $L^+_{j,\alpha}$ to
$L^+_{j,-\alpha}$ and $L^0_{k,\alpha}$ to $L^0_{k,-\alpha}$.
\end{itemize}

\noindent Call any Lie algebra $\lie{g}$ satisfying these assumptions a
\textbf{stratified Virasoro algebra}. Then $U(\lie{g})$ is a Hopf RTA
with Cartan subalgebra $U(\lie{h})$, where $\displaystyle \lie{h} = Z
\oplus \bigoplus_{k=0}^M \F L^0_{k,0} \oplus \bigoplus_{j : 0 \in G_j^+}
\F L^+_{j,0}$. It is not hard to show that this construction of a
stratified Virasoro algebra encompasses all of the variants of
Virasoro-type algebras discussed above. For instance, the usual Virasoro
algebra is a stratified Virasoro algebra with $\dim_\F Z = 1$ and $M = N
= 0$, with $G^0_0 = \Z$.
%}}}

%{{{1 Section 6 - Extended quantum groups for symmetrizable Kac-Moody Lie algebras
\section{Extended quantum groups for symmetrizable Kac-Moody Lie
algebras}\label{Suqg}

The next class of examples consists of quantum groups, which are also
strict Hopf RTAs. Suppose $C$ is a generalized Cartan matrix (GCM)
corresponding to a symmetrizable Kac-Moody Lie algebra $\lie{g} =
\lie{g}(C)$. Our goal is to construct a family of quantum algebras
associated to the generalized Cartan matrix $C$, which we term
\textit{extended quantum groups}. Examples of such algebras include
quantum groups that use neither the co-root lattice $\qhat$ nor the
co-weight lattice $\phat$, but some intermediate lattice, as well as
possible torsion elements. Moreover, we also study conditions under which
all of these algebras satisfy the various Conditions (S).

\subsection{The construction and triangular decomposition}

To define the aforementioned family of quantum groups, some notation is
required. Recall that a GCM is a matrix $C = (c_{ij})_{i,j \in I}$ where
$I$ is finite, $c_{ii} = 2$, $c_{ij}$ is a nonpositive integer for all $i
\neq j \in I$, and $c_{ij} = 0$ if and only if $c_{ji} = 0$. We say that
$C$ is symmetrizable if there exist positive integers $d_i$ such that
$d_i c_{ij} = d_j c_{ji}$ for all $i,j \in I$. We will also use the
\textit{Gaussian integers and binomial coefficients} in the ground field
$\F$: given $q \in \F^\times$ that is not a root of unity, and integers
$0 \leq m \leq n$, define
\[ [n]_q := \frac{q^n - q^{-n}}{q-q^{-1}}, \qquad [n]_q! := \prod_{m=1}^n
[m]_q, \qquad [0]_q! := 1, \ \qquad \binom{n}{m}_q :=
\frac{[n]_q!}{[m]_q! [n-m]_q!}. \]

\begin{definition}
Fix a ground field $\F$ and a nonzero scalar $q \in \F^\times$ that is
not a root of unity.
\begin{enumerate}
\item In this section, an \emph{extended Cartan datum} consists of the
following data:
\begin{itemize}
\item A symmetrizable GCM $C := (c_{ij})_{i,j \in I}$ and a diagonal
matrix $D$ with positive integer diagonal entries $d_i$ such that $d_i
c_{ij} = d_j c_{ji}$.
\item A free abelian group $\qhat \cong \Z^I$ with $\Z$-basis $\{ K_i : i
\in I \}$. (This is the ``co-root lattice" inside $\lie{h}$, in the
symmetrizable Kac-Moody Lie algebra $\lie{g} = \lie{g}(C)$.)
\item An abelian group $\G \supset \qhat$, as well as a finite set of
characters $\Delta' := \{ \nu_i : \G \to \F^\times : i \in I \}$ such
that $\nu_i|_{\qhat} = q^{\alpha_i}$. In other words, $\nu_i(K_j) =
q^{d_j c_{ji}} = \nu_j(K_i)$ for all $i,j \in I$.
\end{itemize}

\item Given an extended Cartan datum $(C, D, \qhat \subset \Gamma,
\Delta' = \{ \nu_i \})$, define the \emph{extended quantum group}
$\U{\Gamma,\Delta'}$ to be the $\F$-algebra generated by $\G$ and $\{
e_i, f_i : i \in I \}$, modulo the following relations:
\begin{align*}
& g e_i g^{-1} = \nu_i(g) e_i, \quad g f_i g^{-1} = \nu_i(g^{-1}) f_i
\quad \forall i \in I, \ g \in \Gamma; \qquad [e_i, f_j] = \delta_{i,j}
\frac{K_i - K_i^{-1}}{q^{d_i} - q^{-d_i}},\\
& \sum_{l=0}^{1 - c_{ij}} (-1)^l \binom{1 - c_{ij}}{l}_{q^{d_i}} e_i^{1 -
c_{ij} - l} e_j e_i^l = 0,\qquad \mbox{($q$-Serre-1)}\\
& \sum_{l=0}^{1 - c_{ij}} (-1)^l \binom{1 - c_{ij}}{l}_{q^{d_i}} f_i^{1 -
c_{ij} - l} f_j f_i^l = 0.\qquad \mbox{($q$-Serre-2)}
\end{align*}

\item Define $B^\pm$ to be the subalgebras of $\U{\G,\Delta'}$ generated
by the $e_i$s and $f_i$s respectively and $H_1 = H_0 := \F \Gamma$.
\end{enumerate}
\end{definition}

\begin{remark}
Extended quantum groups can be defined for $\Gamma$ any intermediate
lattice between $\qhat$ and $\phat$. For instance, for $\Gamma = \qhat$
we recover the usual quantum group $U_q(\lie{g}) = \U{\qhat, \{
q^{\alpha_i} : i \in I \}}$. This is the approach followed in \cite[\S
4.2]{Ja2} (when $A$ is of finite type). In what follows, we will freely
identify $\alpha_i$ with $q^{\alpha_i}$, since we only deal with quantum
groups and $q$ is not a root of unity.

On the other hand, \cite[Section 3.1]{HK} or \cite[Section 3.2.10]{Jos}
work with $\Gamma = \phat$, the co-weight lattice inside $\lie{h}$, and
$\nu_i(q^h) = q^{\alpha_i(h)}$, for the simple roots $\alpha_i \in
\lie{h}^*$. Moreover, $K_i = q^{d_i h_i}$, where $h_i = [e_i, f_i]$ in $U
\lie{g}$. Note that all of these algebras are special cases of extended
quantum groups.

In fact the family of extended quantum groups is more general than the
above examples, because $\G$ is allowed to have torsion elements, in
which case it does not embed into $\Q \otimes_{\Z} \qhat \subset
\lie{h}$. Thus there may not exist a bilinear form (and hence, a Hopf
pairing) on $\G$, as is used in the literature.
\end{remark}

We now list some basic properties of extended quantum groups.

\begin{prop}\label{Pgcm-qg}
Fix an extended Cartan datum and define $\U{\Gamma,\Delta'}$ as above.
\begin{enumerate}
\item $\U{\Gamma,\Delta'}$ has a Hopf algebra structure, with the
comultiplication $\Delta$, counit $\vi$, and antipode $S$ given on
generators by
\begin{align*}
\vi(g) = &\ 1, \qquad \vi(e_i) = \vi(f_i) = 0,
\qquad \forall g \in \Gamma,\ i \in I\\
\Delta(g) = &\ g \otimes g,
\qquad \Delta(e_i) = e_i \otimes K_i^{-1} + 1 \otimes e_i,
\qquad \Delta(f_i) = f_i \otimes 1 + K_i \otimes f_i,\\
S(g) = &\ g^{-1}, \qquad S(e_i) = - e_i K_i, \qquad S(f_i) = - K_i^{-1}
f_i.
\end{align*}

\item $\U{\Gamma,\nu}$ has an involution $T$ that sends $g \in \Gamma$ to
$g^{-1}$ and $e_i$ to $f_i$ for all $i \in I$. Restricted to $B^+$, $T$
is an algebra isomorphism onto $B^-$.

\item $TST = S^{-1} \neq S$, whence $ST \neq TS$ are anti-involutions $i$
on $\U{\Gamma,\nu}$, which restrict to the identity on $H_1$.

\item $A = \U{\Gamma,\Delta'}$, together with the data $(B^\pm, H_1 = H_0
= \F \Gamma, \Delta', i = ST \text{ or } TS)$ forms a strict, based Hopf
RTA (of finite rank) if and only if $\Delta'$ is $\Z$-linearly
independent in $\widehat{H_1}$.
\end{enumerate}
\end{prop}

\begin{proof}
The first three parts are shown by adopting the proofs and arguments
found in \cite[Section 3.1]{HK} to $\U{\G,\Delta'}$.
Since both $TST$ and $S^{-1}$ are $\F$-algebra anti-automorphisms, the
third part follows by checking that they agree on generators. For the
fourth part, one implication is immediate from the axioms, and the
converse is not hard to verify when $\Delta'$ is $\Z$-linearly
independent.
\end{proof}

Extended quantum groups are very similar in structure to the quantum
groups that have been very well-studied in the literature.
In some sense, they ``quantize" the contragredient Lie algebras defined
in Example \ref{Ekk} (i.e., in \cite{KK}), after removing some of the
assumptions therein.
Thus the classical limit and representation theory (at least, for
integrable modules) should be similar to the traditionally well-studied
cases. We expect that the analysis in \cite[Chapter 3]{HK} should go
through for the algebras $\U{\G,\Delta'}$ as well, but do not proceed
further along these lines, as it is not focus of the present paper.

We now show that extended quantum groups of finite type satisfy all of
the Conditions (S). The following is the main result in this section.

\begin{theorem}\label{Tuqg}
Fix a ground field $\F$ with $\ch \F \neq 2,3$ and such that $\F^\times$
is a divisible group (e.g., $\F = \overline{\F}$). Also suppose $q \in
\F^\times$ is not a root of unity, the matrix $C$ is of finite type, and
$[\G : \qhat] < \infty$.

Then there exist extensions $\nu_i$ of the characters $q^{\alpha_i}$ from
$\qhat$ to $\G$. For each such choice $\Delta' = \{ \nu_i : i \in I \}$
of extensions, $A = \U{\G,\Delta'}$ satisfies Condition (S4), and hence
all other conditions (S). In particular, all blocks of $\calo =
\calo[\hofree]$ are highest weight categories with BGG reciprocity.
\end{theorem}

In particular, all Conditions (S) (and properties such as BGG
Reciprocity) hold for all extended quantum groups with $\Gamma$ ``in
between" the co-root and co-weight lattices, or containing additional
finite-order torsion subgroups. We remark that a special case of our
result was known for $\Gamma = \phat$ from \cite[Lemma 8.3.2]{Jos}, which
stated that $\chi_\mu = \chi_\lambda$ on the center of $\U{\phat, \{
q^{\alpha_i} \}}$ if and only if $\mu \in (W \ltimes (\Z / 2 \Z)^I)
\bullet \lambda$.

\begin{proof}
As the proof is somewhat lengthy, we break it up in to steps for ease of
exposition.\medskip

\noindent {\bf Step 1.}
We first extend the characters $q^{\alpha_i}$ from $\qhat$ to $\Gamma$.
Consider the short exact sequence
$0 \to \qhat \mapdef{\iota} \G \to \G / \qhat \to 0$
in the category of abelian groups. Since $\F^\times$ is divisible --
i.e., injective -- this yields:
\begin{equation}\label{E2}
0 \to \hhh_{group}(\G / \qhat , \F^\times) \to \hhh_{group}(\G ,
\F^\times) \mapdef{\iota^*} \hhh_{group}(\qhat , \F^\times) \to 0.
\end{equation}

\noindent Now think of the simple roots $\alpha_i$ as elements of
$\hhh_{group}(\qhat, \F^\times)$, via:
\[ \alpha_i(q^h) := q^{\alpha_i(h)}. \]

\noindent Note that the subgroup generated by the $\alpha_i$ is free
because $q$ is not a root of unity in $\F$. It is then possible to lift
$q^{\alpha_i}$, via the injectivity of $\F^\times$, to any $\nu_i \in
(\iota^*)^{-1}(q^{\alpha_i}) \subset \hhh_{group}(\G ,
\F^\times)$.\medskip

\noindent {\bf Step 2.}
The next claim is that \textit{if $\qhat \subset \Gamma' \subset \Gamma$
are abelian groups with $\Delta' := \{ \nu_i : i \in I \} \subset
\hhh_{group}(\Gamma, \F^\times)$ being $\Z$-linearly independent
characters when restricted to $\Gamma'$, then}
\begin{equation}\label{Eqcenters}
Z(\U{\Gamma',\Delta'|_{\Gamma'}}) = Z(\U{\Gamma,\Delta'}) \cap
\U{\Gamma',\Delta'|_{\Gamma'}}.
\end{equation}

Indeed, the only nontrivial assertion in Equation \eqref{Eqcenters} is to
show that $Z(\U{\Gamma',\Delta'|_{\Gamma'}}) \subset
Z(\U{\Gamma,\Delta'})$. Suppose $z \in Z(\U{\Gamma',\Delta'_{\Gamma'}})$;
since $z$ commutes with $\G'$, it has weight $0$ in
$\U{\G',\Delta'|_{\G'}}$, and hence also in $\U{\G,\Delta'}$ (since the
weight space decompositions of $\U{\G',\Delta'}$ $\hookrightarrow
\U{\G,\Delta'}$ agree). Thus, $z$ commutes with $\G$, and since it
commutes with each $e_i$ and $f_i$, $z$ is central in $\U{\G,\Delta'}$ as
well.\medskip

\noindent {\bf Step 3.}
For convenience, define $\hatt{G} := \hhh_{group}(G, \F^\times)$, for any
group $G$. Thus $\hatt{\G} = \widehat{H_1}$ in our setting.
Now to prove the result, fix $\la \in \hatt{\Gamma}$ and suppose
$\chi_\mu = \chi_\la : Z(\U{\G, \Delta'}) \to \F$ for some $\mu : \Gamma
\to \F^\times$. Then $\chi_\mu, \chi_\lambda$ agree when restricted (by
the previous step) to $Z := Z(U_q(\lie{g}))$, where $U_q(\lie{g}) =
\U{\qhat, \{ q^{\alpha_i} \}}$. Thus $\mu \circ \xi = \la \circ \xi$ on
$Z$. Now recall the following result from \cite[Sections 4.2 and
6.25-6.26]{Ja2}:
\textit{If $\G = \qhat$ and $\nu_i = \alpha_i$ are the ``simple roots",
then the Harish-Chandra map is an isomorphism}
\[ \rho_{H_1}(q^\theta) \circ \xi : Z(U_q(\lie{g})) \mapdef{\sim}
\F[\qhat \cap 2 \phat]^W. \]

\noindent Here, $\theta$ denotes the half-sum of positive roots, and
$\rho_{H_1}$ is the weight-to-root map that was studied in Proposition
\ref{Prho}.
(We identify $\lie{h} \leftrightarrow \lie{h}^*$ via the Killing form.)
It follows from above that $\mu \circ \rho_{H_1}(q^{-\theta}) = \la \circ
\rho_{H_1}(q^{-\theta})$ on $\F[\qhat \cap 2 \phat]^W$.\medskip

\noindent {\bf Step 4.}
The remainder of the proof studies the chain of algebras $\F[\qhat \cap 2
\phat]^W \hookrightarrow \F[\qhat \cap 2 \phat] \hookrightarrow \F[\qhat]
\hookrightarrow \F[\Gamma]$. Note that $\qhat \cap 2 \phat$ is a lattice,
so $\Spec \F[\qhat \cap 2 \phat] = (\F^\times)^{rk(\qhat \cap 2 \phat)}$.
Now recall the \textit{Nagata-Mumford Theorem} from (a special case of)
\cite[Theorem 5.3]{Muk}:
\textit{Suppose a finite group $W$ acts on an affine variety $X$ (i.e., its
coordinate ring $R$). Then the map $\Phi : X = \Spec(R) \to X /\!\!/ W :=
\Spec(R^{W})$ (induced by the inclusion $R^W \hookrightarrow R$) is a
surjection that factors through a bijection $\Phi : X / W \to X /\!\!/
W$, where $X/W$ denotes the $W$-orbits in $X$.}\smallskip

Applying this to $X := \qhat \cap 2 \phat$, it follows that the set of
possible extensions $\nu \in \hatt{(\qhat \cap 2 \phat)}$ of $\la \circ
\rho_{H_1}(q^{-\theta}) : \F[\qhat \cap 2 \phat]^W \to \F$ is a
$W$-orbit, hence finite (thus, $\{ \nu \circ \rho_{H_1}(q^{-\theta}) \}$
is also finite).\medskip

\noindent {\bf Step 5.}
Finally, consider the map $: \hatt{\Gamma} \to \hatt{(\qhat \cap 2
\phat)}$. By the injectivity of $\F^\times$ and an analogue of Equation
\eqref{E2} in this situation, it suffices to show that $\G / (\qhat \cap
2 \phat)$ is finite (for then $\iota^*$ is a surjection with finite
fibers). But $\Gamma / (\qhat \cap 2 \phat)$ is indeed finite, since
\[ [\Gamma : \qhat \cap 2 \phat] = [\Gamma : \qhat] \cdot [\qhat : \qhat
\cap 2 \phat] \leq [\Gamma : \qhat] \cdot [\phat : 2 \phat] < \infty. \]

To conclude, $\{ \mu \in \hatt{\Gamma} = \widehat{H_1} : \chi_\mu =
\chi_\la \} \subset \{ \mu \in \hatt{\Gamma} : \mu \circ
\rho_{H_1}(q^{-\theta}) = \la \circ \rho_{H_1}(q^{-\theta})$ on $(\qhat
\cap 2 \phat)^W \}$, and the latter is a finite set by the above
analysis. Thus $\U{\Gamma,\Delta'}$ satisfies Condition (S4).
\end{proof}
%}}}

%{{{1 Section 7 - Further examples of strict, based Hopf RTAs
\section{Further examples of strict, based Hopf RTAs}\label{Smore}

Before moving on to RTAs that are either not Hopf RTAs or not strict, we
write down some more examples of strict, based Hopf RTAs of low rank. The
first of these examples shows the need to use Condition (S3) instead of
central characters/Condition (S4) in order to obtain a block
decomposition of $\calo$.

\begin{exam}[{\em Rank one infinitesimal Hecke algebras and their
quantized analogues}]\label{Einfhecke1}

Suppose $\ch \F = 0$. The (Lie) rank one infinitesimal Hecke algebra is
defined to be a deformation $\mathcal{H}_z$ of $\mathcal{H}_0 :=
U(\lie{sl}_2(\F) \ltimes \F^2)$, where $\F^2$ is spanned by a weight
basis $x,y$ (over the Cartan subalgebra of $\lie{sl}_2$, which is spanned
by $h$). The deformed relation is $[x,y] = z(C)$, where $C$ is the
quadratic Casimir element of $U(\lie{sl}_2)$ and $z \in \F[T]$ is an
arbitrary polynomial.

The family of algebras $\mathcal{H}_z$ was introduced in \cite{Kh} and
extensively studied in \cite{KT}. It can be seen from
\textit{loc.~cit.}~that $\mathcal{H}_z$ is a strict, based Hopf RTA of
rank one with $H_1 = H_0 = \F[h]$ and $\Delta' = \{ \frac{1}{2} \alpha
\}$, where $\alpha$ is the root of $\lie{sl}_2$. (We remind the reader
that in the literature, \textit{roots} of semisimple Lie algebras are
assumed to lie in the dual space $\lie{h}^*$ of the Cartan Lie
subalgebra, via the weight-to-root map $\rho_{U(\lie{h})}$.) In
particular, $\hofree = \widehat{H_1} = \F$. In \cite{KT}, it is also
shown that similar to complex semisimple Lie algebras (e.g.,
$U(\lie{sl}_2)$),
\begin{itemize}
\item The center $Z(A)$ is isomorphic to a polynomial algebra in one
variable -- the ``quadratic" Casimir element.

\item Condition (S4) holds for $\mathcal{H}_z$ if $z \neq 0$. (Thus
$\calo$ satisfies BGG Reciprocity.)

\item Every central character is of the form $\chi_\lambda$ for some
$\lambda \in \widehat{H_1}$, if $\F$ is algebraically closed of
characteristic zero (see \cite[Exercise (23.9)]{H}).

\item If $z \neq 0$, there are at most finitely many pairwise
non-isomorphic simple finite-dimensional objects in $\calo$.
\end{itemize}

The algebras $\mathcal{H}_z$ possess quantizations $\mathcal{H}_{z,q}$
for $q \neq 0, \pm 1$, which were explored in detail in \cite{GGK}.
The quantum algebras $\mathcal{H}_{z,q}$ turn out to be deformations of
$U_q(\lie{sl}_2(\F)) \ltimes \F[x,y]$ whose classical limits as $q \to 1$
are once again $\mathcal{H}_z$; see \cite{GGK}. They have been found to
possess very similar properties to $\mathcal{H}_z$, including a strict,
based Hopf RTA structure.
However, it was shown in \cite[Theorem 11.1]{GGK} that if $q$ is not a
root of unity, and $z = qyx-xy \neq 0$, then $Z(\mathcal{H}_{z,q}) = \F$.
Thus Condition (S4) clearly fails. Nevertheless, \cite[Propositions 8.2,
8.13]{GGK} show that $\calo = \calo[\hofree]$ is a highest weight
category satisfying Condition (S3). Thus our framework allows us to prove
that $\calo$ is a direct sum of blocks with BGG Reciprocity, even though
it has trivial center and Condition (S4) fails to hold.
\end{exam}

\begin{exam}\label{Eyamane}
The next example is that of a strict Hopf RTA that was recently studied
by Batra and Yamane \cite{BY}. In that work, the authors defined
``generalized quantum groups" $U(\chi, \Pi)$, which are a family of
quantum algebras corresponding to a semisimple Lie algebra (akin to the
algebras $\U{\Gamma, \Delta'}$). The (skew) centers of these algebras and
Harish-Chandra type results were studied in \textit{loc.~cit.} We observe
here that the algebra $U(\chi, \Pi)$ is a strict, based Hopf RTA when
$\chi$ is non-degenerate and $\chi(\alpha_i, \alpha_j)$ is not a root of
unity for any $i,j \in I$.
\end{exam}

The following example is a degenerate one.

\begin{exam}[{\em Regular functions on affine algebraic
groups}]\label{Edeg}

It is well-known that the category of commutative Hopf algebras is dual
to the category of affine algebraic groups. Thus if {\bf G} is any affine
algebraic group, then $H_1 = H_0 = \C[\bf G]$ is a commutative Hopf
algebra, and hence $A = H_1$ is a strict, based HRTA as well, with
$\Delta'$ the empty set. Note that $H_1$ need not be cocommutative in
general (since ${\bf G}$ need not be commutative).

In general, every commutative (Hopf) $\F$-algebra $H_1$ is a strict,
based (Hopf) RTA of rank zero, via: $H_1 = H_0 = \F \otimes H_1 \otimes
\F = Z(H_1)$. In this context, $\calo = \calo[\hofree]$ trivially
satisfies Condition (S4) (and hence Conditions (S1)--(S3)), and also is a
semisimple (highest weight) category.
\end{exam}

The final example in this section is stated for completeness, and is
illustrative in showing how to combine both of the main theorems in
Section \ref{Sos}, in order to study Category $\calo$. (More generally,
one can use Theorem \ref{Tfunct} to create more examples of (strict)
(based) (Hopf) RTAs by taking tensor products.)

\begin{exam}\label{Ezhi}
In \cite{Zhi}, Zhixiang studied the homological properties and
representations of the ``double loop quantum enveloping algebra" (DLQEA),
which is a Hopf algebra isomorphic to $U_q(\lie{sl}_2) \otimes \F[g^{\pm
1}, h^{\pm 1}]$ as an algebra (but not as Hopf algebras). Here $\F$ is an
algebraically closed field of characteristic zero. The aforementioned
algebra isomorphism and Theorem \ref{Tfunct} shows that the DLQEA is a
strict, based HRTA of rank one, and the representation theory of Category
$\calo$ reduces to that for $U_q(\lie{sl}_2)$ and for $\F[g^{\pm 1},
h^{\pm 1}]$. Now use Theorems \ref{Tfirst} and \ref{Tfunct}, Example
\ref{Edeg}, and the results in Section \ref{Suqg} to conclude that the
DLQEA satisfies Condition (S4), and hence, Theorem \ref{Tfirst}.
\end{exam}
%}}}

\section{Rank one RTAs: Triangular Generalized Weyl Algebras}\label{Sgwa}

In the remainder of this paper we discuss more families of based RTAs,
some of which are either not Hopf RTAs, or not strict. These examples
further demonstrate the need to use the full power of our framework. All
of the examples in this section and the next fall under the following
setting.

\begin{definition}\label{Dgwa}
Fix a field $\F$, an associative $\F$-algebra $H$, an $\F$-algebra map
$\theta : H \to H$, and $z_0, z_1 \in H$. The {\em triangular generalized
Weyl algebra (or triangular GWA)} associated to this data is the
$\F$-algebra
\begin{equation}
\cals(H,\theta,z_0,z_1) := H \tangle{d,u} / (uh = \theta(h) u,\ hd = d
\theta(h),\ ud = z_0 + d z_1 u\ \forall h \in H).
\end{equation}
\end{definition}

As we will presently see, this construction is very general and
incorporates a large number of algebras studied in the literature. We now
briefly list the contents of this section. In Section \ref{Sgwa-basic} we
discuss the structure and representation theory of $\calo$ for triangular
GWAs. Sections \ref{Sgwa-example1} and \ref{Sgwa-example2} discuss a
large number of examples of triangular GWAs, many of them arising from
mathematical physics. The examples are of two flavors - ``classical" and
``quantum". In Section \ref{Sgwa-qdef} we explain how these two types of
examples are related in a precise way. Our construction of the
``classical limit" extends -- to a large family of generalized down-up
algebras -- the relation between classical and quantum $\lie{sl}_2$.

%{{{1 Section 8.1 - Structure and block decomposition of $\calo$ over triangular GWAs
\subsection{Structure and block decomposition of $\calo$ over triangular
GWAs}\label{Sgwa-basic}

Henceforth we will assume that $\theta$ is an automorphism, as well as
some other properties that we now discuss.

\begin{lemma}\label{Lgwa}
Suppose $\theta : H \to H$ is an automorphism. Then
$\cals(H,\theta,z_0,z_1)$ satisfies (RTA1) with $B^+ = \F[u]$, $B^- =
\F[d]$, and $H_1 = H$, if and only if $z_0, z_1$ are central in $H$.
\end{lemma}

\begin{proof}
Compute for all $h \in H$:
\begin{align*}
h(ud) = &\ h(z_0 + d z_1 u) = h z_0 + d \theta(h) z_1 u,\\
(hu)d = &\ u \theta^{-1}(h) d = (ud)h = (z_0 + d z_1 u)h = z_0 h + d z_1
\theta(h) u.
\end{align*}

\noindent Now if $\cals(H,\theta,z_0,z_1)$ satisfies (RTA1), then the
equality between these two expressions for all $h \in H$ implies that
$z_0, z_1 \in Z(H)$. To show the converse, use the Diamond Lemma from
\cite{Be} in a manner similar to the proof of Theorem \ref{Trtm}.
%%%{\color{red}
Namely, define the usual set of generators $X = \{ u,d, h_i \}$, where
$\{ h_i : i \in I \}$ range over an $\F$-basis of $H$, with a fixed
element $0 \in I$ corresponding to $h_0 = 1_H$. Also fix a total ordering
of $I$ -- and hence of $\{ h_i \}$ -- in which $0 = \min I$. Now define a
semigroup partial order on the free monoid $\tangle{X}$ generated by $X$,
via: words of longer length are larger, $u > h_i > d\ \forall i$, and now
extend both these to the lexicographic order on words of the same length.

Then the reductions are: $h_i h_j$ reduces via the structure constants
for multiplication in $H$, $ud \mapsto z_0 + d z_1 u,\ uh \mapsto \theta(h)
u$, and $hd \mapsto \theta(h) d$.
These are clearly compatible with the semigroup partial order.
Moreover, given $w = T_1 \cdots T_n \in \tangle{X}$, one checks that the
function $f(w) = n + \# \{ (i,j) : i < j,\ T_i > T_j \in X \}$ is a
misordering index (i.e., it strictly reduces with each reduction).

To now use the Diamond Lemma, note that we only have overlap (minimal)
ambiguities -- and $h h' h'', u h h', h h' d$ are resolved using the
relations in the associative $\F$-algebra $H$.
We now (informally) apply our reductions to the only the ``nontrivial"
ambiguity $uhd$, using also that $z_0, z_1 \in Z(H)$:
\begin{align*}
(uh)d \quad \mapsto & \quad \theta(h) ud \quad \mapsto \quad \theta(h)
(z_0 + d z_1 u) \quad \mapsto \quad \theta(h) z_0 + d \theta^2(h) z_1
u,\\
u(hd) \quad \mapsto & \quad ud \theta(h) \quad \mapsto \quad (z_0 + d z_1
u) \theta(h) \quad \mapsto \quad z_0 \theta(h) + d z_1 \theta^2(h) u.
\end{align*}

\noindent Since $z_0, z_1$ are central, the ambiguity is resolvable and
the deformation is flat (i.e., $\cals(H,\theta,z_0,z_1)$ satisfies
(RTA1)) by the Diamond Lemma.
%%%}
\end{proof}

\begin{assum}\label{Agwa}
For the remainder of this section and the next, assume that $H$ is
commutative, and $\theta$ is an algebra automorphism of $H$ of infinite
order.
\end{assum}

In order to discuss the structure of triangular GWAs, we now introduce a
sequence $\z_n$ of distinguished elements in a triangular GWA (more
precisely, in its subalgebra $H$).

\begin{definition}\label{Delements}
Suppose $\theta : H \to H$ is an algebra automorphism. Given $n \in \N$,
define
\begin{equation}
z'_n := \prod_{i=0}^{n-1} \theta^i(z_1), \qquad z'_0 := 1, \qquad
\z_n := \sum_{j=0}^{n-1} \theta^j(z_0 z'_{n-1-j}), \qquad \z_0 := 0,
\qquad \z_{-n} := \theta^{-n}(\z_n).
\end{equation}

\noindent Now given a weight $\lambda : H \to \F$, define
$[\lambda] := \{ \theta^{-n} * \lambda : n \in \Z, \lambda(\z_n) = 0 \}
\subset \widehat{H}$.
\end{definition}

We now discuss some basic properties of triangular GWAs, which concern
central characters and the block decomposition of $\calo$.

\begin{theorem}\label{Tgwa}
Suppose $A = \cals(H, \theta, z_0, z_1)$ is a triangular GWA (over any
field $\F$). Then $A$ is a strict, based RTA of rank one (with $\Delta :=
\{ \theta \}$ and $H_1 = H_0 := H$) if and only if $A$ satisfies
Assumption \ref{Agwa}.

Suppose henceforth that the triangular GWA $A$ is a strict, based RTA of
rank one.
\begin{enumerate}
\item For all $m,n \geq 0$, the Shapovalov form of $\F[d]$ is given by
$\tangle{d^m, d^n} = \delta_{m,n} \prod_{j=1}^n \z_j$.

\item $S^3(\lambda) = [\lambda]$ for all $\lambda \in \hfree$.

\item Suppose $z_1 = 1$ and $z_0 \in \im(\id_H - \theta)$. Define a
\emph{quadratic Casimir operator} to be $\Omega := du + \zeta$ for any
$\zeta \in H$ satisfying: $(\id_H - \theta)(\zeta) = z_0$. Then,
\[ Z(A) \cap H = \ker(\id_H - \theta), \qquad Z(A) = (Z(A) \cap
H)[\Omega], \qquad S^4(\lambda) \cap (\Z \theta * \lambda) = S^3(\lambda)
= [\lambda]\ \forall \lambda \in \hfree. \]

\noindent Moreover, $\Omega$ is transcendental over $Z(A) \cap H$.
\end{enumerate}
\end{theorem}

\noindent Consequently, $A$ satisfies Condition (S3) if and only if
$|[\lambda]| < \infty$ for all $\lambda \in \hfree$. The last part also
says that the converse to Lemma \ref{Lcs} holds for triangular GWAs when
$z_0 = 1$ and a quadratic Casimir exists.

\begin{proof}
Set $B^+ := \F[u]$, $B^- := \F[d]$, $H_1 = H_0 := H$, and $\Delta := \{
\theta \}$.
Now the first assertion is not hard to show, using the anti-involution
that sends $u$ to $d$ and fixes $H$. To show the next result, we prove
some intermediate equivalences that may be useful in their own right.
First, specializing the analysis in Section \ref{Sverma} to $A$ helps
determine the structure of Verma modules:\medskip

\noindent \textit{For all weights $\mu \in \hfree$, $M(\mu)$ is a
uniserial module, with unique composition series:}
\[ M(\mu) \supset M(\theta^{-n_1} * \mu) \supset M(\theta^{-n_2} * \mu)
\supset \cdots, \]

\noindent \textit{where $0 < n_1 \leq n_2 \leq \cdots$ comprise the set
$\{ n \in \N : \mu(\z_n) = 0 \}$. Thus $\calo$ is finite length if and
only if $[\mu] \cap (-\nn \Delta * \mu)$ is finite for every $\mu \in
\hfree$. Moreover, the following are equivalent, given $n \in \nn$ and
$\mu \in \widehat{H}$:}
(a) \textit{The multiplicity $[M(\theta^n * \mu) : L(\mu)]$ is
nonzero.}
(b) $[M(\theta^n * \mu) : L(\mu)] = 1$.
(c) $(\theta^n * \mu)(\z_n) = 0$.
(d) $\mu(\z_{-n}) = 0$.\medskip

\noindent The proof is straightforward, given that $M(\lambda) \cong
\F[d]$ for all $\lambda$, and $d^n m_\lambda$ spans
$M(\lambda)_{\theta^{-n} * \lambda}$ for all $\lambda \in \hfree$ and $n
\geq 0$. The key computation, which is straightforward but longwinded,
is to show:
\begin{equation}\label{Egwa}
u^m d^n \ \in \ d^{n-m} \cdot \prod_{j=n-m}^{n-1} \z_{j+1} + A \cdot u,
\qquad \forall 0 \leq m \leq n.
\end{equation}

\noindent Setting $m=1$ and applying \eqref{Egwa} to the highest weight
vector of $M(\theta^n * \mu)$ shows that (a) $\Leftrightarrow$ (c).
The remaining equivalences are standard. We now sketch the proofs of the
three assertions. The first part follows using Equation \eqref{Egwa}.
Next, that $S^3(\lambda) = [\lambda]$ can be proved using the
equivalences stated above.

It remains to prove part (3) about the center. That $Z(A) \cap H = \ker
(\id_H - \theta)$ is easily verified. Now suppose $\omega \in Z(A)$ is
central. Then $\omega$ commutes with $H$, whence $\omega \in H[du] =
H[\Omega - \zeta]$. Consider such a central element $\omega := \sum_i
(du)^i h_i$, where $h_i \in H\ \forall i \geq 0$. We then have
\[ \omega = \sum_{i \geq 0} (\Omega - \zeta)^i h_i = \sum_{0 \leq j \leq
i} \binom{i}{j} \Omega^j \zeta^{i-j} h_i = \sum_{j \geq 0} \Omega^j
\sum_{i \geq j} \binom{i}{j} \zeta^{i-j} h_i = \sum_{j \geq 0} \Omega^j
h'_j, \]

\noindent where $h'_j \in H\ \forall j$. Now if $\omega$ is central, we
compute: $0 = [u, \omega] = \sum_j \Omega^j [u,h'_j]$,
whence by the PBW property (RTA1), one checks that $[u,h'_j] = 0\ \forall
j$, whence $h'_j \in \ker(\id_H - \theta)$ from above. Thus $\omega \in
(Z(A) \cap H)[\Omega]$ as claimed. Additionally, it is not hard to see
using (RTA1) that $\Omega$ is transcendental over $Z(A) \cap H$.

Finally, by a previous part and Lemma \ref{Lcs}, it suffices to show that
$S^4(\lambda) \cap (\Z \theta * \lambda) \subset S^3(\lambda)$ for all
$\lambda \in \hfree$. Moreover, it further suffices to show the
\textbf{claim} that $\chi_{\theta^{-n} * \lambda} \equiv \chi_\lambda$
for some $n \geq 0$ if and only if $[M(\lambda) : L(\theta^{-n} *
\lambda)] > 0$. By the proof of Theorem \ref{Tgwa}, this is equivalent to
showing that $\lambda(\z_{n}) = 0$. Now compute using any quadratic
Casimir element and Proposition
\ref{Pbasic}:
\[ \chi_\lambda(\Omega) - \chi_{\theta^{-n} * \lambda}(\Omega) =
\lambda(\zeta) - \lambda(\theta^n(\zeta)) = \lambda \circ (\id_H -
\theta^n)(\zeta) = \lambda \circ (\id_H + \theta + \cdots +
\theta^{n-1})(z_0) = \lambda(\z_n), \]

\noindent since $z_1 = 1$. Thus the above claim follows, completing the
proof.
\end{proof}
%}}}

%{{{1 Section 8.2 - Examples: generalized down-up algebras
\subsection{Examples: generalized down-up algebras}\label{Sgwa-example1}

We now discuss a family of examples of triangular GWAs, which has been
extensively studied in many papers in the literature. These are the
``generalized down-up algebras" introduced by Cassidy and Shelton in
\cite{CS}, and they are strict, based RTAs of rank one, with
\begin{equation}\label{Edownup}
H = \F[h], \qquad \theta = \theta_{r,\gamma}(h) := r^{-1}(h + \gamma),
\qquad z_1 = s^{-1}, \qquad z_0 = s^{-1} f(h),
\end{equation}

\noindent where $r,s \in \F^\times$, $\gamma \in \F$, and $f(h) \in H$ is
a fixed polynomial in $h$. (Note that if $r=1$ then $\cals(\F[h],
\theta_{1,\gamma}, s^{-1} f(h), s^{-1})$ is a strict, based Hopf RTA of
rank one.) The operators $d,u$ in $\cals(\F[h], \theta_{r,\gamma}, s^{-1}
f(h), s^{-1})$ are thought of as ``lowering" and ``raising" operators
respectively (hence the name of ``down-up" algebras). Examples of such
algebras occur in many different settings in the literature:
\begin{enumerate}
\item In representation theory, Smith \cite{Smi} introduced and studied a
family of triangular GWAs (more precisely, generalized down-up algebras)
that are deformations of $U(\lie{sl}_2)$. Smith showed that these
algebras satisfy Condition (S4), as well as an analogue of Duflo's
theorem for primitive ideals and annihilators of simple modules
$L(\lambda)$.

\item In mathematical physics, Witten \cite{Wi} introduced a 7-parameter
family of deformations of $U(\lie{sl}_2)$ that include a large sub-family
of GWAs. Witten's motivations arose from vertex models and duality in
conformal field theory. Witten's family of deformations was later studied
by Kulkarni \cite{Ku1}, and a three-parameter subfamily
$U_{abc}(\lie{sl}_2)$ was studied by Le Bruyn \cite{LeB} under the name
of ``conformal $\lie{sl}_2$-algebras".

\item In the comprehensive paper \cite{Kac1} studying Lie superalgebras,
Kac studied the ``dispin Lie superalgebra" $B[0,1]$. In this case,
\[ U(B[0,1]) = \cals(\C[h], \theta = \theta_{1,-1}, h, 1), \qquad
\theta(h) = h-1. \]

\item In \cite{Wo}, Woronowicz introduced and studied the algebra
$\cals(\F[h], \theta, \nu^{-1} h, \nu^{-2})$ in the context of quantum
groups. This algebra is a generalized down-up algebra where $\nu \in \F
\setminus \{ 0, \pm 1 \}$ and $\theta(h) = \nu^{-4} h + 1 + \nu^{-2}$.

\item These algebras also occur in combinatorics, in certain cases when
``down" and ``up" operators are defined on the span of a partially
ordered set. These were the original ``down-up" algebras, studied by
Benkart and Roby in \cite{BR}, and they are a special case of generalized
down-up algebras with $z_0 = h$ and $z_1 \in \F$. They have been the
subject of continuing interest -- see \cite{CM,Jo2,KM,Ku2,LL} among
others.

\item The algebras studied by Jing and Zhang, as discussed in Example
\ref{EJZ}. In this case one can show that $\calo[\hfree]$ satisfies
Condition (S3) if $q$ is not a root of unity and $\ch \F \neq 2,3$.
\end{enumerate}

Note that in a large number of examples mentioned in the above list, the
generalized down-up algebras of interest are described by \eqref{Edownup}
with parameters $r=1, \gamma \neq 0, f \not\equiv 0$, and $\ch \F = 0$.
In such settings it is possible to describe when the algebra satisfies
Condition (S3). Thus the following result deals with block decompositions
of $\calo$, for all of the above examples at once.

\begin{theorem}\label{Tdownup}
Under the setting of \eqref{Edownup}, and identifying the weights
$\lambda_a : h \mapsto a$ of $\F[h]$ with the corresponding scalars $a
\in \F$, we have:
\[ \hfree = \begin{cases}
\F \setminus \{ \gamma r^{-1} / (1 - r^{-1}) \}, & \text{ if } r \notin
\sqrt{1};\\
\F, & \text{ if } \gamma \neq 0 = \ch(\F),\ r = 1;\\
\emptyset, & \text{ otherwise.}
\end{cases} \]

\noindent If $r=s=1$ and $\ch \F = 0$, a quadratic Casimir operator
$\Omega$ always exists, and the center of $A$ is the polynomial algebra
$\F[\Omega]$.

Now suppose $r = 1$ and $\gamma \neq 0$, $f \not\equiv 0$.
\begin{enumerate}
\item If $s=1$, then $[\lambda]$ is finite for one (equivalently, all)
weights $\lambda$ if and only if $\ch \F = 0$.

\item If $s$ is not a root of unity and $\ch \F = 0$, then $[\lambda]$ is
finite for all $\lambda$.
\end{enumerate}
\end{theorem}

\noindent In particular, we conclude via Theorem \ref{Tgwa} that if $\ch
\F = 0$ and part (1) or (2) holds, then $A$ satisfies Condition (S3) and
hence $\calo[\hfree]$ has BGG Reciprocity.

Theorem \ref{Tdownup} and its proof are similar in flavor to a subsequent
result for ``quantum" down-up algebras (see Theorem \ref{Tkleinian}). The
proofs of both of these results are deferred to Section \ref{Slech}.
%}}}

%{{{1 Section 8.3 - Further examples: quantum triangular GWAs
\subsection{Further examples: quantum triangular
GWAs}\label{Sgwa-example2}

Another well-studied and important class of algebras in the literature is
similar in structure and has many properties in common with down-up
algebras. These algebras have a ``quantum" flavor; a prominent example is
$U_q(\lie{sl}_2)$. We now introduce the general notion of a
\textit{quantum triangular GWA}. This is a strict, based Hopf RTA of rank
one, which includes as examples several algebras studied in the
literature, and also resembles generalized down-up algebras.

To define a quantum triangular GWA, suppose $\Gamma$ is an arbitrary
abelian group equipped with a fixed character (or weight) $\alpha :
\Gamma \to \F^\times$, and $H = \F \Gamma$ is its group algebra.
Now define the associated quantum triangular GWA to be
\begin{equation}\label{Ekleinian}
\cals(\G) := \cals(\F \Gamma, \theta = \rho_H(\alpha), z_0, z_1), \quad
z_0, z_1 \in H,
\end{equation}

\noindent where the weight-to-root map $\rho_H$ was studied in
Proposition \ref{Prho}. Note by Lemma \ref{Lgwa} that $\cals(\G)$
satisfies Conditions (RTA1) and (RTA3).
Moreover, quantum triangular GWAs do not fall under the framework of
generalized down-up algebras, since $H$ is now no longer a polynomial
ring but a group algebra. The present work unites these two settings via
triangular GWAs (from Definition \ref{Dgwa}). Moreover, quantum
triangular GWAs encompass many families of quantum algebras studied in
the literature:
\begin{enumerate}
\item \textit{Quantum $\lie{sl}_2$:}
A motivating and fundamental example is $U_q(\lie{sl}_2)$. This is
obtained by setting $\Gamma = \Z$ (more precisely, $\Gamma = K^\Z$ for
some variable $K$), and $\alpha(K) = q^2, z_1 = 1, z_0 = \frac{K -
K^{-1}}{q - q^{-1}}$ for some $q \neq 0, \pm 1$.
More generally, Ji et.~al.~\cite{JWZ} and Tang \cite{Ta2} studied the
quantum triangular GWAs with arbitrary $z_0 \in H = \F[K^{\pm 1}]$.

\item The Drinfeld quantum double of the positive part of
$U_q(\lie{sl}_2)$ is a special case of a family of quantum algebras
studied by Ji et.~al.~\cite{JWY} as well as Tang-Xu \cite{TX}. These
algebras are also quantum triangular GWAs, where
$H = \F[K^{\pm 1},h^{\pm 1}]$ and $\alpha(K) = q^2, \alpha(h) = q^{-2},
z_1 = 1$.

\item \textit{Double loop quantum enveloping algebras:}
This construction was discussed in Example \ref{Ezhi}.

\item \textit{Quantized Weyl algebras:}
This is a degenerate example that we mention for completeness. Namely,
when $H = \F$, $\alpha$ is the (constant) counit map on $\G$, and $z_0 =
1, z_1 \neq 0$, one obtains the quantized Weyl algebras (and in
particular, the first Weyl algebra $A_1$ if $z_1 = 1$).
\end{enumerate}

\begin{remark}
Recall that Crawley-Boevey and Holland studied noncommutative
deformations of Kleinian singularities in \cite{CBH}. These are algebras
associated with finite subgroups of $SL_2(\C)$. In Type $A$, these
algebras are triangular GWAs with $H = \F[\Z / n \Z]$ for $n \in \N$,
together with $z_1 = 1$ and $\alpha(m + n\Z) := \varepsilon^m$ (where
$\varepsilon \in \F^\times$ is a primitive $n$th root of unity). In
general, one replaces $\Z / n \Z$ by a finite subgroup of $SL_2(\F)$. In
contrast, we will work with subgroups $\Gamma$ of the torus $\F^\times
\subset SL_2(\F)$ (which we assume to be infinite in order to obtain a
strict, based Hopf RTA structure).
\end{remark}

\begin{remark}[Ambiskew polynomial rings]
All of the examples discussed above in this section have in common that
$z_1 \in \F$. Triangular GWAs where $z_1 \in \F^\times$ are known as
\textit{ambiskew polynomial rings}.
The study of ambiskew polynomial rings was initiated and carefully
developed by Jordan (see \cite{Jo1,Jo2} for more details). The subject
continues to attract much interest -- see for instance \cite{BrMa,Ha,JW}
and the references therein. We also remark that the level of generality
in defining an ambiskew polynomial ring has varied throughout the
literature. The current -- and most general -- version of an ambiskew
polynomial ring can be found in \cite[Definition 2.2]{JW}.
\end{remark}

We now state a similar result to Theorem \ref{Tdownup} for quantum
triangular GWAs algebras, which characterizes when Condition (S3) holds
for such algebras. In the following result, as in Theorem \ref{Tdownup},
we will assume that $z_1 \in H^\times$ is a unit.

\begin{theorem}\label{Tkleinian}
In the setting of \eqref{Ekleinian}, the orders of $\theta$, $\alpha$,
and $\Gamma / \ker(\alpha) \cong \alpha(\Gamma)$ are either all infinite,
or all equal. Thus $\widehat{H} = \hfree$ if and only if $\hfree$ is
nonempty, if and only if $\alpha(\Gamma) \subset \F^\times$ is infinite.

\noindent Now suppose $z_1 = s \cdot [1_\Gamma] = s \in \F^\times$.
Define $\sqrt{1}$ to be the roots of unity in $\F^\times$, and define
\[ \Gamma_1 := \{ g \in \Gamma : \alpha(g) = s^{-1} \}, \qquad
\Gamma_2 := \{ g \in \Gamma : \alpha(g) s \in \sqrt{1} \setminus \{ 1 \}
\}, \qquad \Gamma_3 := \{ g \in \Gamma : \alpha(g) s \notin \sqrt{1} \}. \]

\noindent Also write $g \in \Gamma$ to denote $[g]$, and write
\[ z_0 = \sum_{g \in \Gamma_1 \cup \Gamma_2} a_g g + \sum_{i,j} a_{ij}
g_{ij} \in \F \Gamma \]

\noindent with $a_g, a_{ij} \in \F$, and $g_{ij} \in \Gamma_3$ such that
$\alpha(g_{ij}^{-1} g_{kl})$ has finite order if and only if $j=l$.
\begin{enumerate}
\item Suppose there exists $\mu \in \hfree$ such that at least one of the
following equations holds:
\begin{align}
&\ \sum_i \frac{a_{ij} \mu(g_{ij})}{1 - (\alpha(g_{ij})s)^{-1}} = 0 =
\sum_{g \in \Gamma_1} a_g \mu(g), \qquad \forall j,\label{ES3kleinian1}\\
\mbox{or} \qquad &\ \sum_i \frac{a_{ij} \alpha(g_{ij}) \mu(g_{ij})}{1 -
(\alpha(g_{ij})s)^{-1}} = 0 = \sum_{g \in \Gamma_1} a_g \mu(g), \qquad
\forall j.\label{ES3kleinian2}
\end{align}

\noindent Then $[\mu]$ is infinite.

\item Conversely, if $\ch \F = 0$ and $[\lambda]$ is infinite for at
least one $\lambda \in \hfree$, then at least one of \eqref{ES3kleinian1}
and \eqref{ES3kleinian2} holds for some $\mu \in \Z \Delta * \lambda = \Z
\theta * \lambda \subset \hfree$.
\end{enumerate}
\end{theorem}

\noindent As in the case of Theorem \ref{Tdownup}, the proof of Theorem
\ref{Tkleinian} is deferred to Section \ref{Slech}. We also observe that
quantum triangular GWAs with $z_1 = 1$ have certain similarities in
structure and center, to symplectic reflection algebras (which were
discussed in \cite{EG,Eti}). We do not elaborate further on this point in
the present paper.
%}}}

%{{{1 Section 8.4 - Quantization of generalized down-up algebras
\subsection{Quantization of generalized down-up
algebras}\label{Sgwa-qdef}

We now describe a concrete connection between a distinguished class of
quantum triangular GWAs and generalized down-up algebras, which to our
knowledge is not explored in the literature even though both families
have been extensively studied (as indicated by the numerous references in
this section). More precisely, recall that $U_q(\lie{sl}_2)$ is a
quantization of $U(\lie{sl}_2)$, in the sense of taking a ``classical
limit" as $q \to 1$ to obtain $U(\lie{sl}_2)$. Given the family of
deformations of $U(\lie{sl}_2)$ studied in \cite{Smi}, it is natural to
ask if these triangular GWAs also admit quantizations, which are
themselves then flat/PBW deformations of $U_q(\lie{sl}_2)$. We now
introduce a family of quantum triangular GWAs that provides a positive
answer to this question for Smith's family of algebras, and more
generally, for a large class of generalized down-up algebras.

\begin{exam}[Deformations of quantum $\lie{sl}_2 =$ quantization of
generalized down-up algebras]\label{Eqdef}

Consider a generalized down-up algebra given by \eqref{Edownup}, with
$\ch \F = 0 \neq \gamma$ and $r=1$. By Theorems \ref{Tgwa} and
\ref{Tdownup}, $\cals = \cals(\F[h], \theta_{1,\gamma}, s^{-1} f(h),
s^{-1})$ is a strict, based HRTA of rank one, with $\widehat{H} = \hfree
= \F$.

Let $q$ be an indeterminate over $\F$. We now propose a hitherto new
family of triangular GWAs $\cals_q$ over the $\F(q)$-algebra $H_q :=
\F(q)[K,K^{-1}]$, such that $\cals$ is the ``$q \to 1$" quasi-classical
limit of the algebra $\cals_q$. First define a more general family of
$\F(q)$-algebras $\cals(H_q = \F(q)[K^{\pm 1}], \theta, z_0', z_1')$ with
$z_0', z_1' \in H_q$ and $\theta : H_q \to H_q$ an $\F(q)$-algebra
automorphism of infinite order. As above, these algebras are strict,
based RTAs of rank one.
Now for the desired special case: given $l,m,n \in \Z$ with $l \neq 0$,
define the $\F(q)$-algebra $\cals_q(l,m,n)$ to be:
\begin{equation}\label{Edefq}
\cals_q(l,m,n) := \cals(\F(q)[K^{\pm 1}],\ \theta : K \mapsto q^{-l} K,\
s^{-1} q^m K^n f(\textstyle{\frac{-\gamma}{l} \cdot \frac{K-1}{q-1}}),\
s^{-1}).
\end{equation}

\noindent Observe that for various special cases of parameters,
$\cals_q(l,m,n)$ was studied earlier in the literature (but not in
general). Namely, Ji et.~al.~\cite{JWZ} and Tang \cite{Ta2} studied the
sub-family of algebras $\cals_q(2,0,0)$ with $s=1$ and $\theta(K) =
q^{-2} K$.

We now prove that the algebras $\cals_q(l,m,n)$ are indeed quantum
analogues of Smith's family of deformations of $U(\lie{sl}_2)$ -- and
more generally, the quantizations of a large class of generalized down-up
algebras \eqref{Edownup}. Note that if such a result is to hold, then
highest weight modules over $\cals_q(l,m,n)$ should also ``specialize" to
highest weight modules over the classical limit. It is natural to ask how
the corresponding highest weights are related.

To answer these questions, a natural procedure to follow is that in
\cite[Chapter 3]{HK} (see also \cite{Lu}) -- although several of the
steps therein need to be modified, as explained presently. Let $R$ be the
local subring of $\F(q)$, of rational functions that are regular at the
point $q=1$. Also define
\[ (K^n;m)_q := \frac{q^m K^n - 1}{q-1}, \qquad m,n \in \Z. \]

\noindent Now let $\cals^R_q(l,m,n)$ denote the (unital) $R$-subalgebra
of $\cals_q(l,m,n)$ generated by $U,D,K^{\pm 1}$, and $(K;0)_q =
(K-1)/(q-1)$. Then the following result holds.

\begin{theorem}[Deformation-quantization equals quantization-deformation]
Suppose $\F$ is a field of characteristic zero, $\gamma \in \F$, and
$\theta_{1,\gamma} \in \Aut_{\F-alg} \F[h]$ sends $h$ to $h + \gamma$.
Fix $f \in \F[h]$, $r=1$, and $s \in \F^\times$ not a root of unity.
Now define $\cals_q(l,m,n)$ as in \eqref{Edefq}, with $z_1 = s^{-1}$ and
$z_0 = s^{-1} q^m K^n f( -\gamma(K;0)_q / l)$ for some $l \neq 0,m,n \in
\Z$. Then,
\begin{equation}
\cals_1 := \cals^R_q(l,m,n) / (q-1) \cals^R_q(l,m,n) \cong
\cals(\F[h],\theta_{1,\gamma},s^{-1}f(h),s^{-1}).
\end{equation}

\noindent Now fix a scalar $\lambda \in \F(q)^\times$ such that
$\displaystyle \frac{\lambda - 1}{q-1} \in R$, and a highest weight
module $M_q(\lambda) \twoheadrightarrow \vla_q$ over $\cals_q(l,m,n)$,
where we identify $\lambda$ with the $\F(q)$-weight of $H_q$ sending $K$
to $\lambda$. If $v_\lambda \in (\vla_q)_\lambda$ generates $\vla_q$,
then
\begin{equation}
\vla_1 := \cals_q^R(l,m,n) v_\lambda / (q-1) \cals_q^R(l,m,n) v_\lambda
\end{equation}

\noindent is a highest weight module over $\cals_1 \cong \cals(\F[h],
\theta_{1,\gamma}, s^{-1} f(h), s^{-1})$ with highest $\F[h]$-weight
given by $\displaystyle h \mapsto \left. \frac{-\gamma}{l} \cdot
\frac{\lambda(K) - 1}{q-1}\right|_{q \to 1}$, and with the same graded
character as $\vla_q$ (up to modification of the highest weight).
\end{theorem}

\noindent In particular when $s=1$, the family of algebras studied by
Smith \cite{Smi} are indeed ``classical limits" (as $q \to 1$) of
triangular GWAs. Note that these algebras also provide deformations of
$U_q(\lie{sl}_2)$ (for $s=1$).

\begin{proof}
We follow the approach in \cite[Chapter 3]{HK}, developing the results
for both $\cals_1$ and $\vla_1$ simultaneously. We outline the steps,
omitting the proofs when they are similar to those in \textit{loc.~cit.}
The meat of the (new) proof lies in Step 5.
\begin{enumerate}
\item Set $\cals^R_\pm$ to be $R[U], R[D]$ respectively, and $\cals^R_0$
to be the $R$-subalgebra of $H_q = \F(q)[K^{\pm 1}]$ that is generated by
$K^{\pm 1}$ and $(K;0)_q$. Then all elements of the form $(K^n;m)_q$ and
$\frac{\beta K - \beta^{-1} K^{-1}}{q - q^{-1}}$ lie in $\cals^R_0$,
where $m,n \in \Z$ and $\beta \in R^\times$ such that $1 = \beta|_{q \to
1} := \beta \mod (q-1) R$.

\item The multiplication map $: \cals^R_- \otimes_R \cals^R_0 \otimes_R
\cals^R_+ \to \cals^R_q(l,m,n)$, induced from the triangular
decomposition of $\cals_q(l,m,n)$, is an isomorphism of $R$-algebras.

\item Henceforth, fix a weight $\lambda \in \widehat{H_q}$ such that
$\frac{\lambda(K)-1}{q-1} \in R$, as well as a highest weight module
$M_q(\lambda) \twoheadrightarrow \vla_q$. The \textit{$R$-form} of
$\vla_q$ is defined to be $\vla_R := \cals^R_q(l,m,n) v_\lambda$, where
$v_\lambda$ is the image of $1$ under the map $\cals_q(l,m,n)
\twoheadrightarrow M_q(\lambda) \twoheadrightarrow \vla_q$. Via the
previous step, we claim:
\[ \vla_R = \cals^R_- v_\lambda = \bigoplus_{\mu \leq \lambda}
(\vla_R)_\mu. \]

\noindent More precisely, we assert that the $R$-form $\vla_R$ is
$\cals^R_0$-semisimple, with each weight space a free rank one $R$-module
with $R$-basis $D^n v_\lambda$ for (unique) $n \geq 0$. Moreover, $\F(q)
\otimes_R \vla_R = \vla_q$.

In this step, we only explain why $\vla_R$ is $\cals^R_0$-semisimple.
First note that the weights of $\vla_R$ are of the form $\theta^{-n} *
\lambda$ for $n \geq 0$. Thus suppose $v = \sum_{j=1}^k v_j \in \vla_R$
with $v_j$ of weight $\theta^{-n_j} * \lambda$ for $0 \leq n_1 < n_2 <
\cdots$. The first claim is that for each fixed $j$, the ``interpolating
polynomial"
\[ I_j := \prod_{k \neq j} \frac{\lambda(K)^{-1} q^{l n_k} K - 1}{q^{l
(n_k-n_j)}-1} \]

\noindent lies in $\cals^R_0$. Indeed, we show that each factor lies in
$\cals^R_0$ by computing for any $r, 0 < s \in \Z$:
\[ \frac{\lambda(K)^{-1} q^r K - 1}{q^s - 1} = \frac{\lambda(K)^{-1}
q^r}{1 + \cdots + q^{s-1}} (K;0)_q + \lambda(K)^{-1} \frac{q^r-1}{q^s-1}
- \frac{\lambda(K)^{-1}}{1 + \cdots + q^{s-1}} \cdot
\frac{\lambda(K)-1}{q-1}, \]

\noindent and this is indeed in $\cals^R_0$ by assumption. Now apply the
quantity $I_j$ (defined above) to $v_\lambda$ to obtain $v_j$. Thus $v_j
\in \vla_R\ \forall j$, proving the $\cals_0^R$-semisimplicity of
$\vla_R$.

\item Define $\m := (q-1)R \subset R$ to be the unique maximal ideal of
the local ring $R$, and $\cals_1 := \cals^R_q(l,m,n) / \m
\cals^R_q(l,m,n), \ \vla_1 := \vla_R / \m \vla_R$. These are called the
\textit{classical limits} of $\cals_q(l,m,n)$ and $\vla_q$ respectively.
Also define $(\vla_1)_\mu := (R / \m) \otimes_R (\vla_R)_\mu$. Then
$\vla_1$ is a $\cals_1$-module, and each weight space is one-dimensional
with $\F$-basis $\overline{D^n v_\lambda}$ for some integer $n \geq 0$.

\item There exists a surjection of algebras $\pi : \cals = \cals(\F[h],
\theta_{1,\gamma}, s^{-1} f(h), s^{-1}) \twoheadrightarrow \cals_1$,
which sends $u,d,h$ to the images of $U, D, -\gamma (K;0)_q / l$
respectively, under the quotient map $: \cals_q(l,m,n) \twoheadrightarrow
\cals_1$. To see why, first note that the image of $(q-1) (K;0)_q = K-1$
is zero in $\cals_1$, whence $\overline{K} = 1$ in $\cals_1$. This shows
the surjectivity of the map $\pi$ if we show that $\pi$ is an algebra
map. We verify one of the relations; the others are similar. Namely,
$\pi(u) \pi(h)$ is the image in $\cals_1$ of
\[ U \cdot \frac{-\gamma}{l} \frac{K-1}{q-1} = \frac{-\gamma}{l} \frac{K
q^{-l} - 1}{q-1} U = \frac{-\gamma}{l} \cdot K \cdot \frac{q^{-l}-1}{q-1}
U + \frac{-\gamma}{l} \frac{K-1}{q-1} U, \]

\noindent and the image of the right-hand side in $\cals_1$ is precisely
$(-\gamma/l) \cdot 1 \cdot (-l) U + \pi(h) U = (\pi(h) + \gamma) \pi(u)$,
as desired.

The meat of the proof lies in showing that the surjection $\pi$ is an
isomorphism of algebras. We now describe an argument that utilizes the
GWA structure in our setting, as opposed to the symmetries under the Weyl
group in the setting of \cite[Chapter 3]{HK}.

Note that $\pi : \cals \twoheadrightarrow \cals_1$ restricts to a
surjection of algebras on the respective factors in the two triangular
decompositions. We first claim that $\pi$ is an isomorphism of Cartan
subalgebras. Indeed, given $0 \neq p(h) \in \F[h] = \cals_0$, choose $x
\in \F$ such that $p(x) \neq 0$ (since $\F$ is infinite). Define $\lambda
\in \widehat{H_q}$ via:
\[ \lambda : K \mapsto 1 - xl(q-1)/\gamma \in 1 + (q-1) \F \subset 1 + \m
\subset R^\times, \]

\noindent since $R$ is a commutative local ring. Then the above analysis
of $\vla_q$ (in steps (3) and (4)) holds, and $\pi(p(h))$ acts on the
highest weight space of $\vla_1$ by the scalar $p(x) \neq 0$. Therefore
$\pi(p(h)) \neq 0$, whence $\pi|_{\cals_0}$ has zero kernel, and hence is
an isomorphism of Cartan subalgebras.

We now claim that $\pi|_{\cals_-}$ also has trivial kernel. To prove the
claim, first fix any field extension $\F_u$ of $\F$, with $\F_u$ an
uncountable field. Since $\cals_q(l,m,n)$ is the quotient of the tensor
algebra $T_{\F(q)}({\rm span}_{\F(q)}(K,K^{-1},U,D))$ by an ideal, it is
possible to tensor this construction with $\F_u$ to obtain the same
algebra over $\F_u(q)$. Label these algebras $\cals_q^\F$ and
$\cals_q^{\F_u}$ respectively, and similarly for the other algebras
considered in the previous steps.
Now reconsider the entirety of the above procedure over $\F_u(q)$ instead
of $\F(q)$. We then make the \textit{sub-claim} that
$\pi|_{\cals^{\F_u}_-}$ has trivial kernel. To see why, note that
$\cals^{\F_u}_- \cong \F_u[d] \twoheadrightarrow (\cals^{\F_u}_1)_-$, and
this in turn surjects onto every highest weight module. Thus it suffices
to produce an infinite-dimensional Verma module over $\cals_1^{\F_u}$.

Now recall from Definition \ref{Delements} that $\z_n = \sum_{i=0}^{n-1}
s^{-(n-i)} f(h + i \gamma)$ is a nonzero polynomial in $h$ of degree
$\deg(f)$ (since $s$ is not a root of unity). Thus it has finitely many
roots for each $n$. Since $\F_u$ is uncountable, choose $x \in \F_u$ that
is not a root of $\z_n$ for any $n \geq 0$. It follows that the Verma
module $M_1^{\F_u}(\lambda_x)$ is simple over $\cals_1^{\F_u}$. In
particular, $(\cals^{\F_u}_1)_-$ is infinite-dimensional over $\F_u$.
Finally, $(\cals^{\F_u}_1)_- = \F_u \otimes_\F (\cals^\F_1)_-$, so we
obtain that $(\cals_1)_- = (\cals^\F_1)_-$ is also infinite-dimensional
over $\F$. Thus $\pi|_{\cals_-}$ is also an algebra isomorphism as
claimed.

Having shown the claim for $\pi|_{\cals_-}$, one shows the same result
for $\pi|_{\cals_+}$, either by a similar argument using lowest weight
theory, or directly via the anti-involutions in both settings from
(RTA3). Thus $\pi : \cals \twoheadrightarrow \cals_1$ is an isomorphism
of algebras using (RTA1).

\item It follows using the previous step that $\vla_1$ is a
$\cals$-module (since it is a $\cals_1$-module), with the same weight
bases for both module structures.

One now shows that $\vla_1$ is a highest weight module over $\cals$, with
the same formal character as $\vla_q$. Moreover, the highest $h$-weight
for $\vla_1$ is precisely as claimed, since $h$ acts on the highest
weight space via $\pi(h)$, i.e.~by the scalar
$\displaystyle \frac{-\gamma}{l} \cdot \left. \frac{\lambda(K)-1}{q-1}
\right|_{q \to 1}$
as claimed. We also remark that if $\vla_1$ is simple but $\vla_q$ has a
maximal vector of weight $\theta^{-n} * \lambda < \lambda$, then since
the two graded characters are equal, the corresponding vector in $\vla_1$
would also be maximal, which is impossible. It follows that $\vla_q$ is a
simple $\cals_q(l,m,n)$-module if $\vla_1$ is a simple $\cals_1$-module.
\end{enumerate}
%Moreover, $\cals(\F[h],\theta_{1,\gamma},s^{-1} f[h], s^{-1})$ and
%$\scrz_R \hookrightarrow \scrz$ are all filtered algebras (with
%$u,d,U,D$ having degree $1$, and $\F[h], \F[K^{\pm 1}]$ having degree
%$0$). Hence $\scrz_1$ is filtered in the obvious way as well, and the
%surjection $\pi$ respects this filtration. Furthermore, the associated
%graded map between the associated graded algebras is the identity map $:
%\F[u,d] \rtimes \F[h] \to \F[\overline{U},\overline{D}] \rtimes
%\F[\overline{H}]$. This concludes the proof.
\end{proof}
\end{exam}
%}}}

%{{{1 Section 8.5 - Solutions of polynomial-exponential equations
\subsection{Solutions of polynomial-exponential equations}\label{Slech}

We conclude this section by showing Theorems \ref{Tdownup} and
\ref{Tkleinian}. The proofs use a result on ``polynomial-exponential
equations" over a general field. We begin with a result by Schlickewei
\cite{Sch} that was proved for number fields. Namely, Schlickewei showed
that a special family of equations (with argument $n \in \Z$) has only
finitely many integer solutions.

\begin{theorem}[Schlickewei {\cite[Theorem 1.1]{Sch}}]\label{Tschli}
Given a field $\F$ of characteristic zero, consider the {\em
polynomial-exponential equation} (with argument $n \in \Z$):
\begin{equation}\label{Ea1}
F_n := \sum_{j=1}^m p_j(n) \alpha_j^n = 0, \qquad n \in \Z,
\end{equation}

\noindent where $m \in \N$, $0 \not\equiv p_j \in \F[X], \alpha_j \in
\F^\times\ \forall j \leq m$, and $\alpha_i / \alpha_j$ is not a root of
unity for all $i \neq j$.

If $\F$ is an algebraic number field, then \eqref{Ea1} has only finitely
many solutions in $\Z$.
\end{theorem}

It turns out that Theorem \ref{Tschli} is true in all fields of
characteristic zero; as we are unsure if this is mentioned in the
literature, we write down a proof for completeness. (The proof does not
use Theorem \ref{Tschli}.)

\begin{theorem}\label{Tskolem}
The conclusion of Theorem \ref{Tschli} holds over any field $\F$ of
characteristic zero.
\end{theorem}

\begin{proof}
We prove the result in various steps. The first step is to claim that
every such polynomial-exponential function gives rise to a {\em linear
recurrence sequence} $\{ F_n : n \in \Z \}$ (with suitable initial
values); this has essentially been shown for any field in \cite[Section
2]{MvP}.

Now suppose $F_n$ vanishes infinitely often in $\Z$, say on the set $T$.
(We will prove that $p_i \equiv 0\ \forall i$.)
If $T \subset \Z$ is the set of zeros, then we restrict to $T' = T \cap
\N$ if this is an infinite set. Otherwise $T' \cap -\N$ is infinite, and
changing every $\alpha_i$ to $\alpha_i^{-1}$ and $p_i$ to a new
polynomial $q_i(X) := p_i(-X)$ if necessary, we may assume that $F_n = 0$
for all $n$ in an infinite set $T' \subset \N$. (Note that $q_i \equiv 0
\Leftrightarrow p_i \equiv 0$, so we may work with the new setup now.)

Since $\ch \F = 0$, we conclude by the Skolem-Mahler-Lech Theorem
\cite{Lech} that $F_n$ vanishes for all $n$ in an infinite arithmetic
progression, say $r + \N d$. But then
\[ \sum_j (p_j(r + dn) \alpha_j^r) (\alpha_j^d)^n = 0\ \forall n \in \N.
\]

Once again, we fix $d \neq 0, r$ and call the new polynomial $q_j(X) :=
p_j(r + dX)$; then $q_j \equiv 0$ if and only if $p_j \equiv 0$. Also set
$\beta_j := \alpha_j^d$; these are pairwise distinct, and we are left to
prove the following\medskip

\noindent {\bf Claim.} Fix pairwise distinct $\beta_i \in \F$ and
polynomials $q_i(T) \in \F[T]$, for a field $\F$ of characteristic zero.
If $G(n) := \sum_i q_i(n) \beta_i^n = 0\ \forall n \in \N$, then all the
polynomials $q_i$ are identically zero.\medskip

We prove this claim by assuming it to be false and obtaining a
contradiction. If the claim is false, then $D := \sum_{i : q_i \not\equiv
0} \deg(q_i)$ is defined (and nonnegative). Now obtain a contradiction by
induction on $D$. (The base case $D=0$ is treated using the Vandermonde
determinant from $G(1), \dots, G(n)$; for the general case, consider
$H(n) := G(n+1) - \beta_i G(n)$, where $\deg q_i > 0$.)
\end{proof}

It is now possible to show that a large number of ``classical" and
``quantum" generalized down-up algebras satisfy Condition (S3).

\begin{proof}[Proof of Theorem \ref{Tdownup}]
Throughout this proof we use $\theta$ instead of $\theta_{r,\gamma}$.
Fix an algebra map $\lambda : H \to \F$. First suppose that $r=1$; then
$\lambda \circ \theta^n(h) = \lambda(h) + n \gamma$.
Thus $\lambda \in \hfree$ if and only if $n \gamma$ is never zero for
$n>0$, i.e., $\ch \F = 0 \neq \gamma$. Next, if $r \neq 1$, then compute:
\[ \lambda \circ \theta^n(h) = r^{-n} \lambda(h) + r^{-1} \gamma \frac{1
- r^{-n}}{1-r^{-1}}. \]

\noindent It is clear that if $r$ is a root of unity, then this
expression equals $\lambda(h)$ for all $h$, for infinitely many $n$. On
the other hand, if $r \notin \sqrt{1}$, then it is clear for any $n>0$
that
\[ \lambda \equiv \lambda \circ \theta^n \quad \Longleftrightarrow \quad
\lambda(h) = \frac{r^{-1} \gamma}{1 - r^{-1}}, \]

\noindent and this completes the proof of the first part. Next when
$r=s=1$, it is not hard to show that $Z(A) \cap \F[h] = \F$. Moreover, a
quadratic Casimir operator always exists because of the identity
$\binom{X}{k} = \binom{X-1}{k-1} + \binom{X-2}{k-1} + \cdots$, which
helps show that power sums $\sum_{i=1}^n i^k$ are polynomials in $n$ of
degree $k+1$ with rational coefficients.

Finally, to study the sets $[\lambda]$ we first compute for $n>0$ and
$r=1$:
\[ \z_n = \sum_{i=0}^{n-1} s^{-1} \theta^i(f(h)) s^{-i} =
\sum_{i=0}^{n-1} s^{-1-i} f(\theta^i(h)) = \sum_{i=0}^{n-1} s^{-1-i} f(h
+ i \gamma). \]

\noindent If $f \equiv 0$ then clearly $\z_n = 0$ and $[\lambda]$ is
infinite for every weight $\lambda$. Now suppose $\gamma \neq 0$ and $f
\not\equiv 0$ is of the form $f(h) = \sum_{j=1}^k c_j h^{m_j}$ for
integers $0 \leq m_1 < \cdots < m_k$, with $c_k \in \F^\times$. We first
assume that $r=1$ and compute:
\[ \z_n = \sum_{i=0}^{n-1} s^{-1-i} \sum_{j=1}^k c_j (h + i
\gamma)^{m_j} = \sum_{i=0}^{n-1} \sum_{j=1}^k \sum_{l=0}^j s^{-1-i} c_j
\binom{j}{l} h^{j-l} \gamma^l i^l
= \sum_{j=1}^k \sum_{l=0}^j c_j \binom{j}{l} h^{j-l} \gamma^l
\sum_{i=0}^{n-1} s^{-1-i} i^l. \]

\noindent If moreover $s=1$, then it is clear that $\z_n = 0$ if $\ch(\F)
| n$ (since for every prime $p>0$ and all integers $l \geq 0$,
$\sum_{i=0}^{p-1} i^l$ is divisible by $p$, by using the primitive
generator of $\Z / p \Z$). Now if $\ch \F = 0$, then $\lambda(\z_n)$ is a
polynomial in $n$ of degree at most $1 + m_k$, so it has only finitely
many roots $n>0$. A similar argument for $n<0$ shows that $[\lambda]$ is
always finite if $r=s=1$ and $\ch \F = 0$. On the other hand, if $\ch \F
> 0$ and $r=s=1$, then $[\lambda]$ is always infinite.

Now suppose $\ch \F = 0$, $\gamma \neq 0$, $r=1$, and $s \not\in
\sqrt{1}$. First assume by a change of variables that $\gamma = 1$,
without loss of generality; since $\ch \F = 0$, one can then write the
polynomial $f(h)$ as a linear combination of the basis elements
$t_{n,s}(h) := s^{-1}\binom{h+1}{n} - \binom{h}{n}$ of $\F[h]$. Now if $f
\equiv \sum_{j \geq 0} a_j t_{j,s}$ (finite sum), then define $\tilf(h)
:= s^{-1} \sum_{j \geq 0} a_j \binom{h}{j}$. Then for $n \geq 0$,
\begin{align*}
\z_n = &\ \sum_{i=0}^{n-1} s^{-1-i} f(h+i) = \sum_{i=0}^{n-1} s^{-i}
(s^{-1} \tilf(h+i+1) - \tilf(h+i)) = s^{-n} \tilf(h+n) - \tilf(h), \\
\z_{-n} = &\ \theta^{-n}(\z_n) = s^{-n} (\tilf(h) - s^n \tilf(h-n)).
\end{align*}

\noindent Now given any weight $\lambda$, applying Theorem \ref{Tskolem}
to the nontrivial polynomial-exponential equation (in $n \in \Z$) given
by
\[ F_n := \lambda(\z_n) = \tilf(\lambda(h)) 1^n + (-\tilf(\lambda(h)+n))
(s^{-1})^n = 0 \]

\noindent shows that there are only finitely many integer solutions,
whence $[\lambda]$ is finite for every $\lambda$.
\end{proof}

Finally, we show the analogous result (to Theorem \ref{Tdownup}) for
quantum triangular GWAs.

\begin{proof}[Proof of Theorem \ref{Tkleinian}]
%The first part is standard (noting that if $\alpha(\Gamma) \subset
%\F^\times$ is a finite group, then it is cyclic.)
%%%{\color{red}
Clearly, the orders of $\theta$ and $\alpha$ are either both infinite or
both equal. Next, if $\Gamma / \ker(\alpha)$ has finite order, say $N$,
then for all $g \in \Gamma$, $\alpha^N(g) = \alpha(g)^N = \alpha(g^N) \in
\alpha(\ker(\alpha)) = 1$, whence $\alpha$ has finite order as well.
Moreover, $\alpha : \Gamma/\ker(\alpha) \to \F^\times$ is an injection,
whence $\Gamma / \ker(\alpha)$ embeds into a finite group of units in
$\F$, which must therefore be cyclic. Hence $\Gamma / \ker(\alpha)$ is
cyclic, and generated by some $g_0$ of order $N$. This implies that
$\alpha$ also has order $N$.
Conversely, say $\alpha$ has order $N$. Then the order of each $g$
divides $N$. But (via $\alpha$,) there are only finitely many such values
of $\alpha(g)$, namely, (powers of) $N$th roots of unity. Hence $\Gamma /
\ker(\alpha)$ must be finite, since it maps faithfully into these $N$th
roots. It is also easy to see that a primitive $N$th root is in the image
of $\alpha$, which proves the first assertion.
%%%}

In order to show the next two parts, we first define $N_0$ to be the
least common multiple of the orders of the roots of unity $\{ \alpha(g) s
: g \in \Gamma_1 \cup \Gamma_2 \} \subset \sqrt{1}$, as well as of the
orders of $\alpha(g_{ij}^{-1} g_{kj})$ over all $i,j,k$. Now compute for
any $\mu \in \hfree$ that $\mu(z'_n) = z'_n = s^n$ for all $n \geq 0$.
Therefore we obtain for $n>0$:
\begin{align}\label{ES32}
\mu(\z_n) = &\ \sum_{i=0}^{n-1} s^{n-1-i} \sum_{g \in \Gamma} a_g
\alpha(g)^{-i} \mu(g) = n s^{n-1} \sum_{g \in \Gamma_1} a_g \mu(g) +
s^{n-1} \sum_{g \in \Gamma_2 \cup \Gamma_3} a_g \mu(g) \frac{1 -
(\alpha(g)s)^{-n}}{1 - (\alpha(g)s)^{-1}},\notag\\
\mu(\z_{-n}) = &\ \mu(\theta^{-n}(\z_n)) = \sum_{i=0}^{n-1} s^{n-1-i}
\sum_{g \in \Gamma} a_g \alpha(g)^{n-i} \mu(g) = \sum_{g \in \Gamma}
\sum_{i=0}^{n-1} a_g \alpha(g) \mu(g) (\alpha(g) s)^i\\
= &\ n s^{-1} \sum_{g \in \Gamma_1} a_g \mu(g) + \sum_{g \in \Gamma_2
\cup \Gamma_3} a_g \alpha(g) \mu(g) \frac{1 - (\alpha(g)s)^n}{1 -
(\alpha(g)s)}.\notag
\end{align}

\noindent We now show the two remaining parts in this result.
\begin{enumerate}
\item If \eqref{ES3kleinian1} holds, then we claim that $\mu(\z_{m N_0}) =
0$ for all $m \in \N$. Indeed, the sum in \eqref{ES32} over $g \in
\Gamma_1$ vanishes by assumption, and we are left with:
\[ \mu(\z_{m N_0}) = s^{m N_0 - 1} \sum_{g \in \Gamma_2} a_g \mu(g)
\frac{1 - (\alpha(g)s)^{-m N_0}}{1 - (\alpha(g)s)^{-1}} + 
s^{m N_0 - 1} \sum_j \sum_i a_{ij} \mu(g_{ij}) \frac{1 -
(\alpha(g_{ij})s)^{-m N_0}}{1 - (\alpha(g_{ij})s)^{-1}}. \]

\noindent By construction, each summand of the sum over $g \in \Gamma_2$
vanishes, and moreover, the element $(\alpha(g_{ij})s)^{-m N_0}$ is
independent of $i$ for each fixed $j$. Thus, we obtain:
\[ \mu(\z_{m N_0}) = s^{m N_0 - 1} \sum_j (1 - (\alpha(g_{1j}) s)^{-m
N_0}) \sum_i \frac{a_{ij} \mu(g_{ij})}{1 - (\alpha(g_{1j})s)^{-1}}, \]

\noindent which vanishes by assumption, proving the claim.

Similarly, one shows using \eqref{ES32} that if \eqref{ES3kleinian2}
holds, then $\mu(\z_{-m N_0}) = 0$ for all $m \in \N$.

\item Conversely, suppose $[\lambda]$ is infinite for $\lambda \in
\hfree$. Then at least one of $[\lambda] \cap (\pm \N \theta * \lambda)$
is infinite. Suppose first that the former case holds. Define $N_0$ as
above; then there exists $n_0 \in \N$ such that $[\lambda] \cap ((n_0 +
N_0 \N) \theta * \lambda)$ is infinite. Thus, fix $0 < n_1 < n_2 < \dots$
such that $(n_0 + N_0 n_k) \theta * \lambda \in [\lambda]$ for all $k>0$.
Then using \eqref{ES32},
\[ 0 = s^{1 - n_0 - N_0 n_k} \lambda(\z_{n_0 + N_0 n_k}) = (n_0 + N_0
n_k) \sum_{g \in \Gamma_1} a_g \lambda(g) + \sum_{g \in \Gamma_2 \cup
\Gamma_3} a_g \lambda(g) \frac{1 - (\alpha(g)s)^{-n_0 - N_0 n_k}}{1 -
(\alpha(g)s)^{-1}}, \]

\noindent for all $k \in \N$. Rearranging this expansion, we obtain that
\[ (p_{00} + N_0 \sum_{g \in \Gamma_1} a_g \lambda(g) X) 1^X + \sum_j p_j
\beta_j^X = 0, \qquad \forall X = n_1, n_2, \dots \]

\noindent where
\begin{align}
p_{00} := &\ \sum_{g \in \Gamma_2} a_g \lambda(g) \frac{1 -
(\alpha(g)s)^{-n_0}}{1 - (\alpha(g)s)^{-1}} + \sum_{g \in \Gamma_3}
\frac{a_g \lambda(g)}{1 - (\alpha(g)s)^{-1}} + n_0 \sum_{g \in \Gamma_1}
a_g \lambda(g) \in \F, \notag \\
p_j := &\ - s^{-n_0} \sum_i \frac{a_{ij} \lambda(g_{ij})}{1 -
(\alpha(g_{ij}) s)^{-1}} \alpha(g_{ij})^{-n_0}, \\
\beta_j := &\ (\alpha(g_{1j}) s)^{-N_0}. \notag
\end{align}

\noindent Now note that $1$ and the $\beta_j$ are distinct, and the ratio
of no two of these is a root of unity. Since $\ch \F = 0$, Theorem
\ref{Tskolem} now implies that
\[ p_{00} = \sum_{g \in \Gamma_1} a_g \lambda(g) = p_j = 0 \qquad
\forall j. \]

\noindent Finally, define $\mu := \theta^{n_0} * \lambda$. Then,
\begin{align*}
0 = &\ -s^{n_0} p_j = \sum_i \frac{a_{ij} \mu(g_{ij})}{1 - (\alpha(g_{ij}
s))^{-1}} \qquad \forall j,\\
0 = &\ s^{n_0} \sum_{g \in \Gamma_1} a_g \lambda(g) = \sum_{g \in
\Gamma_1} a_g \lambda(g) \alpha(g)^{-n_0} = \sum_{g \in \Gamma_1} a_g
\mu(g),
\end{align*}

\noindent and \eqref{ES3kleinian1} follows. A similar analysis shows
using \eqref{ES32} and Theorem \ref{Tskolem} that if $[\lambda] \cap (-\N
\theta * \lambda)$ is infinite, then \eqref{ES3kleinian2} holds, which
concludes the proof.
\end{enumerate}
\end{proof}
%}}}

%{{{1 Section 9 - Non-Hopf examples of RTAs
\section{Non-Hopf examples of RTAs}\label{Sgwa2}

Note that all of the previous examples of triangular GWAs in Section
\ref{Sgwa} -- with the exception of generalized down-up algebras
\eqref{Edownup} with $r \neq 1$ (such as Example \ref{EJZ}) -- were
strict Hopf RTAs.
We now provide examples of triangular GWAs that are not Hopf RTAs. The
Hopf structure in the examples gets increasingly weaker, in the following
precise sense:
\begin{itemize}
\item As a first example, consider Example \ref{EJZ}, in which $H =
\F[h]$ is a Hopf algebra, but the Hopf structure is (necessarily) not
used.

\item In the second example -- see Example \ref{Echmutova} -- $H$ is a
topological Hopf algebra but not a Hopf algebra.

\item In the final example -- see Example \ref{Emirza} -- $H$ is not even
a topological Hopf algebra.
\end{itemize}

\begin{exam}[{\em Continuous Hecke algebra of
$\lie{gl}_1$}{}]\label{Echmutova}
Let $\F$ be any field, and $H = \calo(\F^\times)^* = \F[T^{\pm 1}]^* =
\F[[t^{\pm 1}]]$, the algebra of ``Fourier series" or distributions on
the unit circle (if $\F = \C$). This is a topological Hopf algebra with
coordinatewise multiplication, and other Hopf operations given by
\[ \eta(1) = \sum_{n \in \Z} t^n, \quad \Delta(t^m) = \sum_{n \in \Z} t^n
\otimes t^{m-n} \in H \widehat{\otimes} H, \quad \vi(\sum_{n \in \Z} a_n
t^n) = a_0, \quad S(\sum_{n \in \Z} a_n t^n) = \sum_{n \in \Z} a_n
t^{-n}. \]

The corresponding triangular GWA with $z_1 = 1$ is the {\it continuous
Hecke algebra} of $GL(1)$ and $\F \oplus \F^*$, where $\theta$ is
(non-coordinatewise) multiplication by $t$, i.e.,
\[ \hk(GL(1), \F \oplus \F^*) := \cals(\calo(\F^\times)^*, \theta,
\kappa, 1), \quad \kappa \in \calo(\F^\times)^*, \quad \theta(\sum_{n \in
\Z} a_n t^n) := \sum_{n \in \Z} a_n t^{n+1} = \sum_{n \in \Z} a_{n-1}
t^n. \]

\noindent Continuous Hecke algebras were introduced in \cite{EGG} as
``continuous" generalizations of Drinfeld's family of degenerate affine
Hecke algebras. The family of algebras under discussion is in some sense
the simplest special case, of ``Lie rank zero". Higher (Lie) rank
examples of infinitesimal Hecke algebras are discussed in the following
section. In this section and the next, we differentiate between the Lie
rank of an infinitesimal Hecke algebra (which equals the rank of the
underlying reductive Lie algebra $\lie{g}$) and the ``(RTA) rank" of a
strict, based RTA given in Definition \ref{Dstrict}. In fact, the based
Hopf RTAs considered in Section \ref{Sinfhecke} are not strict, hence we
will only talk about their Lie rank, but not their RTA-rank.

\begin{remark}
Observe that $H = \F[[t^{\pm 1}]]$ is the $\F$-algebra of functions on
$\Z$. Thus if $\kappa = 0$, then the triangular GWA $\hk(GL(1), \F \oplus
\F^*)$ also equals $\cala(\nn)$, where $\cala(Q^+_0)$ was defined in
Theorem \ref{Tmonoid}, with $\theta_1 \ltimes \theta_2^{-1} :=
\theta_2^{-1}$ for $\theta_1, \theta_2 \in Q^+_0 = \nn$.
\end{remark}

We now list some of the properties of (Lie) rank zero continuous Hecke
algebras, which are triangular GWAs from above.

\begin{prop}
Suppose $\kappa = \sum_{n \in \Z} a_n t^n \in \calo(\F^\times)^*$ and
$\hk := \hk(GL(1), \F \oplus \F^*)$.
\begin{enumerate}
\item $\hk$ is a strict, based RTA of rank one, but not a Hopf RTA.

\item The set of weights is $\widehat{H} = \{ \mu_m : m \in \Z \}$ with
$\mu_m(t^n) := \delta_{m,n}$. Moreover, $\theta^n * \mu_m = \mu_{n+m}$
for $m,n \in \Z$, so $\hfree = \widehat{H}$.

\item For all $m \in \Z$, define $s_{m,n}(\kappa) := a_m + a_{m-1} +
\dots + a_{m+n+1}$ for $n<0$, $s_{m,0}(\kappa) := 0$, and
$s_{m,n}(\kappa) := a_{m+1} + \dots + a_{m+n}$ for $n>0$. Then the Verma
module $M(\mu_m)$ is uniserial, with
\[ l(M(\mu_m)) = \# \{ n \leq 0 : s_{m,n} = 0 \}, \qquad |S^3(\mu_m)| =
|[\mu_m]| = \# \{ n \in \Z : s_{m,n} = 0 \}. \]
\end{enumerate}
\end{prop}

\noindent We remark that $s_{m,n}(\kappa) = \sum_{j = 1 +
\min(m,m+n)}^{\max(m,m+n)} a_j$ for all $m,n \in \Z$.

\begin{proof}
The first part follows from Theorem \ref{Tgwa}, and the second holds
since $\{ t^n : n \in \Z \}$ is a complete set of primitive idempotents
in $H$. To show the third part, we compute using Definition
\ref{Delements} and that $z_1 = 1$:
\[ \z_n = \sum_{i=0}^{n-1} \theta^i(\kappa) = \sum_{m \in \Z}
\sum_{i=0}^{n-1} a_m t^{m+i} = \sum_{m \in \Z} t^m \sum_{i=0}^{n-1}
a_{m-i}, \qquad
\z_{-n} = \sum_{m \in \Z} t^m \sum_{i=0}^{n-1}
a_{m+n-i}, \qquad \forall n > 0. \]

\noindent The third part now follows from results on the uniseriality of
Verma modules, as discussed in the proof of Theorem \ref{Tgwa}.
\end{proof}
\end{exam}

Next is an example of an RTA in which the Cartan subalgebra is not a
(topological) Hopf algebra.

\begin{exam}[{\em GWA arising from geometry}]\label{Emirza}
Suppose $X$ is an object in some category $\scrc$ of topological spaces
containing the real line, and $T : X \to X$ is an automorphism in $\scrc$
such that $X^\vee := \hhh_\scrc(X,\R)$ is an $\R$-algebra containing the
constant map $: X \to 1$, which is stable under pre-composition with $T$.
We now construct a ``first approximation" to a GWA. Consider the
subalgebra $A' \subset \End_\R(X^\vee)$ generated by the operators $H_X
:= \{ M_f : f \in X^\vee \}$, and two additional operators $U,D$, where:
\begin{itemize}
\item $M_f$ corresponds to multiplication by $f$ in $X^\vee$;
\item $U(f) := f \circ T$ and $D(f) := f \circ T^{-1}$ for $f \in
X^\vee$.
\end{itemize}

\noindent Then $U,D$ ``count" the dynamics of applying $T^{\pm 1}$ to
$X$, i.e., the following equations hold in $\End_\R(X^\vee)$:
\[ U^n f(-) = f(T^n(-)) U^n, \qquad D^n f(-) = f(T^{-n}(-)) D^n. \]

\noindent Moreover, $UD = DU = 1_{H_X}$ in $\End_\R(X^\vee)$, $H_X \cong
X^\vee$, and $T^* : H_X \to H_X$ is indeed an algebra automorphism. Thus
$A' = \cals(H_X, \theta = T^*, 0, 1) / (UD - 1_{H_X}, DU - 1_{H_X})$.

We now define an associated family of triangular GWAs as follows.
Suppose $T : X \to X$ is an automorphism in $\scrc$ of infinite order
that stabilizes $X^\vee$. For each $z_0, z_1 \in H_X$, define $A :=
\cals(H_X, T^*, z_0, z_1)$. This is a strict, based RTA of rank one, but
not necessarily a (topological) Hopf RTA, since $H_X \cong X^\vee$ is not
a (topological) Hopf algebra for every topological space $X$.
\end{exam}

We conclude with a conjectural example involving twisted generalized Weyl
algebras.

\begin{exam}[{\em Twisted generalized Weyl algebras}]\label{ETGWA}

We follow the treatment in the paper \cite{FH}. Given a TGW datum $(R,
\sigma, t)$, define the twisted GWA $A := \mathcal{A}_\mu(R,\sigma,t)$,
constructed as the quotient of $\mathcal{C}_\mu(R,\sigma,t)$ by the ideal
$\mathcal{I}_\mu(R,\sigma,t)$, as in \cite[Definition 2.3]{FH}. (These
algebras were originally defined by Mazorchuk and Turowska \cite{MT}.) We
further \textbf{assume} that the algebra $A$ satisfies three additional
conditions:
\begin{itemize}
\item The parameter matrix $(\mu_{ij})$ with diagonals removed, is
symmetric.

\item  The ``middle" subalgebra $R$ is isomorphic to a polynomial algebra
$H[t_1, ..., t_n]$ over some commutative $\F$-algebra $H$.
(Then $t_i$ equals $y_i x_i$ as in the defining algebra relations.)

\item The algebra $A$ satisfies \cite[Definition 2.5 and Theorem 2.7]{FH}
of ``$\mu$-consistency".
\end{itemize}

In this case, a natural question to ask is if the algebra $A$ is an RTA.
That (RTA3) holds is not hard to show, but the other two RTA axioms are
not known to hold in this degree of generality. Specifically, are the
subalgebras $B_x, B_y$ generated by the $x_i$ and the $y_i$ respectively,
isomorphic as vector spaces to polynomial algebras in these variables?
Does the condition (RTA1) hold?

Another question of interest is to verify whether or not the type $A_1^n$
case of a multiparameter twisted GWA (defined in \cite[Theorem 4.1]{FH})
is an RTA.
\end{exam}
%}}}

\section{Non-strict RTAs: higher Lie rank infinitesimal Hecke
algebras}\label{Sinfhecke}

In the final section we address yet another motivation for this paper --
to construct a framework that includes RTAs that are not strict. In this
section we consider \textit{infinitesimal Hecke algebras}
$\mathcal{H}_\beta(\lie{g},V)$, which are deformations of $H_0(\lie{g},V)
:= U(\lie{g} \ltimes V)$, with $\lie{g}$ a reductive Lie algebra and $V$
a finite-dimensional $\lie{g}$-module. Note that these algebras include
reductive Lie algebras, for which $V = 0$. In this section we work over a
ground field $\F$ of characteristic zero.

The first example of infinitesimal Hecke algebras is over $\lie{sl}_2$. A
family of these algebras was described in Example \ref{Einfhecke1} and
studied in detail in \cite{Kh,KT}, and they are strict, based Hopf RTAs
of rank one. The next two classes of examples discussed in this section,
were introduced in \cite{EGG}.

%{{{1 Section 10.1 - Partial examples
\subsection{Partial examples}

Before discussing specific families of infinitesimal Hecke algebras, we
first mention a general framework for such algebras, in which one can
show that Condition (HRTA2) is related to Ginzburg's {\em Generalized
Duflo Theorem} \cite[Theorem 2.3]{Gi}.

\begin{prop}\label{Pduflo}
Suppose an $\F$-algebra $A$ is generated by an abelian Lie algebra
$\lie{h}_1$ and a finite-dimensional $\lie{h}_1$-semisimple module $M$,
with $M_0 = 0 = \ch \F$. The following are equivalent:
\begin{enumerate}
\item ``HRTA2" holds; in other words, there exist
\begin{itemize}
\item a Lie subalgebra $\lie{h}_0 \subset \lie{h}_1$,

\item a decomposition $M = M^+ \oplus M^-$ into $\lie{h}_1$-semisimple
submodules, and

\item an $\F$-linearly independent set $\Delta' \subset \lie{h}_0^*$,
\end{itemize}

\noindent such that $M^\pm = \bigoplus_{\mu \in \pm \nn \Delta'}
M^\pm_\mu$. (In particular, the subalgebras generated by $M^\pm$ are
$\lie{h}_1$-semisimple, with finite-dimensional weight spaces, and
one-dimensional zero weight space spanned by the unit.)

\item There exists a codimension $d$ subspace $K \subset \lie{h}_1^*$
(for some $d$), such that modulo $K$, and up to a change of basis,
$\overline{\wt(M)} := \wt(M) + K \subset \Q^d \setminus \{ 0 \}$.

\item There exists $\delta \in \lie{h}_1$ such that $\wt(M)(\delta)
\subset \Z \setminus \{ 0 \}$.
\end{enumerate}
\end{prop}

\begin{remark}\hfill
\begin{enumerate}
\item For example, for the infinitesimal Hecke algebras associated to
$(\lie{g},V) = (\lie{gl}_n$, $\F^n \oplus (\F^n)^*)$ or $(\lie{sp}_{2n},
\F^{2n})$ (which were characterized in \cite{EGG}), the second condition
is easily verified, for $M = V \oplus \lie{n}^+ \oplus \lie{n}^-,\ K=0$,
and the basis consisting of the fundamental weights (and one additional
weight in $Z(\lie{g})^*$ for $\lie{gl}_n$).

\item The third condition indicates that for infinitesimal algebras
$H_\beta(\lie{g},V)$ with $V_0 = 0$, one can always take $\Delta'$ to be
a singleton. (In particular, this also holds for semisimple Lie algebras
$\lie{g}$.) This is why the present paper discusses the ``Lie rank" of
non-strict based RTAs, but does not define the RTA rank for such
algebras.

\item The first of the three equivalent conditions is what is needed to
show that $A$ is an HRTA; the second is what typically comes as ``given
data" for $A$; and the third is needed to apply Ginzburg's Generalized
Duflo Theorem \cite{Gi}.

\item Note that the conditions in (HRTA2) are stated in terms of $B^\pm$,
unlike the first statement above. However, in the case of infinitesimal
Hecke algebras $\mathcal{H}_\beta(\lie{g},V)$, the spaces $M^\pm$ are
typically Lie algebras if $\beta = 0$, and $B^\pm$, which are the
subalgebras generated by $M^\pm$ inside $\mathcal{H}_\beta(\lie{g},V)$,
are deformations of $U(M^\pm) \subset \mathcal{H}_0(\lie{g},V)$. In
particular, given (RTA1), a suitable version of the PBW property yields
the regularity conditions inside (HRTA2).

\end{enumerate}
\end{remark}

\begin{proof}
We prove a series of cyclic implications.\medskip

\noindent $(1) \Rightarrow (2)$:
Since $\wt M$ is finite, choose a finite subset $\Delta_0 \subset
\Delta'$ such that $M = \bigoplus_{\mu \in \pm \Z \Delta_0} M_\mu$. Now
define $d := |\Delta_0|,\ \lie{h}_{00} := {\rm span}_{\F}(\Delta_0)$, and
$K := \lie{h}_{00}^\perp \subset \lie{h}_1^*$. Then (2) follows.\medskip

\noindent $(2) \Rightarrow (3)$:
Since $\Q$ is an infinite field and $0 \notin \overline{\wt(M)}$, choose
a hyperplane $K_1 \subset \Q^d \setminus \overline{\wt(M)}$, and consider
$0 \neq h_0 \in (K_1 + K)^\perp = (\overline{K_1})^\perp$.
Since these weights all lie in a $\Q$-vector space, there exists $c \in
\F^\times$ such that
\[ \alpha(h_0) \in \Q^\times \cdot c\ \forall \alpha \in \wt(M) \subset
\lie{h}_1^*. \]

\noindent Now rescale $h_0$ using that $\ch \F = 0$, to obtain $\delta$
such that $\alpha(\delta) \in \pm \N\ \forall \alpha \in \wt(M)$.\medskip

\noindent $(3) \Rightarrow (1)$:
Set $\lie{h}_0 = \F \cdot \delta$, $M^\pm := \bigoplus_{n \in \pm \N}
M_n$ with respect to $\ad \delta$, and $\alpha \in \lie{h}_0^*$ via:
$\alpha(\delta) = 1$. Now set $\Delta' := \{ \alpha \}$.
\end{proof}
%}}}

%{{{1 Section 10.2 - The general linear case
\subsection{The general linear case}\label{Sgln}

We now show that all infinitesimal Hecke algebras of the form $\infgl$
are based Hopf RTAs. First recall the definition of these algebras from
\cite[Section 4.1.1]{EGG}:
Set $\lie{g} = \lie{gl}_n(\F)$ and $V= \F^n \oplus (\F^n)^*$. Identify
$\lie{g}$ with $\lie{g}^*$ via the trace pairing $\lie{g} \times \lie{g}
\to \F :\ (A,B) \mapsto \tr(AB)$, and identify $U \lie{g}$ with $\sym
\lie{g}$ via the symmetrization map.
Then for any $x \in (\F^n)^*,\ y \in \F^n,\ A \in \lie{g}$, one writes 
\[ (x, (1 - T A)^{-1} y) \det (1 - T A)^{-1} = r_0(x,y)(A) + r_1(x,y)(A)
T + r_2(x,y)(A) T^2 + \cdots \]

\noindent where $r_i(x,y)$ is a polynomial function on $\lie{g}$, for all
$i$. Now for each polynomial $\beta = \beta_0 + \beta_1 T + \beta_2 T^2 +
\dots \in \F[T]$, the authors define in \cite{EGG} the algebra $\infgl$
as a quotient of $T(\F^n \oplus (\F^n)^*) \rtimes U \lie{g}$ by the
relations
\[ [x,x'] = 0,\qquad [y,y'] = 0,\qquad [y,x] = \beta_0 r_0(x,y) + \beta_1
r_1(x,y) + \cdots \]

\noindent for all $x,x' \in (\F^n)^*$ and $y,y' \in \F^n$. It is proved
in \cite{EGG} that these algebras are infinitesimal Hecke algebras (so
the ``PBW property" holds). Also note that if $\beta \equiv 0$, then
$\mathcal{H}_0(\lie{gl}_n, \F^n \oplus (\F^n)^*) = U(\lie{gl}_n \ltimes
(\F^n \oplus (\F^n)^*))$.

The algebras $\infgl$ provide us with the first examples of RTAs for
which one needs to use a non-strict structure to analyze them.

\begin{prop}\label{Pinfgl}
If $\ch \F = 0$, then $A = \infgl$ is a based Hopf RTA with $B^+ =
\lie{n}^+ \oplus (\F^n)^*$. Moreover, $\infgl$ is not strict for any $n
\geq 2$ and polynomial $\beta$.
\end{prop}

\begin{proof}
We first make the necessary identifications: set $\lie{h}_1$ to be the
Cartan subalgebra of $\lie{gl}_n$, $B^+$ as above, and $B^- := \lie{n}^-
\oplus (\F^n)$. Then this algebra satisfies (RTA1) by \cite{EGG}, where
$B^\pm \cong U(\lie{n}^+ \ltimes V), U(\lie{n}^- \ltimes V^*)$
respectively. Moreover, the verification of (HRTA2) is the same as what
is done in proving Proposition \ref{Pduflo}. In particular, $H_0 =
U(\lie{h}_0)$ can be chosen with $\lie{h}_0 = \F \cdot \delta$
one-dimensional. We now claim that for $n>1$, the Hopf RTA structure is
necessarily not strict. This is because if all of $\F^n$ is ``positive"
(i.e., with $H_1$-roots in $\calq_1^+$), then so is the sum of the
$\lie{h}_1$-weights in it. But this sum is over an integrable
$\lie{sl}_n$-module, hence $W$-invariant, hence has zero projection when
restricted to the Cartan subalgebra of $\lie{sl}_n$, while the eigenvalue
with respect to the central element ${\rm diag}(1, \dots, 1)$ is constant
on all of $\F^n \oplus (\F^n)^*$. This contradicts the RTA axioms.

To conclude the proof, we now present a map from \cite{KT}, which we
\textbf{claim} is an anti-involution satisfying (RTA3) for general
$n,\beta$: $j$ takes $X \in \lie{gl}_n$ to $X^T$, and $v_i
\leftrightarrow -v_i^*\ \forall i$.
To show the claim, first observe that $j$ is an anti-involution on
$\lie{gl}_n$. %since $j(X) = X^T$ is anti-multiplicative.
Next, $[e_{ij}, v_k] = \delta_{jk} v_i$ and $[e_{ji}, v_k^*] =
-\delta_{jk} v_i^*$ are clearly interchanged by $j$, so these relations
are also preserved.
Third, $[v_1, v_2] \equiv [v_1^*, v_2^*] \equiv 0$ are also $j$-stable
relations (for $v_i \in \F^n,\ v_i^* \in (\F^n)^*$).

It remains to consider the relations: $[v_l, v_k^*] = \sum_{i \geq 0}
\beta_i r_i(v_k^*, v_l)$. Note that each $r_i(v^*,v)$ is in $U \lie{g}$ -
and at the same time, identified with a function $r_i(v^*,v)(-) : \lie{g}
\to \F$, via the symmetrization map.
Now first analyze the left side: $j([v_l, v_k^*]) = [v_k, v_l^*] =
\sum_{i \geq 0} \beta_i r_i(v_l^*, v_k)$. Recall how the $r_k$ were
defined. Treating $v \in \lie{h}$ and $v^* \in \lie{h}^*$ as column and
row vectors respectively, the inner product $(v^*, Av)$ is merely matrix
multiplication $v^* A v$. Thus, we compute (inside our algebra):
\begin{align*}
\sum_{i \geq 0} r_i(v_l^*, v_k)(A) T^i
%= (v_l^*, (1 - TA)^{-1} v_k) \det(1 - TA)^{-1}
= &\ v_l^T (1 - TA)^{-1} v_k \cdot \det(1 - TA)^{-1}
= v_k^T (1 - T A^T)^{-1} v_l \cdot \det(1 - TA^T)^{-1}\\
= &\ (v_k^*, (1 - TA^T)^{-1} v_l) \det(1 - TA^T)^{-1}
= \sum_{i \geq 0} r_i(v_k^*, v_l)(A^T) T^i.
\end{align*}

\noindent Finally, use Proposition \ref{Psymm} below to show that
$j(r_i(v_k^*, v_l)(A)) = r_i(v_k^*, v_l)(A^T)$ for all $i,k,l$. Then
using the above computation of power series equality, 
\[ j \left( \sum_{i \geq 0} \beta_i r_i(v_k^*, v_l)(A) \right) = \sum_{i
\geq 0} \beta_i r_i(v_k^*, v_l)(A^T) = \sum_{i \geq 0} \beta_i r_i(v_l^*,
v_k)(A) = [v_k, v_l^*] = j([v_l, v_k^*]), \]

\noindent which shows that $j$ does indeed preserve these last relations.
\end{proof}

\begin{remark}
The based HRTA structure in Proposition \ref{Pinfgl} is not unique. For
instance, one checks that taking $\delta$ to be the matrix ${\rm
diag}(2n-1, 2n-5, \dots, 3-2n)$ works for $\infgl$ for all $n$ and all
linear $\beta = \beta_0 + \beta_1 T$.
\end{remark}

Higher rank continuous and infinitesimal Hecke algebras continue to be
the focus of much recent and ongoing research -- see
e.g.~\cite{DT,tika4,tika1,tsy} for more results and references. In
particular, Category $\calo$ has been defined and studied over $A =
\infgl$ for all $\beta$. Using Proposition \ref{Pinfgl} and the theory
developed in Section \ref{Sverma}, we now claim:

\begin{theorem}
Suppose $\F$ is algebraically closed of characteristic zero. For all
$n,\beta$, the category $\calo = \calo[\hofree]$ over $\infgl$ splits
into a direct sum of highest weight categories.
\end{theorem}

\noindent This is because in \cite{tika4}, Tikaradze computed the center
of this algebra, and showed that it satisfies Condition (S4).
%}}}

%{{{1 Section 10.3 - The symplectic case
\subsection{The symplectic case}

These algebras are generated by $\lie{g} = \lie{sp}_{2n}(\F)$ and its
natural representation, $V = \F^{2n}$. The bases for these that we use
are $e_i, e_{i+n}$ for $\F^{2n}$ with $1 \leq i \leq n$, and
\[ u_{jk} := e_{jk} - e_{k+n,j+n}, \quad v_{jk} := e_{j,k+n} + e_{k,j+n},
\quad w_{jk} := e_{j+n,k} + e_{k+n,j}, \qquad 1 \leq j, k \leq n. \]

\noindent As discussed in \cite{KT}, given a scalar parameter $\beta_0$,
the algebras $\mathcal{H}_{\beta_0}(\lie{sp}_{2n}, \F^{2n})$ are
generated by $\lie{sp}_{2n} \oplus V$, modulo the usual Lie algebra
relations for $\lie{g} = \lie{sp}_{2n}$, the ``semidirect product"
relations $[X,v] = X(v)$ for all $X \in \lie{g}, v \in V$, and the
relations $[e_i, e_j] = \beta_0 \delta_{|i-j|, n} (i-j)/n$.

\begin{prop}\label{Psymp}
The algebras $\mathcal{H}_{\beta_0}(\lie{sp}_{2n}, \F^{2n})$ are based
Hopf RTAs (assuming $\ch \F = 0$).
\end{prop}

There are other based Hopf RTAs of ``symplectic" type -- e.g., (Lie) rank
one infinitesimal Hecke algebras $\mathcal{H}_\beta(\lie{sl}_2, \F^2)$
for any $\beta$, which were discussed in Example \ref{Einfhecke1} above.
Moreover, for all $n$ and ``all possible" $\beta$, we show below that
$H_\beta(\lie{sp}_{2n}, \F^{2n})$ always has an anti-involution as in
(RTA3).

\begin{proof}
Define $h_0 := {\rm diag}(n, n-1, \dots, 1, -n, -(n-1), \dots, -1)$, and
consider the standard triangular decomposition $\lie{g} = \lie{n}^-
\oplus \lie{h}_1 \oplus \lie{n}^+$. Then $\lie{g} \oplus V$ has a basis
of eigenvectors for $\lie{h}_1$, and in particular, for $h_0$ (with
eigenvalues in $\Z$). Write $\lie{g} \oplus V = \lie{n}^{'-} \oplus
\lie{h}' \oplus \lie{n}^{'+}$, a decomposition into spans of eigenvectors
with negative, zero, and positive eigenvalues respectively.
Then $\lie{h}'$ is indeed the Cartan subalgebra $\lie{h}_1 \subset
\lie{g}$, and $\lie{n}^{'\pm} = \lie{n}^\pm \oplus V^\pm$ are Lie
subalgebras in $H_\beta$, where $V^\pm$ are the spans of $\{ e_1, \dots,
e_n \}$ and $\{ e_{n+1}, \dots, e_{2n} \}$ respectively.

Next, define $\lie{h}_0 := \F h_0,\ H_r := \sym \lie{h}_r$ for $r=0,1$,
and $B^\pm := U (\lie{n}^{'\pm})$.
Then $\mathcal{H}_{\beta_0}(\lie{sp}_{2n}, \F^{2n})$ has the required
triangular decomposition by \cite{EGG}, and $H_1$ is a commutative Hopf
algebra with sub-Hopf algebra $H_0$.
Moreover, $\widehat{H_1} = \hofree = \lie{h}_1^*$ surjects onto
$\widehat{H_0} = \lie{h}_0^* \cong \F$. Define $\calq'{}^+_0 := \nn
\Delta' := \nn \{ \alpha \}$, where $\alpha(h_0) = 1$. Then $\Z \Delta'$
is generated by the $\ad h_0$-weights of $\lie{g} \oplus V$. The
remaining part of (HRTA2) is shown as for RTLAs. Finally, that there
exists an anti-involution was shown in \cite{KT}:
\begin{equation}\label{Einv}
j : u_{kl} \leftrightarrow u_{lk},\ v_{kl} \leftrightarrow -w_{lk},\ e_i
\leftrightarrow e_{i+n}.
\end{equation}
\end{proof}
%}}}

%{{{1 Section 10.4 - The symmetrization map and anti-involutions
\subsection{The symmetrization map and anti-involutions}

We end this section by studying anti-involutions in infinitesimal Hecke
algebras. The first result is that all algebras $\infsp$ possess an
anti-involution as in (RTA3), which generalizes a part of Proposition
\ref{Psymp}. To see why, we first define these algebras for general
$n,\beta$ as in \cite{EGG}.
Denote by $\omega$ the symplectic form on $V = \F^{2n}$; one then
identifies $\lie{g} = \lie{sp}_{2n}(\F)$ with $\lie{g}^*$ via the pairing
$\lie{g} \times \lie{g} \to \F,\ (A,B) \mapsto \tr(AB)$, and $\sym
\lie{g}$ with $U \lie{g}$ via the symmetrization map. Write
\[ \omega(x, (1 - T^2 A^2)^{-1} y) \det(1 - T A)^{-1} = l_0(x,y)(A) +
l_2(x,y)(A) T^2 + \cdots \]

\noindent where $x,y \in V, A \in \lie{g}$, and $l_i(x,y) \in \sym
\lie{g} \cong U \lie{g}$ is a polynomial in $\lie{g}$ for all $i$.
For each polynomial $\beta = \beta_0 + \beta_2 T^2 + \dots \in \F[T]$,
the algebra $\infsp$ is the quotient of $TV \rtimes U \lie{g}$ by the
relations
\[ [x,y]=\beta _0 l_0(x,y) + \beta _2 l_2(x,y) + \cdots \]

\noindent for all $x,y \in V$. We now show:

\begin{prop}\label{Psp}
For all $n,\beta$, the map $j : \infsp \to \infsp$ defined in Equation
\eqref{Einv} is an anti-involution that fixes $H_1 = \sym \lie{h}_1$ (the
Cartan subalgebra of $U(\lie{g})$). Moreover, the conditions of
Proposition \ref{Pduflo} are satisfied.
\end{prop}

\begin{proof}
The first step is to show the following facts via straightforward
computations:
\begin{enumerate}
\item The map $j$ on $\lie{sp}_{2n}$ can be extended to all of
$\lie{gl}_{2n}$, via: $j(C) = \tau C^T \tau$ -- where $\displaystyle \tau
= \tau^{-1} = \begin{pmatrix} {\rm Id}_n & 0\\0 & -{\rm Id}_n
\end{pmatrix} \in GL(2n)$.

\item One has $\omega(x,Cy) = \omega(j(y), j(C) j(x))$, for all $x,y \in
\F^{2n},\ C \in \lie{gl}_{2n}$ (using $j$ as in the previous part).

\item $j \left( (1 - T^2 A^2)^{-1} \right) = (1 - T^2 j(A)^2)^{-1}$.
\end{enumerate}

Now note that the conditions of Proposition \ref{Pduflo} hold here, if
one defines $\delta := h_0$, the special element from the proof of
Proposition \ref{Psymp}. As for the proposed anti-involution, it is not
hard to check that $j$ is an anti-involution on $\lie{sp}_{2n}$, which
preserves the relations $[X,v] = X(v)$ for $X \in \lie{sp}_{2n}$ and $v
\in \F^{2n}$. We are left to consider the relations $[x,y]$. Now compute
using the above facts:
\begin{align*}
&\ \sum_{i \geq 0} l_{2i}(x,y)(A) T^{2i} = \omega(x, (1 - T^2 A^2)^{-1}y)
\det (1 - TA)^{-1}\\
= &\ \omega(j(y), (1 - T^2 j(A)^2)^{-1} j(x)) \det (1 - T j(A))^{-1}
= \sum_{i \geq 0} l_{2i}(j(y),j(x))(j(A)) T^{2i},
\end{align*}

\noindent where the second equality is not hard to show. In particular,
replacing $A$ by $j(A)$ and equating coefficients of $T$, it follows that
\begin{equation}\label{Esp}
l_{2i}(x,y)(j(A)) = l_{2i}(j(y),j(x))(A) \quad \forall x,y \in \F^{2n},\
i \geq 0.
\end{equation}

\noindent Now compute:
\[ j([x,y]) = j \left( \sum \beta_i l_{2i}(x,y)(A) \right)
= \sum \beta_i l_{2i}(x,y)(j(A)) = \sum_i \beta_i l_{2i}(j(y),j(x))(A) =
[j(y),j(x)], \]

\noindent where the first and last equalities are by definition, the
second uses Proposition \ref{Psymm} below (via the trace form), and the
third follows from Equation \eqref{Esp}.
\end{proof}

We finally mention a result that was used in proving that every
infinitesimal Hecke algebra over $\lie{gl}_n$ has an anti-involution
that is required to make it a (based) Hopf RTA.

\begin{prop}\label{Psymm}
Suppose $\lie{g}$ is any Lie algebra, and we identify $\sym \lie{g}$ with
$U \lie{g}$ via the {\em symmetrization map}
\[ \symm : X_1 \dots X_n \mapsto \frac{1}{n!} \sum_{\sigma \in S_n}
X_{\sigma(1)} \dots X_{\sigma(n)}. \]

\noindent Suppose $j$ is a Lie algebra anti-involution of $\lie{g}$. Then
the automorphism $j$ of $\sym \lie{g}$ is transferred to $U \lie{g}$ via
$\symm$.
\end{prop}

\begin{proof}
Observe that the symmetrization map commutes with $j$, and both composite
maps are $\F$-vector space isomorphisms.
\end{proof}

Applying this result to infinitesimal Hecke algebras over $(\lie{gl}_n,
\F^n \oplus (\F^n)^*)$ in the proof of Proposition \ref{Pinfgl}, we get
(via a further identification of $\lie{g} \leftrightarrow \lie{g}^*$ by
the trace form):
\[ r_i(v, v^*)(A^T) = r_i(v, v^*)(j(A)) = j(r_i(v, v^*)(A)), \]

\noindent as desired. A similar application yields the anti-involution
mentioned above for infinitesimal Hecke algebras over $(\lie{sp}_{2n},
\F^{2n})$.
%}}}

\subsection*{Concluding example}

Recall that the construction in Section \ref{Sexample} provided a setting
that could not be studied using previous theories of $\calo$, because the
``root lattice" $\calq^+_0$ is not abelian.
Using the above results on $\calo$ for general RTAs, as well as
the examples studied above, we now present a second example of a regular
triangular algebra, whose study requires the full generality of our
axiomatic framework and not a more specialized setting. The following
example, combined with the Existence Theorems in Section
\ref{Snonabelian}, reinforces the viewpoint that our theory is not merely
abstract, but is required in its totality in applications to specific
regular triangular algebras.

\begin{exam}[A non-strict, non-Hopf, RTA]\label{Efull}
Suppose $q \in \C^\times$ is not a root of unity, $\beta \in \C[T]$, and
$z$ is a nonzero polynomial in the quantum Casimir in $U_q(\lie{sl}_2)$.
Also suppose $X$ is a topological space with the algebra of continuous
functions $C(X,\R)$ not a Hopf algebra, and $T$ is a homeomorphism of $X$
of infinite order. Now define
\begin{equation}\label{Efullrta}
A := U'_q(\lie{sl}_2) \otimes \mathcal{H}_{z,q} \otimes \left( \C
\otimes_\R \cals(C(X,\R), T^*, 1, 1) \right) \otimes
\mathcal{H}_\beta(\lie{gl}_n, \C^n \oplus (\C^n)^*),
\end{equation}

\noindent where the individual tensor factors were studied in Examples
\ref{EJZ}, \ref{Einfhecke1}, \ref{Emirza}, and Section \ref{Sgln}
respectively. We now claim that $A$ is an RTA satisfying BGG Reciprocity,
and that the study of Category $\calo$ over $A$ requires the full scope
of our general framework and no less.

\begin{theorem}\label{Tfullrta}
The algebra $A$ defined in \eqref{Efullrta} has the following properties:
\begin{enumerate}
\item $A$ is a based Regular Triangular Algebra but not a strict one.

\item Neither of the algebras $H_1 \supsetneq H_0$ is a Hopf algebra, so
$A$ is not an HRTA.

\item The simple roots $\Delta$ are not weights for $H_0$.

\item $\calo[\hofree] \subsetneq \calo$, because $\hofree \subsetneq
\widehat{H_1}$.

\item Condition (S4) is not satisfied because the center is not ``large
enough". Thus, central characters cannot be used to obtain a block
decomposition of $\calo[\hofree]$ into blocks with finitely many simple
objects.
\end{enumerate}

\noindent Nevertheless, the algebra $A$ satisfies Condition (S3). Hence
Theorem \ref{Tfirst} holds and $\calo[\hofree]$ decomposes into a direct
sum of finite length, self-dual blocks. Each block has finitely many
simples and enough projectives/injectives, and is a highest weight
category satisfying BGG Reciprocity.
\end{theorem}

\begin{proof}
That the algebra $A$ is a based RTA follows from the analysis in the
aforementioned examples by using Theorem \ref{Tfunct}, since each
individual tensor factor is a based RTA. Next, the RTA is not strict
because $\mathcal{H}_\beta(\lie{gl}_n, \C^n \oplus (\C^n)^*)$ is
necessarily not a strict RTA by Proposition \ref{Pinfgl}. The algebras
$H_0, H_1$ are not Hopf algebras because $\C \otimes_\R C(X, \R)$ is not
a Hopf algebra by assumption. Properties (3),(4) hold in $A$ because they
hold respectively in $\C \otimes_\R \cals(C(X,\R), T^*, 1, 1)$ and in
$U'_q(\lie{sl}_2)$.
Finally, Condition (S4) is not satisfied in $\mathcal{H}_{z,q}$ by
\cite{GGK}, since $Z(\mathcal{H}_{z,q}) = \C$. Hence $A$ also does not
satisfy (S4), by Theorem \ref{Tfunct}.

Next, to show that $A$ satisfies Condition (S3) it suffices to verify the
same property for each of the tensor factors. As stated in Section
\ref{Sgwa-example1}, if $q$ is not a root of unity then
$U'_q(\lie{sl}_2)$ satisfies Condition (S3) over $\C$. That
$\mathcal{H}_{z,q}$ satisfies Condition (S3) was shown in \cite{GGK}
(also see Example \ref{Einfhecke1}); and that
$\mathcal{H}_\beta(\lie{gl}_n, \C^n \oplus (\C^n)^*)$ satisfies Condition
(S4) (and hence (S3)) was shown in \cite{tika4}. Finally, $\C \otimes_\R
\cals(C(X,\R), T^*, 1, 1)$ satisfies Condition (S3) by Theorem
\ref{Tgwa}(2), since $z_0 = z_1 = 1$.
\end{proof}
\end{exam}

\subsection*{Acknowledgements}

I thank the following people for useful conversations concerning various
examples considered in this paper:
Hiroyuki Yamane for Example \ref{Eyamane};
Tatyana Chmutova for Example \ref{Echmutova};
Jonas T.~Hartwig and Maryam Mirzakhani for Example \ref{Emirza}; and
Jonas T.~Hartwig and Vyacheslav Futorny for Example \ref{ETGWA}.
I also thank Shahn Majid for pointing me to the reference \cite{GM}
concerning matched pairs of monoids, and Daniel Bump and Jacob Greenstein
for stimulating conversations.
Finally, I am indebted to Akaki Tikaradze for many interesting and useful
discussions, including about the proof of Theorem \ref{Tskolem} (together
with Justin Sinz).
A part of this work was carried out at a conference organized by Punita
Batra at the Harish-Chandra Research Institute in December 2014, and I
thank her for providing a productive and stimulating atmosphere.

%{{{1 Bibliography / references

%\nocite{*}
%\bibliographystyle{amsplain}
%\bibliography{hfmod}

%%\providecommand{\bysame}{\leavevmode\hbox to3em{\hrulefill}\thinspace}
%%\providecommand{\MR}{\relax\ifhmode\unskip\space\fi MR }
%%\providecommand{\MRhref}[2]{%
%%  \href{http://www.ams.org/mathscinet-getitem?mr=#1}{#2}
%%}
%%\providecommand{\href}[2]{#2}

%}}}

\end{document}